\pdfoutput=1
\RequirePackage{ifpdf}
\ifpdf 
\documentclass[pdftex]{sigma}
\else
\documentclass{sigma}
\fi

\usepackage{stmaryrd,mathrsfs,bbm,accents}
\usepackage[all]{xy}
\usepackage{url}

\numberwithin{equation}{section}

\newtheorem{Theorem}{Theorem}[section]
\newtheorem{Proposition}[Theorem]{Proposition}
\newtheorem{Corollary}[Theorem]{Corollary}
\newtheorem{Lemma}[Theorem]{Lemma}

{ \theoremstyle{definition}
\newtheorem{Definition}[Theorem]{Definition}
\newtheorem{Remark}[Theorem]{Remark}
\newtheorem{Remarks}[Theorem]{Remarks}
}

\begin{document}
\allowdisplaybreaks

\newcommand{\arXivNumber}{2201.10931}

\renewcommand{\PaperNumber}{050}

\FirstPageHeading

\ShortArticleName{Spherical Representations of $C^*$-Flows II}

\ArticleName{Spherical Representations of $\boldsymbol{C^*}$-Flows II:\\ Representation System and Quantum Group Setup}

\Author{Yoshimichi UEDA}

\AuthorNameForHeading{Y.~Ueda}

\Address{Graduate School of Mathematics, Nagoya University,\\ Furocho, Chikusaku, Nagoya, 464-8602, Japan}
\Email{\href{mailto:ueda@math.nagoya-u.ac.jp}{ueda@math.nagoya-u.ac.jp}}
\URLaddress{\url{http://www.math.nagoya-u.ac.jp/~ueda/index.html}}

\ArticleDates{Received February 07, 2022, in final form June 26, 2022; Published online July 05, 2022}

\Abstract{This paper is a sequel to our previous study of spherical representations in the operator algebra setup. We first introduce possible analogs of dimension groups in the present context by utilizing the notion of operator systems and their relatives. We then apply our study to inductive limits of compact quantum groups, and establish an analogue of Olshanski's notion of spherical unitary representations of infinite-dimensional Gelfand pairs of the form $G < G\times G$ (via the diagonal embedding) in the quantum group setup. This, in particular, justifies Ryosuke Sato's approach to asymptotic representation theory for quantum groups.}


\Classification{22D25; 22E66; 46L67; 17B37}
\vspace{2mm}

\section{Introduction}
\vspace{2mm}

In \cite{Ueda:Preprint20} we introduced the notion of $(\alpha^t,\beta)$-spherical representations for $C^*$-flows $\alpha^t$ and developed its general theory including Vershik--Kerov's ergodic method and the spectral decomposition of such a representation for a certain class of $C^*$-flows. The present paper is a sequel to that paper, and still attempts to provide some general results. In fact, we will introduce a natural algebraic structure for a certain class of $C^*$-flows, which plays a similar r\^{o}le as dimension groups in Vershik--Kerov's theory. We will also explain how the general theory given in the previous paper and the first part of this paper nicely fits the asymptotic representation theory for quantum groups initiated by Ryosuke Sato \cite{Sato:JFA19}.

Let $G = \varinjlim G_n$ be an inductive limit of compact groups, which is usually not locally compact. When $G$ is not locally compact, we cannot utilize standard methods, based on $C^*$-algebras, to investigate unitary representations of $G$. Nevertheless, Vershik--Kerov's asymptotic representation theory (see \cite{VershikKerov:FunctAnalAppl81,VershikKerov:SovMathDokl82} and Kerov's thesis \cite{Kerov:Book}) as well as Olshanski's spherical representation theory (see \cite{Olshanski:Proc90,Olshanski:LNM03,Olshanski:JFA03}) work well for both the infinite symmetric group $\mathfrak{S}_\infty = \varinjlim\mathfrak{S}_n$ (that is locally compact) and the infinite-dimensional unitary group $\mathrm{U}(\infty) = \varinjlim\mathrm{U}(n)$ (that is not locally compact) in an almost parallel fashion. This is a bit surprising phenomenon to us, because the \emph{natural} $C^*$-algebras associated with those groups are of different kinds. Actually, the group $C^*$-algebra $C^*(\mathfrak{S}_\infty)$ is an AF-algebra, while we think, through the analysis of qunatum groups, that the natural $C^*$-algebra of $\mathrm{U}(\infty)$ should be a $C^*$-inductive limit of atomic $W^*$-algebras since the associated branching graph is not locally finite.

Vershik--Kerov's theory should be understood as a theory for characters (rather than representations), and is applicable even for tracial states of general AF-algebras. In their theory, the concept of dimension groups (or ordered $K_0$-groups) of AF-algebras plays an important r\^{o}le; see~\cite{VershikKerov:JSovietMath87,VershikKerov:Proc90}. In fact, the dimension group of the group $C^*$-algebra $C^*(\mathfrak{S}_\infty)$ is computed from the so-called Littlewood--Richardson ring of the inductive sequence $\mathfrak{S}_n$ (and hence it admits a natural ring structure). This suggests us that the part of dealing with dimension groups in Vershik--Kerov's theory is essential from the viewpoint of representation theory. However, we cannot apply the idea of dimension groups to general inductive limits of compact groups directly, because those inductive limits may not be dealt with within the class of AF-algebras since their branching graphs are no longer locally finite as pointed out above. Although Boyer~\cite{Boyer:JFA87} attempted to generalize Vershik--Kerov's use of dimension groups to $\mathrm{U}(\infty)$, etc.\ based on Geissinger's idea~\cite{Geissinger:LNM77}, we focus on Olshanski's recent work \cite{Olshanski:AdvMath16} that introduced an analogous ring for $\mathrm{U}(\infty)$ based on the theory of symmetric functions. In fact, we are interested in understanding the construction of Olshanski's ring in a general setup like Vershik--Kerov's theory \cite{VershikKerov:JSovietMath87} dealing with traces on general AF-algebras.

As mentioned above, Sato \cite{Sato:JFA19} initiated the study of asymptotic representation theory for quantum groups motivated by Gorin's work \cite{Gorin:AdvMath12}. The main idea of that work \cite{Sato:JFA19} is to replace tracial states with KMS states with respect to deformation (or scaling) flows that arise as the effect of $q$-deformation of classical groups. Thus we will construct a natural operator system (and its relative) for a certain class of flows on inductive limits of atomic $W^*$-algebras (whose special cases are AF-algebras), which plays a r\^{o}le of dimension groups in Vershik--Kerov's works and generalizes Olshanski's ring for $\mathrm{U}(\infty)$ naturally. Moreover, we will answer, from the viewpoint of spherical unitary representations, the questions of why KMS states and which inverse temperature are appropriate in the quantum group setting. Note that there are no special representation-theoretic reasons for those choices in Sato's approach because of the character of Vershik--Kerov's theory that his approach is modeled~\mbox{after}.

This paper consists of two parts. The contents are as follows.

In the first part, for a given continuous, inductive flow on a locally atomic $W^*$-algebra, we will introduce an operator system (and its relative) playing a r\^{o}le of dimension groups in Vershik--Kerov's theory, and explain that the resulting operator system has a natural pairing with all the locally normal KMS states with respect to the flow. We also introduce a~certain module structure on the operator system under a certain assumption. We will also show that (the relative of) that operator system coincides with Olshanski's ring in the case of~$\mathrm{U}(\infty)$.\looseness=1

The second (and main) part concerns basic theory of spherical unitary representations for inductive limits $G = \varinjlim G_n$ of compact quantum groups.

\looseness=1
We will explain how our general theory works for those $G$. Namely, we will naturally define the pair $G < G\times G$ and its spherical unitary representations. Then we will show that any spherical unitary representations fall into the framework of $(\alpha^t,\beta)$-spherical representations. As~a~consequence, we will be able to justify the use of KMS states in the asymptotic representation theory for quantum groups and to see which inverse temperature should be selected.

We will also upgrade Sato's notion of quantized characters \cite{Sato:JFA19} to ``quantized character functions over $G$''. This notion is a natural generalization of functions $g \in G \mapsto \chi(g) \in \mathbb{C}$ with characters $\chi$ on an ordinary group $G$.

In closing of the second part, we will explain how to apply the theory we have developed so far to the example $\mathrm{U}_q(\infty)$ from quantum unitary groups $\mathrm{U}_q(n)$. As an application of the notion of quantized character functions, we will give a clearer interpretation of Gorin's $q$-Schur generating functions (see \cite{Gorin:AdvMath12}) in terms of $\mathrm{U}_q(\infty)$. We try to make it clear that all the computations can essentially be done by means of quantized universal enveloping algebras $U_q\mathfrak{gl}(n)$ (although the formulation of spherical unitary representations certainly needs the operator algebraic notion of compact quantum groups due to its character of unitary representations).

\subsection*{Notations} We will use several kinds of tensor products throughout; $\otimes$, $\bar{\otimes}$, $\otimes_{\min}$ and $\otimes_{\max}$ denote algebraic, $W^*$-algebraic, minimal $C^*$-algebraic and maximal $C^*$-algebraic tensor products, respectively. The reader can find their basic facts in Brown--Ozawa's book \cite{BrownOzawa:Book}.

Similarly to our previous paper \cite{Ueda:Preprint20}, we encourage the reader who is not familiar with operator algebras to consult Bratteli--Robinson's two-volume book \cite{BratteliRobinson:Book1,BratteliRobinson:Book2} for basic facts on operator algebras. We will also use the notion of operator systems and their relatives. The necessary basic facts on them are summarized, e.g., in~\cite[Section 2]{MawhinneyTodorov:DissertationesMath18}, and a standard reference book on them is Paulsen's famous one \cite{Paulsen:Book}.

We will simply call unital $*$-representations of unital $*$-algebras just $*$-representations throu\-ghout.

\section{General setup and notations}\label{S2}

Let $A_n$, $n \geq 0$, be an inductive sequence of atomic $W^*$-algebras with separable preduals such that $A_0 = \mathbb{C}$ and that each inclusion $A_n \hookrightarrow A_{n+1}$ is unital and normal, that is, $A_n$ is a unital $W^*$-subalgebra of $A_{n+1}$. Denote by $\mathfrak{Z}_n$ all the minimal projections of the center $\mathcal{Z}(A_n)$. By the general structure of atomic $W^*$-algebras, each $zA_n$ with $z \in \mathfrak{Z}_n$ is identified with all the bounded operators $B(\mathcal{H}_z)$ on a unique separable Hilbert space $\mathcal{H}_z$. We will write $\dim(z) := \dim(\mathcal{H}_z)$ for short.

With the inductive sequence $A_n$ is associated the following graded graph, called the Bratteli diagram: The vertex set is $\mathfrak{Z} := \bigsqcup_{n\geq0}\mathfrak{Z}_n$, a disjoint union, and the edge set is $\mathfrak{E} := \bigsqcup_{n\geq1} \mathfrak{E}_n$, a~disjoint union, with $\mathfrak{E}_n := \{(z_2,z_1) \in \mathfrak{Z}_{n}\times\mathfrak{Z}_{n-1};\, z_2 z_1\neq0\}$. Each pair $(z_2,z_1) \in \mathfrak{E}_n$ should be read as a directed edge from $z_1$ to $z_2$ so that its source $s(z_2,z_1)$ and range $r(z_2,z_1)$ are $z_1$ and~$z_2$, respectively. We also have a function $m\colon \mathfrak{E} \to \mathbb{N}\cup\{\infty\}$, called the multiplicity function, defined by $m(z_2,z_1) := \dim(z_2 z_1(A_n \cap (A_{n-1})'))^{1/2}$ with $(z_2,z_1) \in \mathfrak{E}_n$. As explained in \cite[Section~9.1]{Ueda:Preprint20} one can reconstruct the inductive sequence $A_n$ from the data $(\mathfrak{Z},\mathfrak{E},m)$. However, the multiplicity function is not so important in the analysis below. Another important function on the edges $\mathfrak{E}$ will be recalled in Section \ref{S3.2}.

Let $A = \varinjlim A_n$ be an inductive (or direct) limit $C^*$-algebra (see, e.g., \cite[Chapter 2]{Effros:CBMSBook}). Assume that we have a flow $\alpha^t\colon \mathbb{R} \curvearrowright A$ such that $\alpha^t(A_n) = A_n$ holds for every $t \in \mathbb{R}$ and $n \geq 0$, and moreover, the restriction of $\alpha^t$ to each $A_n$, denoted by $\alpha_n^t\colon \mathbb{R} \curvearrowright A_n$, is continuous in the $u$-topology (see, e.g., \cite[Lemma 7.1]{Ueda:Preprint20} for this topology). This assumption forces every flow $\alpha_n^t$ to fix elements in $\mathcal{Z}(A_n)$. See \cite[Lemma 7.1]{Ueda:Preprint20} for this fact. In particular, the restriction of $\alpha^t$ to~$zA_n$ with $z \in \mathfrak{Z}_n$ gives a continuous flow $\alpha_z^t$ on $zA_n = B(\mathcal{H}_z)$ in the $u$-topology.

\emph{Throughout this paper, we will consider only an inverse temperature $\beta \in \mathbb{R}$, for which there is a $($unique$)$ normal $(\alpha_z^t,\beta)$-KMS state $\tau^{(\alpha_z^t,\beta)}$ on each $zA_n$ with $z \in \mathfrak{Z}_n$.} This is the case for all $\beta \in \mathbb{R}$, when every $zA_n = B(\mathcal{H}_z)$ is finite dimensional. (See, e.g., \cite[Section~7]{Ueda:Preprint20}.) For every $a \in A_n$ we write
\begin{gather*}
E^{(\alpha_n^t,\beta)}(a) := \sum_{z\in\mathfrak{Z}_n} \tau^{(\alpha_z^t,\beta)}(za)z,
\end{gather*}
which defines a faithful normal conditional expectation $E^{(\alpha_n^t,\beta)}\colon A_n \to \mathcal{Z}(A_n)$. (See \cite[equation~(7.2)]{Ueda:Preprint20}, where we denoted it by $E_n^{(\alpha^t,\beta)}$ instead.) This normal conditional expectation can be characterized by the (operator-valued) KMS condition
\begin{gather*}
E^{(\alpha_n^t,\beta)}(ab) = E^{(\alpha_n^t,\beta)}\big(b\alpha^{{\rm i}\beta}(a)\big)
\end{gather*}
for all $a \in (A_n)_{\alpha_n^t}^\infty$ and $b\in A_n$, where $(A_n)_{\alpha_n^t}^\infty$ denotes the $\sigma$-weakly dense $*$-subalgebra of $\alpha_n^t$-analytic elements of $A_n$. When $\alpha_n^t$ is the trivial flow or $\beta=0$, the conditional expectation~$E^{(\alpha_n^t,\beta)}$ must be the center-valued trace. We also remark that
\begin{gather}\label{Eq2.2}
E^{(\alpha_n^t,\beta)}\circ E^{(\alpha_m^t,\beta)}(a) = E^{(\alpha_n^t,\beta)}(a), \qquad
a \in A_m
\end{gather}
holds if $n > m \geq 0$. See \cite[Lemma 7.2(3)]{Ueda:Preprint20}. For the ease of notation, we often write $\tau_z := \tau^{(\alpha_z^t,\beta)}$ and $E_n := E^{(\alpha_n^t,\beta)}$ if the flow $\alpha^t$ and the inverse temperature $\beta$ are clear from context.

We will denote by $K_\beta^\mathrm{ln}(\alpha^t)$ the convex set of all locally normal $(\alpha^t,\beta)$-KMS states on $A$, which is a face of all the $(\alpha^t,\beta)$-KMS states $K_\beta(\alpha^t)$. Here, we recall that an $(\alpha^t,\beta)$-KMS state~$\omega$ on $A$ is locally normal if the restriction of $\omega$ to each $A_n$ is normal. Our previous work shows that any locally normal $(\alpha^t,\beta)$-spherical representation is the standard form of $A$ associated with a locally normal $(\alpha^t,\beta)$-KMS state on $A$, where we borrow the term ``standard form'' from theory of von Neumann (or $W^*$-)algebras. Thus, the investigation of $K_\beta^\mathrm{ln}(\alpha^t)$ is of particular importance in our study.

\section{Representation (operator) system} 

We keep the setting in Section \ref{S2} and fix an inverse temperature $\beta$ with the requirement there throughout. In this section, we are seeking for a suitable algebraic structure playing a r\^{o}le of dimension groups in Vershik--Kerov's asymptotic representation theory (see \cite{VershikKerov:JSovietMath87,VershikKerov:Proc90}).

\subsection{Preparatory consideration} \label{S3.1}
Every $\mathcal{Z}(A_n)$ is naturally identified with all the bounded complex-valued functions $\ell^\infty(\mathfrak{Z}_n)$ over~$\mathfrak{Z}_n$ by $f \in \ell^\infty(\mathfrak{Z}_n) \mapsto \sum_{z \in \mathfrak{Z}_n} f(z)z \in \mathcal{Z}(A_n)$. Accordingly, the predual $\mathcal{Z}(A_n)_*$ is identified with all the summable complex-valued functions $\ell^1(\mathfrak{Z}_n)$ over $\mathfrak{Z}_n$ in such a way that each $\omega \in \mathcal{Z}(A_n)_*$ is considered as a function $z \in \mathfrak{Z}_n \mapsto \omega(z) \in \mathbb{C}$. Thus, we have a natural pairing
\[
\langle \omega, f\rangle_n := \sum_{z\in\mathfrak{Z}_n} \omega(z)f(z), \qquad \omega \in \mathcal{Z}(A_n)_* = \ell^1(\mathfrak{Z}_n), \quad f \in \mathcal{Z}(A_n)=\ell^\infty(\mathfrak{Z}_n).
\]
In what follows, we will freely identify $\mathcal{Z}(A_n) = \ell^\infty(\mathfrak{Z}_n)$ and $\mathcal{Z}(A_n)_* = \ell^1(\mathfrak{Z}_n)$.

Let $\omega \in K_\beta^\mathrm{ln}(\alpha^t)$ be arbitrarily given. The restriction of $\omega$ to each $zA_n = B(\mathcal{H}_z)$ with $z \in \mathfrak{Z}_n$ is a normal $(\alpha^t_z,\beta)$-KMS state. By the uniqueness of KMS states, we have $\omega = \omega(z)\tau_z$ on $zA_n$. Hence $\omega_n := \omega\!\upharpoonright_{\mathcal{Z}(A_n)}$, the restriction to $\mathcal{Z}(A_n)$, is in $\mathcal{Z}(A_n)_*$ and
\begin{equation}\label{Eq3.1}
\omega(a) = \sum_{z \in \mathfrak{Z}_n} \omega(za) = \sum_{z\in\mathfrak{Z}_n} \omega(z)\tau_z(za) = \langle\omega_n, E_n(a)\rangle_n, \qquad a \in A_n.
\end{equation}
Remark that this equation must hold for every $n$. Moreover, the right-most side of \eqref{Eq3.1} clearly gives a normal $(\alpha_n^t,\beta)$-KMS state on $A_n$. Consequently, the duality pair $(\mathcal{Z}(A_n)_*,\mathcal{Z}(A_n))$ contains essential information of $K_\beta^\mathrm{ln}(\alpha^t)$.

It is known, see, e.g., \cite[Proposition 6.1]{Reich:K-theory01}, that the $K_0$-group $K_0(A_n)$ is isomorphic, via the center-valued trace, to the pointwise additive group consisting of $f \in \ell^\infty(\mathfrak{Z}_n,\mathbb{R})$, the real-valued bounded functions on $\mathfrak{Z}_n$, such that $f(z) \in \{ k/\dim(z);\, k \in \mathbb{Z}\}$ for every $z \in \mathfrak{Z}_n$. Moreover, the equivalence class $[1]$ of the unit $1 \in A_n$ in $K_0(A_n)$ corresponds, via the isomorphism, to the constant function $f\equiv1$. Consequently, we may and do assume that $(K_0(A_n),[1])$ is an additive subgroup of the hermitian elements $\mathcal{Z}(A_n)_h := \{ a \in \mathcal{Z}(A_n);\, a =a^*\} = \ell^\infty(\mathfrak{Z}_n,\mathbb{R})$ and the unit $1 \in \mathcal{Z}(A_n)$. Note, in particular, that $K_0(A_n)\otimes_{\mathbb{Z}}\mathbb{R} \subset K_0(A_n)\otimes_{\mathbb{Z}}\mathbb{C}$ are exactly $\mathcal{Z}(A_n)_h \subset \mathcal{Z}(A_n)$, respectively, when $A_n$ is finite dimensional.

Keeping these observations in mind, we will consider $(\mathcal{Z}(A_n),1)$ as an operator system (an~ordered $*$-vector space with an additional order structure) in place of $(K_0(A_n),[1])$.

\subsection{Constructions} \label{S3.2}

Motivated by the consideration in the previous section we will introduce an operator system $\mathcal{S}(\alpha^t,\beta)$ as an analog of dimension groups associated with $(\alpha^t,\beta)$ and a kind of its predual.

\subsubsection{Inductive and projective sequences} \label{S3.2.1}
The diagram
\[
\xymatrix{
A_n \ar@{^{(}-^{>}}[r] & A_{n+1} \ar[d]^{E_{n+1}}\\
\mathcal{Z}(A_n) \ar@{^{(}-^{>}}[u] \ar@{.>}[r]
& \mathcal{Z}(A_{n+1})
}
\]
defines a normal UCP map $\Theta_{n+1,n} \colon \mathcal{Z}(A_n) \to \mathcal{Z}(A_{n+1})$ as the restriction of $E_{n+1}$ to $\mathcal{Z}(A_n)$ ($\subset A_n \subseteq A_{n+1}$). For each $f \in \ell^\infty(\mathfrak{Z}_n) = \mathcal{Z}(A_n)$ we have
\[
\Theta_{n+1,n}(f)=\sum_{z\in\mathfrak{Z}_n} f(z)E_{n+1}(z) = \sum_{z\in\mathfrak{Z}_n} f(z) \sum_{z' \in \mathfrak{Z}_{n+1}} \tau_{z'}(z'z)z'
\]
and hence
\begin{equation}\label{Eq3.2}
\Theta_{n+1,n}(f)(z') = \sum_{z \in \mathfrak{Z}_n} \tau_{z'}(z'z)f(z), \qquad z' \in \mathfrak{Z}_{n+1}
\end{equation}
holds as an element in $\ell^\infty(\mathfrak{Z}_{n+1})$. In what follows, we simply write $\Theta_{n,m} := \Theta_{n,n-1}\circ\cdots\circ\Theta_{m+1,m}$, which coincides with the restriction of $E_n$ to $\mathcal{Z}(A_m)$ thanks to \eqref{Eq2.2}, for each pair $n > m$.

We then consider the (pre)dual map $\Theta_{n+1,n}^* \colon \mathcal{Z}(A_{n+1})_* \to \mathcal{Z}(A_n)_*$. For any $\omega \in \mathcal{Z}(A_{n+1})_*$ ($=\ell^1(\mathfrak{Z}_{n+1})$) and $f \in \mathcal{Z}(A_n)$ ($=\ell^\infty(\mathfrak{Z}_n)$), we have
\begin{align*}
\langle\Theta_{n+1,n}^*(\omega),f\rangle_n
&=
\langle\omega,\Theta_{n+1,n}(f)\rangle_{n+1} \\
&=
\sum_{z'\in\mathfrak{Z}_{n+1}} \omega(z') \Theta_{n+1,n}(f)(z') \\
&=
\sum_{z' \in \mathfrak{Z}_{n+1}} \sum_{z \in \mathfrak{Z}_n} \omega(z') \tau_{z'}(z'z)f(z),
\end{align*}
and hence
\begin{equation}\label{Eq3.3}
\Theta_{n+1,n}^*(\omega)(z) = \sum_{z' \in \mathfrak{Z}_{n+1}} \omega(z') \tau_{z'}(z'z), \qquad z \in \mathfrak{Z}_n
\end{equation}
holds as an element of $\ell^1(\mathfrak{Z}_n) = \mathcal{Z}(A_n)_*$. We will write $\Phi_{n,n+1} := \Theta_{n+1,n}^*$ in what follows.

In this way, we have the inductive sequence
\begin{equation}\label{Eq3.4}
\cdots \overset{\Theta_{n,n-1}}{\longrightarrow} \mathcal{Z}(A_n)
\overset{\Theta_{n+1,n}}{\longrightarrow} \mathcal{Z}(A_{n+1})
\overset{\Theta_{n+2,n+1}}{\longrightarrow} \cdots
\end{equation}
in the category of operator systems and the projective sequence
\begin{equation}\label{Eq3.5}
\cdots \overset{\Phi_{n-1,n}}{\longleftarrow} \mathcal{Z}(A_n)_*
\overset{\Phi_{n,n+1}}{\longleftarrow} \mathcal{Z}(A_{n+1})_*
\overset{\Phi_{n+1,n+2}}{\longleftarrow} \cdots
\end{equation}
in the category of Banach spaces, that is, every $\Phi_{n,n+1}$ is contractive since every $\Theta_{n+1,n}$ is UCP. We remark that inductive sequence \eqref{Eq3.4} is precisely Banach space dual to projective sequence~\eqref{Eq3.5}.

Here we recall a definition from \cite[Section 7]{Ueda:Preprint20}.

\begin{Definition}\label{D3.1} We define the \emph{link} $\kappa = \kappa_{(\alpha^t,\beta)} \colon \bigsqcup_{n\geq1} \mathfrak{Z}_n\times\mathfrak{Z}_{n-1} \to [0,1]$ by
\[
\kappa(z',z) :=
\begin{cases}
\tau_{z'}(z'z), & (z',z) \in \mathfrak{E},
\\
0, & \text{otherwise}.
\end{cases}
\]
\end{Definition}

Note that $[\kappa(z',z)]$ defines a stochastic matrix over $\mathfrak{Z}_{n+1}\times\mathfrak{Z}_n$ for each $n \geq 0$. See \cite[Lemma~7.9]{Ueda:Preprint20}. Identities \eqref{Eq3.2} and \eqref{Eq3.3} can be rewritten as
\begin{gather}
\label{Eq3.6}
\Theta_{n+1,n}(f)(z') = \sum_{z\in\mathfrak{Z}_n} \kappa(z',z)f(z), \\
\label{Eq3.7}
\Phi_{n,n+1}(\omega)(z) = \sum_{z'\in\mathfrak{Z}_{n+1}} \omega(z')\kappa(z',z),
\end{gather}
respectively. The predual $\mathcal{Z}(A_n)_*$ is equipped with a natural involution $\omega^*(a) := \overline{\omega(a^*)}$ for $a \in \mathcal{Z}(A_n)$, and this involution is nothing but the standard involution on $\ell^1(\mathfrak{Z}_n)$ taking pointwise complex-conjugation.

Observe that all the hermitian elements in $\mathcal{Z}(A_n)_* = \ell^1(\mathfrak{Z}_n)$, i.e., $\mathcal{Z}(A_n)_{*,h} = \ell^1(\mathfrak{Z}_n,\mathbb{R})$, becomes a (real) Banach lattice. By \eqref{Eq3.7} every morphism $\Phi_{n,n+1}$ is positive with respect to this Banach lattice structure. Namely, we have
\begin{equation*}
\cdots \overset{\Phi_{n-1,n}}{\longleftarrow} \mathcal{Z}(A_n)_{*,h} \overset{\Phi_{n,n+1}}{\longleftarrow} \mathcal{Z}(A_{n+1})_{*,h} \overset{\Phi_{n+1,n+2}}{\longleftarrow} \cdots
\end{equation*}
in the category of Banach lattices (whose morphisms should be positive and contractive), and projective sequence \eqref{Eq3.5} is obtained as the complexification of this sequence. Hence, projective sequence \eqref{Eq3.5} can be understood in the category of complex Banach lattices. See, e.g., \cite[p.~157]{Pietsch:HistoryBook} for the notion of complex Banach lattices. However, we will not take any viewpoint of Banach lattices in what follows.

\begin{Remark}\label{R3.2}
Assume for a while that $\alpha^t$ is the trivial flow or $\beta=0$ and further that all $A_n$ are finite dimensional, that is, $A$ is an AF-algebra. It is not difficult to see that the matrix $[d(z')\kappa(z',z)/d(z)]$ over $\mathfrak{Z}_{n+1}\times\mathfrak{Z}_n$ is the usual group homomorphism from $\mathbb{Z}^{\mathfrak{Z}_n}$ ($\cong K_0(A_n)$) to $\mathbb{Z}^{\mathfrak{Z}_{n+1}}$ ($\cong K_0(A_{n+1})$) arising from $A_n \hookrightarrow A_{n+1}$ in a standard representation of dimension groups for finite-dimensional $C^*$-algebras. Thus, $\mathcal{Z}(A_n) \overset{\Theta_{n+1,n}}{\longrightarrow} \mathcal{Z}(A_{n+1})$ coincides with $K_0(A_n)\otimes_\mathbb{Z}\mathbb{C} \to K_0(A_{n+1})\otimes_\mathbb{Z}\mathbb{C}$ arising from $A_n \hookrightarrow A_{n+1}$ under the realization of $K_0(A_n)$ inside $\mathcal{Z}(A_n)$ explained in Section \ref{S3.1}, and accordingly, $\mathcal{Z}(A_n)_h \overset{\Theta_{n+1,n}}{\longrightarrow} \mathcal{Z}(A_{n+1})_h$ does with $K_0(A_n)\otimes_\mathbb{Z}\mathbb{R} \to K_0(A_{n+1})\otimes_\mathbb{Z}\mathbb{R}$.
\end{Remark}

Because of this remark, we will regard inductive sequence \eqref{Eq3.4} as a counterpart of inductive sequences of dimension groups. We also point out that this is completely consistent with the picture that Olshanski explained, in \cite[Section 1.4]{Olshanski:AdvMath16}, about his view of ``Fourier transform'' on~$\mathrm{U}(\infty)$.

\subsubsection{Inductive limit} 

The concept of inductive limits of operator systems was first introduced and used by Kirchberg~\cite{Kirchberg:JFA95}, but he treated it in the category of norm-complete operator systems. A systematic study of the concept was recently conducted by Mawhinney and Todorov \cite{MawhinneyTodorov:DissertationesMath18} in the category of general (not necessarily norm-complete) operator systems. We will follow this recent study to discuss the inductive limit of \eqref{Eq3.4}.\pagebreak

\begin{Definition}
\quad
\begin{enumerate}\itemsep=0pt
\item[$1.$] Let
\[
\ddot{\mathcal{S}}(\alpha^t,\beta) := \varinjlim \big(\mathcal{Z}(A_n) \overset{\Theta_{n+1,n}}{\longrightarrow} \mathcal{Z}(A_{n+1})\big)
\]
be the inductive limit in the category $\mathbf{MOU}$ of matrix ordered $*$-vector spaces with mat\-rix order unit {\rm(}say {\rm MOU} spaces{\rm)}, equipped with matrix order unit $\ddot{e}$ and UCP maps $\ddot{\Theta}_{\infty,n} \colon \mathcal{Z}(A_n) \to \ddot{\mathcal{S}}(\alpha^t,\beta)$ such that $\ddot{\Theta}_{\infty,n+1}\circ\Theta_{n+1,n} = \ddot{\Theta}_{\infty,n}$ holds for every $n \geq 0$.

\item[$2.$] Taking the ``Archimedeation'' quotient (see, e.g., \cite[Theorem 2.6 and~Remark 2.7]{MawhinneyTodorov:DissertationesMath18}) gives the inductive limit
\[
\mathcal{S}(\alpha^t,\beta) := \varinjlim \big(\mathcal{Z}(A_n) \overset{\Theta_{n+1,n}}{\longrightarrow} \mathcal{Z}(A_{n+1})\big)
\]
in the category $\mathbf{OS}$ of operator systems, equipped with matrix order unit $e$ and UCP maps $\Theta_{\infty,n} \colon \mathcal{Z}(A_n) \to \mathcal{S}(\alpha^t,\beta)$ such that $\Theta_{\infty,n+1}\circ\Theta_{n+1,n} = \Theta_{\infty,n}$ holds for every $n \geq 0$. Namely, there is a UCP quotient map $q_\mathcal{S} \colon \ddot{\mathcal{S}}(\alpha^t,\beta) \twoheadrightarrow \mathcal{S}(\alpha^t,\beta)$ such that
\begin{itemize}\itemsep=0pt
\item[$(i)$] $\Theta_{\infty,n} = q_\mathcal{S}\circ\ddot{\Theta}_{\infty,n}$ holds for every $n \geq 0$,
\item[$(ii)$] $q_\mathcal{S}(\ddot{\Theta}_{\infty,n}(f)) = 0$ if and only if $\lim_{m\to\infty} \Vert\Theta_{m,n}(f)\Vert = 0$ \cite[Proposition 3.9]{MawhinneyTodorov:DissertationesMath18}.
\end{itemize}
\end{enumerate}
\end{Definition}

We call these $\ddot{\mathcal{S}}(\alpha^t,\beta)$ and $\mathcal{S}(\alpha^t,\beta)$ the \emph{representation system} and the \emph{representation operator system}, respectively, of the $C^*$-flow $\alpha^t$ at inverse temperature $\beta$.

\begin{Remarks}\label{R3.4} \quad
\begin{enumerate}\itemsep=0pt
\item[$1.$] By its construction, $\ddot{\mathcal{S}}(\alpha^t,\beta)$ is also the inductive limit in the category $\mathbf{OU}$ of ordered $*$-vactor spaces with order unit (say {\rm OU} spaces) as well as even in the category of vector spaces. Thus, for any sequence of (unital positive) linear maps $\Psi_n \colon \mathcal{Z}(A_n) \to \mathcal{V}$ with a single ({\rm OU}) vector space $\mathcal{V}$ such that $\Psi_{n+1}\circ\Theta_{n+1,n} = \Psi_n$ holds for every $n \geq 0$, the correspondence $\ddot{\Theta}_{\infty,n}(f) \mapsto \Psi_n(f)$ defines a well-defined (unital positive) linear map from $\ddot{\mathcal{S}}(\alpha^t,\beta)$ to $\mathcal{V}$. See the proof of \cite[Theorem 3.5]{MawhinneyTodorov:DissertationesMath18}. This property allows us to describe the vector space $\ddot{\mathcal{S}}(\alpha^t,\beta)$ itself as follows. Let $\sum_{n\geq0} \mathcal{Z}(A_n)$ be the algebraic direct sum of the $\mathcal{Z}(A_n)$. Define a transition operator $P \colon \sum_{n\geq0} \mathcal{Z}(A_n) \to \sum_{n\geq0} \mathcal{Z}(A_n)$ by $f \in \mathcal{Z}(A_n) \mapsto \Theta_{n+1,n}(f) \in \mathcal{Z}(A_{n+1})$ in the direct sum $\sum_{n\geq0} \mathcal{Z}(A_n)$. With the identity operator $I$ on $\sum_{n\geq0}\mathcal{Z}(A_n)$, the vector space $\ddot{\mathcal{S}}(\alpha^t,\beta)$ is identified with the quotient space
\[
\bigg(\sum_{n\geq0}\mathcal{Z}(A_n)\bigg)\bigg\slash\operatorname{Im}(P-I)
\]
by the image of $P-I$, and $\ddot{\Theta}_{\infty,n}$ is given by the composition of the inclusion map $\mathcal{Z}(A_n) \hookrightarrow \sum_{n\geq0}\mathcal{Z}(A_n)$ and the quotient map. This can be shown by using \cite[Remark~3.1]{MawhinneyTodorov:DissertationesMath18}; actually, $\sum_{n=0}^N \ddot{\Theta}_{\infty,n}(f_n) = 0$ implies $\sum_{n=0}^N \Theta_{N',n}(f_n) = 0$ in $\mathcal{Z}(A_{N'})$ for some $N' \geq N$ so that $f_0 = - g_0$ and $f_{n+1} = \Theta_{n+1,n}(g_{n}) - g_{n+1}$ ($n \geq 0$) with $g_n := \sum_{k=0}^{n-1} \Theta_{n,k}(f_k) - f_n$ ({\it n.b.}, $g_n = 0$ if $n > N'$ by assumption). This description of $\ddot{\mathcal{S}}(\alpha^t,\beta)$ is completely consistent with the description of dimension groups used in \cite[Chapter 2, Section 2]{VershikKerov:JSovietMath87} and \cite[Section~1.1]{VershikKerov:Proc90}.

\item[$2.$] Since $\mathbf{OS}$ is understood as the category of \emph{not necessarily norm-complete} operator systems, $\mathcal{S}(\alpha^t,\beta) = \bigcup_{n\geq0} \Theta_{\infty,n}(\mathcal{Z}(A_n))$ holds by definition. In our analysis we regard the representation (operator) system $\mathcal{S}(\alpha^t,\beta)$ or $\ddot{\mathcal{S}}(\alpha^t,\beta)$ as a counterpart of the complexification of ``dimension group'' associated with $A$ because of Remark \ref{R3.2}.

\item[$3.$] The hermitian elements $\ddot{\mathcal{S}}(\alpha^t,\beta)_h = \{ x \in \ddot{\mathcal{S}}(\alpha^t,\beta);\, x=x^*\}$ and $\mathcal{S}(\alpha^t,\beta)_h = \{ x \in \mathcal{S}(\alpha^t,\beta);\, x=x^*\}$ are exactly $\bigcup_{n\geq0} \ddot{\Theta}_{\infty,n}(\mathcal{Z}(A_n)_h)$ and $\bigcup_{n\geq0} \Theta_{\infty,n}(\mathcal{Z}(A_n)_h)$, respectively. One may regard each of these ``real subspaces'' as a counterpart of ``dimension group'' of~$A$ to study $(\alpha^t,\beta)$-spherical representations.
 \end{enumerate}
\end{Remarks}

The representation operator system $\mathcal{S}(\alpha^t,\beta)$ has the following special feature:

\begin{Proposition}
Any unital positive map $\Psi \colon \mathcal{S}(\alpha^t,\beta) \to \mathcal{T}$ with another operator system $\mathcal{T}$ is automatically CP.
\end{Proposition}
\begin{proof}
We need to recall the detailed construction of operator system inductive limits in \cite[Section 4.2]{MawhinneyTodorov:DissertationesMath18}. Write $\ddot{\mathcal{S}} = \ddot{\mathcal{S}}(\alpha^t,\beta)$ and $\mathcal{S} := \mathcal{S}(\alpha^t,\beta)$ for the ease of notation. We will use the standard notations such as all the $k\times k$ matrices $\mathbb{M}_k(\ddot{\mathcal{S}})$ whose entries are from $\ddot{\mathcal{S}}$, and $\ddot{\Theta}_{\infty,n}^{(k)}(X) := [\ddot{\Theta}_{\infty,n}(X_{ij})]$ with $X = [X_{ij}] \in \mathbb{M}_k(\ddot{\mathcal{S}})$; see \cite[Section 2.2]{MawhinneyTodorov:DissertationesMath18}.

We observe that $\ddot{\Theta}_{\infty,n}^{(k)}(X) \geq 0$ in $\mathbb{M}_k(\ddot{\mathcal{S}})$ implies that there is an $m > n$ so that $\Theta_{m,n}^{(k)}(X) \geq 0$ in $\mathbb{M}_k(\mathcal{Z}(A_m))$. Then we have
\[
(\Psi\circ q_\mathcal{S})^{(k)}\big(\ddot{\Theta}_{\infty,n}^{(k)}(X)\big) = (\Psi\circ q_\mathcal{S}\circ\ddot{\Theta}_{\infty,m})^{(k)}\big(\Theta_{m,n}^{(k)}(X)\big).
\]
Since $\mathcal{Z}(A_m)$ is a commutative $C^*$-algebra, $\Psi\circ q_\mathcal{S}\circ\ddot{\Theta}_{\infty,m}$ must be CP (see, e.g., \cite[Theorem~3.11]{Paulsen:Book}). Therefore, $(\Psi\circ q_\mathcal{S})^{(k)}\big(\ddot{\Theta}_{\infty,n}^{(k)}(X)\big) \geq 0$ so that $\Psi\circ q_\mathcal{S}$ is CP.

We will confirm that $\Psi$ itself is CP. Choose an arbitrary $q_\mathcal{S}^{(k)}(Y) \geq 0$ in $\mathbb{M}_k(\mathcal{S})$. By construction, for any $r > 0$, there are $Z_r \geq 0$ in $\mathbb{M}_k(\ddot{\mathcal{S}})$ and $N_r \in \mathbb{M}_k(\mathrm{Ker}(q_\mathcal{S}))$ so that $Y + r \ddot{e}^{(k)} = Z_r + N_r$ holds. Applying $\Psi^{(k)}\circ q_\mathcal{S}^{(k)}$ to this identity we have
\[
\Psi^{(k)}\big(q_\mathcal{S}^{(k)}(Y)\big) + r\Psi(e)^{(k)} = \Psi^{(k)}\big(q_\mathcal{S}^{(k)}(Z_r)\big) = (\Psi\circ q_\mathcal{S})^{(k)}(Z_r) \geq 0
\]
in $\mathbb{M}_k(\mathfrak{T})$. By the Archimedean property of $\mathcal{T}$ we conclude that $\Psi^{(k)}\big(q_\mathcal{S}^{(k)}(Y)\big) \geq 0$.
\end{proof}

\subsubsection{Projective limit} 

Since every $\Phi_{n,n+1}$ is contractive as remarked in Section \ref{S3.2.1} we can take the projective limit of \eqref{Eq3.5} in the category of Banach spaces. The projective limit is
\begin{gather*}
\mathcal{L}(\alpha^t,\beta) := \big\{ \omega = (\omega_n);\, \omega_n \in \mathcal{Z}(A_n)_*,\, \Phi_{n,n+1}(\omega_{n+1}) = \omega_n\, \text{for all $n \geq0$},\, \sup_n \Vert\omega_n\Vert < +\infty \big\}
\end{gather*}
equipped with norm $\Vert\omega\Vert := \sup_n \Vert\omega_n\Vert$ for $\omega=(\omega_n) \in \mathcal{L}(\alpha^t,\beta)$ and contractive linear maps $\Phi_{n,\infty} \colon \mathcal{L}(\alpha^t,\beta) \to \mathcal{Z}(A_n)_*$ sending $\omega=(\omega_n)$ to $\omega_n$. It is well known that $(\mathcal{L}(\alpha^t,\beta),\Vert\,\cdot\,\Vert)$ becomes a Banach space.

Here is a possible duality relation between $\mathcal{L}(\alpha^t,\beta)$ and $\mathcal{S}(\alpha^t,\beta)$.

\begin{Proposition}\label{P3.6} There exists an isometric embedding $\iota \colon \mathcal{L}(\alpha^t,\beta) \hookrightarrow \mathcal{S}(\alpha^t,\beta)^*$ such that
\[
\iota(\omega)(\Theta_{\infty,n}(f)) = \langle \Phi_{n,\infty}(\omega),f\rangle_n, \qquad f \in \mathcal{Z}(A_n),\quad \omega \in \mathcal{L}(\alpha^t,\beta).
\]
\end{Proposition}
\begin{proof}
For any $\omega \in \mathcal{L}(\alpha^t,\beta)$, we observe that
\[
\langle \Phi_{n+1,\infty}(\omega),\Theta_{n+1,n}(f)\rangle_{n+1}
=
\langle \Phi_{n,n+1}(\Phi_{n+1,\infty}(\omega)),f)_n = \langle \Phi_{n,\infty}(\omega),f)_n
\]
for every $f \in \mathcal{Z}(A_n)$. Hence, $\ddot{\Theta}_{\infty,n}(f) \mapsto \langle \Phi_{n,\infty}(\omega),f\rangle_n$ defines a well-defined linear functional~$\ddot{\iota}(\omega)$ on $\ddot{\mathcal{S}}(\alpha^t,\beta)$.

For any $\omega \in \mathcal{L}(\alpha^t,\beta)$ and $f \in \mathcal{Z}(A_n)$, we have
\begin{align*}
|\ddot{\iota}(\omega)(\ddot{\Theta}_{\infty,n}(f))|
&=
|\ddot{\iota}(\omega)(\ddot{\Theta}_{\infty,n+k}(\Theta_{n+k,n}(f))| \\
&=
|\langle\Phi_{n+k,\infty}(\omega),\Theta_{n+k,n}(f)\rangle_m| \\
&\leq
\Vert\Phi_{n+k,\infty}(\omega)\Vert\,\Vert\Theta_{n+k,n}(f)\Vert \\
&\leq
\Vert\omega\Vert\,\Vert\Theta_{n+k,n}(f)\Vert
\longrightarrow \Vert\omega\Vert\,\Vert\Theta_{\infty,n}(f)\Vert
\end{align*}
as $k\to\infty$ by \cite[Proposition 4.16]{MawhinneyTodorov:DissertationesMath18}. It follows that $\mathrm{Ker}(q_\mathcal{S}) \subseteq \mathrm{Ker}(\ddot{\iota}(\omega))$. Hence there is a~unique $\iota(\omega) \in \mathcal{S}(\alpha^t,\beta)^*$ so that $\iota(\omega)\circ q_\mathcal{S} = \ddot{\iota}(\omega)$, and $\Vert\iota(\omega)\Vert \leq \Vert\omega\Vert$ holds.

Moreover, for any $\varepsilon>0$ there is an $m$ so that $\Vert \omega\Vert - \varepsilon/2 < \Vert\Phi_{m,\infty}(\omega)\Vert$. Also, there is a~$g \in \mathcal{Z}(A_m)$ such that $\Vert g\Vert \leq 1$ and $\Vert\Phi_{m,\infty}(\omega)\Vert - \varepsilon/2 < |\langle \Phi_{m,\infty}(\omega),g\rangle_m|$. Then we obtain that $\Vert\Theta_{\infty,m}(g)\Vert \leq \Vert g\Vert \leq 1$ and $\Vert\iota(\omega)\Vert \geq |\iota(\omega)(\Theta_{\infty,m}(g))| = |\langle\Phi_{m,\infty}(\omega),g\rangle_m| > \Vert\Phi_{m,\infty}(\omega)\Vert - \varepsilon/2 > \Vert\omega\Vert - \varepsilon$. Since $\varepsilon>0$ is arbitrary, we conclude that $\Vert\iota(\omega)\Vert = \Vert\omega\Vert$.

Since all the $\Phi_{n,\infty}$ are linear, so is $\omega \mapsto \iota(\omega)$ and thus we are done.
\end{proof}

We will clarify what r\^{o}le the Banach space $\mathcal{L}(\alpha^t,\beta)$ plays in the next section.

\subsection[Locally normal states on S(alpha\textasciicircum{}t,beta)]
{Locally normal states on $\boldsymbol{\mathcal{S}(\alpha^t,\beta)}$} 

We begin with a definition.

\begin{Definition} An $\omega \in \mathcal{S}(\alpha^t,\beta)^*$ is said to be \emph{locally normal}, if there is an $m \geq 0$ such that $\Theta^*_{\infty,n}(\omega) := \omega\circ\Theta_{\infty,n}$ falls into $\mathcal{Z}(A_n)_*$ for every $n \geq m$. The set of all locally normal states on~$\mathcal{S}(\alpha^t,\beta)$ is denoted by $S^\mathrm{ln}(\mathcal{S}(\alpha^t,\beta))$, whose topology is the weak$^*$ one.

A locally normal state on $\ddot{\mathcal{S}}(\alpha^t,\beta)$ is defined similarly with $\ddot{\Theta}_{\infty,n}$ in place of $\Theta_{\infty,n}$, and the set of those states is denoted by $S^\mathrm{ln}(\ddot{\mathcal{S}}(\alpha^t,\beta))$.
\end{Definition}

\begin{Remarks}
\quad
\begin{enumerate}\itemsep=0pt
\item[$1.$] If $\omega$ is locally normal, then $\Theta_{\infty,n}^*(\omega)$ falls into $\mathcal{Z}(A_n)_*$ for every $n \geq 0$ since all $\Theta_{n+1,n}$ are normal.

\item[$2.$] It is clear, from the above definition, that the canonical bijection $\omega \in S(\mathcal{S}(\alpha^t,\beta)) \mapsto \omega\circ q_\mathcal{S} \in S(\ddot{\mathcal{S}}(\alpha^t,\beta))$ between state spaces (see, e.g., \cite[Theorem 2.6]{MawhinneyTodorov:DissertationesMath18}) sends $S^\mathrm{ln}(\mathcal{S}(\alpha^t,\beta))$ onto $S^\mathrm{ln}(\ddot{\mathcal{S}}(\alpha^t,\beta))$. Thus, no difference between $\mathcal{S}(\alpha^t,\beta)$ and $\ddot{\mathcal{S}}(\alpha^t,\beta)$ occurs in the study of locally normal states.
 \end{enumerate}
\end{Remarks}

\begin{Lemma}\label{L3.9} The set of all locally normal elements of $\mathcal{S}(\alpha^t,\beta)^*$ is exactly $\iota(\mathcal{L}(\alpha^t,\beta))$ with the embedding $\iota\colon \mathcal{L}(\alpha^t,\beta) \hookrightarrow \mathcal{S}(\alpha^t,\beta)^*$ in Proposition $\ref{P3.6}$.
\end{Lemma}
\begin{proof}
Let $\omega \in \mathcal{S}(\alpha^t,\beta)^*$ be locally normal. Then $\Theta_{\infty,n}^*(\omega) = \omega\circ\Theta_{\infty,n} \in \mathcal{Z}(A_n)_*$ by definition. Observe that
\[
\langle \Phi_{n,n+1}(\Theta_{\infty,n+1}^*(\omega)),f\rangle_n = \omega\circ\Theta_{\infty,n+1}(\Theta_{n+1,n}(f)) = \omega\circ\Theta_{\infty,n}(f) = \langle\Theta_{\infty,n}^*(\omega),f\rangle_n
\]
for all $f \in \mathcal{Z}(A_n)$. It follows that $\Phi_{n,n+1}(\Theta_{\infty,n+1}^*(\omega))=\Theta_{\infty,n}^*(\omega)$ holds for every $n \geq 0$. Clearly, $\Vert\Theta_{\infty,n}^*(\omega)\Vert \leq \Vert\omega\Vert$ holds for every $n\geq0$, and hence $\Theta_{\infty,\bullet}^*(\omega) := (\Theta_{\infty,n}^*(\omega)) \in \mathcal{L}(\alpha^t,\beta)$. Moreover, $\iota(\Theta_{\infty,\bullet}^*(\omega))(\Theta_{\infty,n}(f)) = \langle \Phi_{n,\infty}(\Theta_{\infty,\bullet}^*(\omega)),f\rangle_n = \omega(\Theta_{\infty,n}(f))$ holds for all $f \in \mathcal{Z}(A_n)$ and $n$, and thus $\omega = \iota(\Theta_{\infty,\bullet}^*(\omega)) \in \iota(\mathcal{L}(\alpha^t,\beta))$.

On the other hand, the identify in the statement of Proposition \ref{P3.6} shows that any element of $\iota(\mathcal{L}(\alpha^t,\beta))$ must be locally normal. Hence we are done.
\end{proof}

Here we recall a definition from \cite[Section 7]{Ueda:Preprint20}.

\begin{Definition} A positive function $\nu$ on $\mathfrak{Z}=\bigsqcup_{n\geq0}\mathfrak{Z}_n$ is said to be \emph{normalized} and \emph{$\kappa_{(\alpha^t,\beta)}$-harmonic}, if $\nu(1) = 1$ and
\[
\nu(z') = \sum_{z \in \mathfrak{Z}_{n+1}}\nu(z)\,\kappa_{(\alpha^t,\beta)}(z,z'), \qquad
z' \in \mathfrak{Z}_n
\]
holds for every $n \geq 0$, where $\kappa_{(\alpha^t,\beta)}$ is the link in Definition $\ref{D3.1}$. The convex set of all positive, normalized $\kappa_{(\alpha^t,\beta)}$-harmonic functions is denoted by $H_1^+(\kappa_{(\alpha^t,\beta)})$, whose topology is determined by pointwise convergence.
\end{Definition}

\begin{Theorem}\label{T3.11} The following assertions hold:
\begin{enumerate}\itemsep=0pt
\item[$1.$] The mapping $\nu \in H_1^+(\kappa_{(\alpha^t,\beta)}) \mapsto \nu_\bullet := (\nu_n) \in \mathcal{L}(\alpha^t,\beta)$ with $\nu_n := \nu\!\upharpoonright_{\mathfrak{Z}_n}$, the restriction of~$\nu$ to $\mathfrak{Z}_n$, for each $n$ is well defined, injective and affine.
\item[$2.$] $S^\mathrm{ln}(\mathcal{S}(\alpha^t,\beta)) = \big\{ \iota(\nu_\bullet);\, \nu \in H_1^+(\kappa_{(\alpha^t,\beta)})\big\}$.
\item[$3.$]
There are affine homeomorphisms
\[
\xymatrix{
K_\beta^\mathrm{ln}(\alpha^t) \ar@{->}[rr]^{(a)}
& & H_1^+(\kappa_{(\alpha^t,\beta)}) \ar@{->}[dl]^{(b)} \\
& S^\mathrm{ln}(\mathcal{S}(\alpha^t,\beta)) \ar@{->}[ul]^{(c)}
&
}
\]
such that any composition of these three maps along the arrows becomes the identity map. Here, map $(a)$ is the affine homeomorphism $\omega \mapsto \nu[\omega]$ in {\rm \cite[Proposition 7.10]{Ueda:Preprint20}}, map $(b)$ is the affine bijection $\nu \mapsto \iota(\nu_\bullet)$, and map $(c)$, denoted by $\omega \mapsto E_\bullet^*(\omega)$, is defined as follows. For any $\omega \in S^\mathrm{ln}(\mathcal{S}(\alpha^t,\beta))$, a unique $E_\bullet^*(\omega) \in K_\beta^\mathrm{ln}(\alpha^t)$ is defined in such a way that
\[
E_\bullet^*(\omega)\!\upharpoonright_{A_n} = \omega\circ\Theta_{\infty,n}\circ E_n
\]
for every $n$.
\end{enumerate}
\end{Theorem}

\begin{proof} (1) That $\nu_\bullet \in \mathcal{L}(\alpha^t,\beta)$ immediately follows from the definition of $\kappa_{(\alpha^t,\beta)}$-harmonicity and \eqref{Eq3.7}. Being injective and affine is trivial.

(2) By Lemma \ref{L3.9}, any element of $S^\mathrm{ln}(\mathcal{S}(\alpha^t,\beta))$ is of the form $\iota(\omega)$ with $\omega \in \mathcal{L}(\alpha^t,\beta)$. Then, $\iota(\omega)\circ\Theta_{\infty,n}(f) = \langle \Phi_{n,\infty}(\omega), f\rangle_n$ holds for every $f \in \mathcal{Z}(A_n)$. Since $\iota(\omega)\circ\Theta_{\infty,n}$ is a state, $\Phi_{n,\infty}(\omega)$ is a positive function on $\mathfrak{Z}_n$ and $\sum_{z \in \mathfrak{Z}_n} \Phi_{n,\infty}(\omega)(z) = \Vert\Phi_{n,\infty}(\omega)\Vert = 1$ holds. We define a function $\nu \colon \mathfrak{Z} \to [0,1]$ by $\nu\!\upharpoonright_{\mathfrak{Z}_n} := \Phi_{n,\infty}(\omega)\!\upharpoonright_{\mathfrak{Z}_n}$ for every $n \geq 0$. By the definition of $\mathcal{L}(\alpha^t,\beta)$ together with~\eqref{Eq3.7} one easily sees that $\nu$ falls in $H_1^+(\kappa_{(\alpha^t,\beta)})$ and moreover $\iota(\nu_\bullet) = \omega$ by the construction of~$\nu$. Hence we have shown that $S^\mathrm{ln}(\mathcal{S}(\alpha^t,\beta)) \subseteq \big\{ \iota(\nu_\bullet);\, \nu \in H_1^+(\kappa_{(\alpha^t,\beta)})\big\}$.

On the other hand, we start with an arbitrary $\nu \in H_1^+(\kappa_{(\alpha^t,\beta)})$. By item (1) $\nu_\bullet$ falls in $\mathcal{L}(\alpha^t,\beta)$. Then, for each $n$, $\iota(\nu_\bullet)(\Theta_{\infty,n}(f)) = \langle \nu_n,f\rangle_n = \sum_{z\in\mathfrak{Z}_n} \nu_n(z) f(z)$ is always positive if $f \in \mathcal{Z}(A_n)$, and equals $1$ when $f \equiv 1$. Thus, $\iota(\nu_\bullet)$ must be a state on $\mathcal{S}(\alpha^t,\beta)$. Since $\iota(\nu_\bullet)$ has already been known to be locally normal by Lemma \ref{L3.9}, we have $\iota(\nu_\bullet) \in S^\mathrm{ln}(\mathcal{S}(\alpha^t,\beta))$. Hence we have confirmed item (2).

(3) Let us first consider map $(c)$. Denote $E_n^*(\omega) := \omega\circ\Theta_{\infty,n}\circ E_n$, which clearly defines a~normal state on $A_n$. Since $\Theta_{n+1,n}$ is the restriction of $E_{n+1}$ to $\mathcal{Z}(A_n)$, we have, for any $a \in A_n$,
\begin{align*}
E_{n+1}^*(\omega)(a)
&= \omega\circ\Theta_{\infty,n+1}\circ E_{n+1}(a) \\
&= \omega\circ\Theta_{\infty,n+1}\circ E_{n+1}\circ E_n(a) \quad \text{(by \ref{Eq2.2})} \\
&= \omega\circ\Theta_{\infty,n+1}\circ \Theta_{n+1,n}\circ E_n(a) \\
&= \omega\circ\Theta_{\infty,n}\circ E_n(a) = E_n^*(\omega)(a).
\end{align*}
Thus, by the proof of \cite[Proposition 7.4]{Ueda:Preprint20} the sequence $(E_n^*(\omega))$ defines the desired element $E_\bullet^*(\omega) \in K_\beta^\mathrm{ln}(\alpha^t)$. Hence we have seen that map~$(c)$ is well defined. The injectivity of map~$(c)$ is clear from the construction, since $\bigcup_n A_n$ is norm-dense in $A$. Also, it is easy to see that $\omega \mapsto E_\bullet^*(\omega)$ is affine.

Assume that $\omega_\lambda \to \omega$ in $S^\mathrm{ln}(\mathcal{S}(\alpha^t,\beta))$ (with respect to the weak$^*$ topology). For any $a \in A_n$ we have $E_\bullet^*(\omega_\lambda)(a) = \omega_\lambda(\Theta_{\infty,n}(E_n(a))) \to \omega(\Theta_{\infty,n}(E_n(a))) = E_\bullet^*(\omega)(a)$. Since $\bigcup_n A_n$ is norm-dense in $A$, we conclude that $E_\bullet^*(\omega_\lambda) \to E_\bullet^*(\omega)$ in $K_\beta^\mathrm{ln}(\alpha^t)$ (with respect to the weak$^*$ topology). Thus, map~$(c)$ is continuous.

We then prove that map~$(b)$ is continuous. Assume that $\nu_\lambda \to \nu$ in $H_1^+(\kappa_{(\alpha^t,\beta)})$ (pointwisely). It is a standard task (see the proofs of \cite[Theorem 7.8 and Proposition 7.10]{Ueda:Preprint20}) to show that
\begin{align*}
\iota((\nu_\lambda)_\bullet)(\Theta_{\infty,n}(f))
= \sum_{z \in \mathfrak{Z}_n} \nu_\lambda(z)f(z)
\to
\sum_{z \in \mathfrak{Z}_n} \nu(z)f(z)
= \iota(\nu_\bullet)(\Theta_{\infty,n}(f))
\end{align*}
for any $f \in \mathcal{Z}(A_n)$ and $n$. It follows that $\iota((\nu_\lambda)_\bullet) \to \iota(\nu_\bullet)$ in $S^\mathrm{ln}(\mathcal{S}(\alpha^t,\beta))$ (with respect to the weak$^*$ topology).

We finally confirm that the diagram commutes in the sense of the statement. Choose an arbitrary $\omega \in K_\beta^\mathrm{ln}(\alpha^t)$. For any $a \in A_n$ we have
\[
E_\bullet^*(\iota(\nu[\omega]_\bullet))(a) = \iota(\nu[\omega]_\bullet)\circ\Theta_{\infty,n}(E_n(a)) = \langle \nu[\omega]_n,E_n(a)\rangle_n = \omega(a)
\]
by \eqref{Eq3.1}, since $\nu[\omega]_n(z) = \omega(z)$ for every $z \in \mathfrak{Z}_n$. Since $\bigcup_n A_n$ is norm-dense in $A$, it follows that $E_\bullet^*(\iota(\nu[\omega]_\bullet)) = \omega$ holds, that is, the diagram commutes. This fact together with the bijectivity of both maps $(a)$, $(b)$ implies that map~$(c)$ is bijective too. The commutativity also shows that maps $(b)$, $(c)$ are homeomorphisms.
\end{proof}

Thanks to Remark \ref{R3.2}, the above assertion (3) can be regarded as a KMS state analog of a~well-known fact on tracial states on AF-algebras with dimension groups (see, e.g., \cite[Chapter~2, Section 3.2]{VershikKerov:JSovietMath87}). Actually, the representation (operator) system $\mathcal{S}(\alpha^t,\beta)$ (or $\ddot{\mathcal{S}}(\alpha^t,\beta)$) with inductive structure $\Theta_{\infty,n}$ (resp.\ $\ddot{\Theta}_{\infty,n}$) has all information on the locally normal $(\alpha^t,\beta)$-KMS states $K_\beta^\mathrm{ln}(\alpha^t)$.

\subsection{Left/right multiplicative structure} 

Throughout this section, we assume that the inductive sequence $A_n$ in question has the following structure, called a \emph{left multiplicative structure}: For each pair $m, n \geq 0$, there is a normal injective $*$-homomorphism $\iota_{m,n} \colon A_m\,\bar{\otimes}\,A_n \to A_{m+n}$ such that
\begin{itemize}\itemsep=0pt
\item[$(m1)$] $\iota_{0,n}$ and $\iota_{m,0}$ are the natural maps $\mathbb{C}\otimes A_n \cong A_n$ and $A_m\otimes\mathbb{C} \cong A_m$, respectively;
\item[$(m2)$] for each triple $\ell,m,n\geq0$,
\[
\xymatrix{
A_\ell\,\bar{\otimes}\,A_m\,\bar{\otimes}\,A_n \ar@{->}[r]^{\iota_{\ell,m}\,\bar{\otimes}\,\mathrm{id}} \ar@{->}[d]_{\mathrm{id}\,\bar{\otimes}\,\iota_{m,n}} & A_{\ell+m}\,\bar{\otimes}\,A_n \ar@{->}[d]^{\iota_{\ell+m,n}} \\
A_\ell\,\bar{\otimes}\,A_{m+n} \ar@{->}[r]^{\iota_{\ell,m+n}} \ar@{}[ur]|{\circlearrowright} & A_{\ell+m+n}
}
\]
and this commutative diagram defines $\iota_{\ell,m,n}\colon A_\ell\,\bar{\otimes}\,A_m\,\bar{\otimes}\,A_n \to A_{\ell+m+n}$ naturally;
\item[$(m3)$] $\iota_{m,n}\circ(\alpha_m^t\,\bar{\otimes}\,\alpha_n^t) = \alpha_{m+n}^t\circ\iota_{m,n}$ holds for every pair $m$, $n$;
\item[$(L)$] for each $n \geq 0$,
\[
\xymatrix{
A_n \ar@{^{(}-_{>}}[rr] \ar@{->}[dr]_{a\mapsto a\otimes1}
&
& A_{n+1} \\
& A_n\,\bar{\otimes}\,A_1 ,\ar@{->}[ur]_{\iota_{n,1}} \ar@{}[u]|{\circlearrowleft}
&
}
\]
where the horizontal arrow is the natural inclusion.
\end{itemize}
We call an inductive sequence $A_n$ with left multiplicative structure $\iota_{m,n}$ a \emph{left multiplicative inductive sequence}. The notion of right multiplicative structure is defined analogously with changing $(L)$ into
\begin{itemize}\itemsep=0pt
\item[$(R)$] For each $n \geq 0$,
\[
\xymatrix{
A_n \ar@{^{(}-_{>}}[rr] \ar@{->}[dr]_{a\mapsto 1\otimes a}
&
& A_{n+1} \\
& A_1\,\bar{\otimes}\,A_n, \ar@{->}[ur]_{\iota_{1,n}} \ar@{}[u]|{\circlearrowleft}
&
}
\]
where the horizontal arrow is the natural inclusion.
\end{itemize}
We remark that the discussion below also works for right multiplicative structures with trivial modifications.

Okada's abstract formulation of Littlewood--Richardson rings \cite{Okada:JpnProc94} motivated us to introduce this structure, but item $(L)$ or $(R)$ is a tricky part and formulated to hold in the quantum group setup; see Section \ref{S5.4}.

We will prove that the above left multiplicative structure makes the algebraic direct sum $\Sigma(\alpha^t,\beta) := \sum_{n\geq0} \mathcal{Z}(A_n)$ a unital algebra over $\mathbb{C}$, and moreover, both $\ddot{\mathcal{S}}(\alpha^t,\beta)$ and $\mathcal{S}(\alpha^t,\beta)$ left $\Sigma(\alpha^t,\beta)$-modules.

We start with a technical lemma, which follows from the uniqueness of $\tau_z = \tau^{(\alpha^t_z,\beta)}$.

\begin{Lemma}\label{L3.12} For each triple $\ell, m, n \geq 0$ and any $z_1 \in \mathfrak{Z}_\ell$, $z_2 \in \mathfrak{Z}_m$, $z_3 \in \mathfrak{Z}_n$, $z \in \mathfrak{Z}_{\ell+m+n}$, $z' \in \mathfrak{Z}_{\ell+m}$, $z'' \in \mathfrak{Z}_{m+n}$ we have the following:
\begin{itemize}\itemsep=0pt
\item[$1.$] $\tau_{z'}(x)\,\tau_z(z\,\iota_{\ell+m,n}(z'\otimes z_3)) = \tau_z(z\,\iota_{\ell+m,n}(x\otimes z_3))$ for all $x \in z'A_{\ell+m}$.
\item[$2.$] $\tau_z(z\,\iota_{\ell,m+n}(z_1\otimes z''))\,\tau_{z''}(y) = \tau_z(z\,\iota_{\ell,m+n}(z_1\otimes y))$ for all $y \in z'' A_{m+n}$.
\end{itemize}
\end{Lemma}
\begin{proof}
These can be proved in the same way as in \cite[Lemma 7.2]{Ueda:Preprint20} using the uniqueness of normal $(\alpha_{z'}^t,\beta)$- and $(\alpha_{z''}^t,\beta)$-KMS states on $z' A_{\ell+m}$ and $z'' A_{m+n}$, respectively, where requirement $(m3)$ is important.
\end{proof}

\begin{Corollary}\label{C3.13} For each triple $\ell,m,n \geq 0$ we have\vspace{-1ex}
\begin{align*}
E_{\ell+m+n}\circ\iota_{\ell+m,n}\circ(E_{\ell+m}\,\bar{\otimes}\,\mathrm{id})\circ(\iota_{\ell,m}\,\bar{\otimes}\,\mathrm{id})
&= E_{\ell+m+n}\circ\iota_{\ell,m,n} \\
&=
E_{\ell+m+n}\circ\iota_{\ell,m+n}\circ(\mathrm{id}\,\bar{\otimes}\,E_{m+n})\circ(\mathrm{id}\,\bar{\otimes}\,\iota_{m,n})
\end{align*}
holds on $\mathcal{Z}(A_\ell)\,\bar{\otimes}\,\mathcal{Z}(A_m)\,\bar{\otimes}\,\mathcal{Z}(A_n)\,(\subset A_\ell\,\bar{\otimes}\,A_m\,\bar{\otimes}\,A_n)$.
\end{Corollary}
\begin{proof}
It suffices to confirm the desired identity against $z_1\otimes z_2 \otimes z_3$ with $z_1 \in \mathfrak{Z}_\ell$, $z_2 \in \mathfrak{Z}_m$, $z_3 \in \mathfrak{Z}_n$, because all the involved maps are normal.

We have\vspace{-1ex}
\begin{align*}
&\big(E_{\ell+m+n}\circ\iota_{\ell+m,n}\circ(E_{\ell+m}\,\bar{\otimes}\,\mathrm{id}) \circ(\iota_{\ell,m}\,\bar{\otimes}\,\mathrm{id})\big)(z_1\otimes z_2\otimes z_3) \\
&\qquad=
\sum_{z\in\mathfrak{Z}_{\ell+m+n}}\sum_{z' \in \mathfrak{Z}_{\ell+m}} \tau_{z'}(z'\,\iota_{\ell,m}(z_1\otimes z_2))\,\tau_z(z\,\iota_{\ell+m,n}(z'\otimes z_3))\,z \\
&\qquad=
\sum_{z\in\mathfrak{Z}_{\ell+m+n}}\sum_{z' \in \mathfrak{Z}_{\ell+m}} \tau_z(z\,\iota_{\ell+m,n}((z'\,\iota_{\ell,m}(z_1\otimes z_2))\otimes z_3))\,z \\
&\qquad=
\sum_{z\in\mathfrak{Z}_{\ell+m+n}} \tau_z(z\,\iota_{\ell,m,n}(z_1\otimes z_2\otimes z_3))\,z \\
&\qquad=
\sum_{z\in\mathfrak{Z}_{\ell+m+n}}\sum_{z'' \in \mathfrak{Z}_{m+n}} \tau_z(z\,\iota_{\ell,m+n}(z_1\otimes(z''\,\iota_{m,n}(z_2\otimes z_3))))\,z \\
&\qquad=
\sum_{z\in\mathfrak{Z}_{\ell+m+n}}\sum_{z'' \in \mathfrak{Z}_{m+n}} \tau_z(z\,\iota_{\ell,m+n}(z_1\otimes z''))\,\tau_{z''}(z''\,\iota_{m,n}(z_2\otimes z_3))))\,z \\
&\qquad=
\big(E_{\ell+m+n}\circ\iota_{\ell,m+n}\circ(\mathrm{id}\,\bar{\otimes}\,E_{m+n})\circ(\mathrm{id}\,\bar{\otimes}\,\iota_{m,n})\big)(z_1\otimes z_2\otimes z_3),
\end{align*}
where the second equality follows from Lemma \ref{L3.12}(1) with $x=z'\,\iota_{\ell,m}(z_1\otimes z_2)$, the third from~$(m2)$, and the fifth from Lemma \ref{L3.12}(2) with $y = z''\,\iota_{m,n}(z_2\otimes z_3)$. Note that the fourth line is $(E_z\circ\iota_{\ell,m,n})(z_1\otimes z_2\otimes z_3)$. Hence we are done.
\end{proof}

The corollary implies the next theorem. Remark that items (2), (3) below are naturally expected from item (1) and Remarks \ref{R3.4}(2). In what follows, the topology on $\Sigma(\alpha^t,\beta)$ is the relative one induced, via the natural embedding, from the product topology on $\prod_{n\geq0} \mathcal{Z}(A_n)$ of the $\sigma$-weak topology on each $\mathcal{Z}(A_n)$.

\begin{Theorem}\label{T3.14} With the left multiplicative structure $\iota_{m,n}$ we have the following{\rm:}
\begin{itemize}\itemsep=0pt
\item[$1.$] $\Sigma(\alpha^t,\beta) = \sum_{n\geq0}\mathcal{Z}(A_n)$ is a unital algebra with multiplication
\[
\sum_n f_n \cdot \sum_n g_n := \sum_n \bigg(\sum_{k+\ell=n} E_n(\iota_{k,\ell}(f_k\otimes g_\ell))\bigg),
\]
and its unit is $1_0$. Here, the unit of $\mathcal{Z}(A_n)$ is denoted by $1_n$ when it is considered as an element of $\Sigma(\alpha^t,\beta)$. Moreover, the multiplication is separately continuous.
\item[$2.$] The mapping
\begin{gather}
(f,\ddot{\Theta}_{\infty,n}(g)) \in \mathcal{Z}(A_m)\times\ddot{\mathcal{S}}(\alpha^t,\beta) \nonumber
\\ \qquad
{}\mapsto
f\cdot\ddot{\Theta}_{\infty,n}(g) := \ddot{\Theta}_{\infty,m+n}(E_{m+n}(\iota_{m,n}(f\otimes g))) \in \ddot{\mathcal{S}}(\alpha^t,\beta)\label{Eq3.9}
\end{gather}
is well defined and gives a left $\Sigma(\alpha^t,\beta)$-module structure on $\ddot{\mathcal{S}}(\alpha^t,\beta)$. Moreover, repla\-cing~$\ddot{\Theta}_{\infty,n}$ with $\Theta_{\infty,n}$ in \eqref{Eq3.9} still works for $\mathcal{S}(\alpha^t,\beta)$. Furthermore,
\[
f\cdot\ddot{e} = \ddot{\Theta}_{\infty,m}(f), \qquad f\cdot e = \Theta_{\infty,m}(f)
\]
hold for every $m \geq 0$ and $f \in \mathcal{Z}(A_m)$.
\item[$3.$] For every $n \geq 0$ and $f \in \mathcal{Z}(A_n) \subset \Sigma(\alpha^t,\beta)$, $f\cdot 1_1 = E_{n+1}(f)$ holds in $\Sigma(\alpha^t,\beta)$. Then the kernel of the surjective linear mapping $x \in \Sigma(\alpha^t,\beta) \mapsto x\cdot\ddot{e} \in \ddot{\mathcal{S}}(\alpha^t,\beta)$ is the principal left ideal $\Sigma(\alpha^t,\beta)\cdot(1_1 - 1_0)$, and hence $\Sigma(\alpha^t,\beta)/(\Sigma(\alpha^t,\beta)\cdot(1_1 - 1_0)) \cong \ddot{\mathcal{S}}(\alpha^t,\beta)$ naturally.
\end{itemize}
\end{Theorem}

\begin{proof}
(1) Since all the involved maps are linear, it suffices to confirm the associativity and that $1 \in \mathcal{Z}(A_0)=\mathbb{C}$ is a unit.

{\samepage
By Corollary \ref{C3.13} we have
\begin{align*}
\bigg(\sum_n f_n\cdot\sum_n g_n\bigg)\cdot\sum_n h_n
&=
\sum_n \bigg(\sum_{k+\ell=n} E_n(\iota_{k,\ell}(f_k\otimes g_\ell))\bigg)\cdot\sum_n h_n \\
&=
\sum_n\bigg(\sum_{k+\ell+m=n} E_n(\iota_{k+\ell,m}(E_{k+\ell}(\iota_{k,\ell}(f_k\otimes g_\ell))\otimes h_m)\bigg) \\
&=
\sum_n \bigg(\sum_{k+\ell+m=n} E_n(\iota_{k,\ell+m}(f_k\otimes E_{\ell+m}(\iota_{\ell,m}(g_\ell\otimes h_m)))\bigg) \\
&=
\sum_n f_n\cdot \sum_n \bigg(\sum_{\ell+m=n} E_n(\iota_{\ell,m}(g_\ell\otimes h_m))\bigg) \\
&=
\sum_n f_n\cdot\Big(\sum_n g_n\cdot\sum_n h_n\Big).
\end{align*}
That $1_0$ is a unit follows from requirement $(m1)$.

}

The separate continuity of multiplication follows from that the $E_n$ and the $\iota_{m,n}$ are all normal maps.

(2) We first confirm that \eqref{Eq3.9} is well defined. We have
\begin{align*}
&\ddot{\Theta}_{\infty,m+n+1}(E_{m+n+1}(\iota_{m,n+1}(f\otimes\Theta_{n+1,n}(g)))) \\
&\qquad=
\ddot{\Theta}_{\infty,m+n+1}(E_{m+n+1}(\iota_{m,n+1}(f\otimes E_{n+1}(\iota_{n,1}(g\otimes1))))) \quad \text{(use $(L)$)} \\
&\qquad=
\ddot{\Theta}_{\infty,m+n+1}(E_{m+n+1}(\iota_{m+n,1}(E_{m+n}(\iota_{m,n}(f\otimes g))\otimes1))) \quad \text{(use Corollary \ref{C3.13})} \\
&\qquad=
\ddot{\Theta}_{\infty,m+n+1}(\Theta_{m+n+1,m+n}(E_{m+n}(\iota_{m,n}(f\otimes g)))) \quad \text{(use $(L)$)} \\
&\qquad=
\ddot{\Theta}_{\infty,m+n}(E_{m+n}(\iota_{m,n}(f\otimes g))),
\end{align*}
which shows that the mapping given by \eqref{Eq3.9} in question is well defined thanks to \cite[Remark~3.1]{MawhinneyTodorov:DissertationesMath18}. It is obvious that the resulting map is bilinear.

For any $f_1 \in \mathcal{Z}(A_\ell)$, $f_2 \in \mathcal{Z}(A_m)$ and $\ddot{\Theta}_{\infty,n}(g) \in \ddot{\mathcal{S}}(\alpha^t,\beta)$, we have
\begin{align*}
&(f_1\cdot f_2)\cdot\ddot{\Theta}_{\infty,n}(g)
\\
&\qquad=
E_{\ell+m}(\iota_{\ell,m}(f_1\otimes f_2))\cdot\ddot{\Theta}_{\infty,n}(g) \\
&\qquad=
\ddot{\Theta}_{\infty,\ell+m+n}(E_{\ell+m+n}(\iota_{\ell+m,n}(E_{\ell+m}(\iota_{\ell,m}(f_1\otimes f_2))\otimes g))) \\
&\qquad=
\ddot{\Theta}_{\infty,\ell+m+n}(E_{\ell+m+n}(\iota_{\ell,m+n}(f_1\otimes E_{m,n}(\iota_{m,n}(f_2 \otimes g))))) \quad \text{(use Corollary \ref{C3.13})} \\
&\qquad=
f_1\cdot\ddot\Theta_{\infty,m+n}(E_{m,n}(\iota_{m,n}(f_2 \otimes g))) \\
&\qquad=
f_1\cdot\big(f_2\cdot\ddot{\Theta}_{\infty,n}(g)\big).
\end{align*}
Hence \eqref{Eq3.9} indeed gives a left $\Sigma(\alpha^t,\beta)$-module structure on $\ddot{\mathcal{S}}(\alpha^t,\beta)$.

For any $m\geq0$, $\ell>n\geq0$, $k > m+ \ell$, $f \in \mathcal{Z}(A_m)$ and $g \in \mathcal{Z}(A_n)$ one has
\[
\Vert \Theta_{k,m+n}(E_{m+n}(\iota_{m,n}(f\otimes g)))\Vert
=
\Vert \Theta_{k,m+\ell}(E_{m+\ell}(\iota_{m,\ell}(f\otimes\Theta_{\ell,n}(g))))\Vert
\leq
\Vert f\Vert\,\Vert\Theta_{\ell,n}(g)\Vert
\]
by the above computation utilizing Corollary \ref{C3.13}. Then, taking the limit as $k\to\infty$ we have, by \cite[Proposition 4.16]{MawhinneyTodorov:DissertationesMath18},
\[
\Vert \Theta_{\infty,m+n}(E_{m+n}(\iota_{m,n}(f\otimes g)))\Vert
\leq
\Vert f\Vert\,\Vert\Theta_{\ell,n}(g)\Vert
\]
as long as $\ell > n$. Taking the limit as $\ell\to\infty$ we conclude that
\[
\Vert \Theta_{\infty,m+n}(E_{m+n}(\iota_{m,n}(f\otimes g)))\Vert
\leq
\Vert f\Vert\,\Vert\Theta_{\infty,n}(g)\Vert.
\]
This inequality shows that replacing $\ddot{\Theta}_{\infty,n}$ with $\Theta_{\infty,n}$ in \eqref{Eq3.9} still works to define a left $\Sigma(\alpha^t,\beta)$-module structure on $\mathcal{S}(\alpha^t,\beta)$.

The last assertion is trivial from requirement $(m1)$.

(3) Observe that
\[
f\cdot1_1
=
E_{n+1}(\iota_{n+1,1}(f\otimes 1))
= E_n(f)
\]
by requirement $(L)$.
The desired assertion follows from this observation and Remarks \ref{R3.4}(1).
\end{proof}

Here is an observation about the positivity on $\ddot{\mathcal{S}}(\alpha^t,\beta)$ and $\mathcal{S}(\alpha^t,\beta)$ in relation with the left $\Sigma(\alpha^t,\beta)$-module structure established above.

\begin{Proposition}\label{P3.15} The positive cones $\ddot{\mathcal{S}}^+(\alpha^t,\beta)$ on $\ddot{\mathcal{S}}(\alpha^t,\beta)$ and $\mathcal{S}^+(\alpha^t,\beta)$ on $\mathcal{S}(\alpha^t,\beta)$, respectively, enjoy that
\[
\mathcal{Z}(A_m)^+\cdot\ddot{\mathcal{S}}^+(\alpha^t,\beta) \subseteq \ddot{\mathcal{S}}^+(\alpha^t,\beta), \qquad
\mathcal{Z}(A_m)^+\cdot\mathcal{S}^+(\alpha^t,\beta) \subseteq \mathcal{S}^+(\alpha^t,\beta).
\]
for every $m \geq0$, where $\mathcal{Z}(A_m)^+$ denotes the positive elements of $\mathcal{Z}(A_m)$.
\end{Proposition}
\begin{proof}
Let $f \in \mathcal{Z}(A_m)^+$ and $\ddot{s} = \ddot{\Theta}_{\infty,n}(g)$ with $g \in \mathcal{Z}(A_n)$ be given. Assume $\ddot{s} \geq 0$. Then $\Theta_{\ell,n}(g) \geq 0$ in $\mathcal{Z}(A_\ell)$ for some $\ell \geq n$ thanks to the construction of $\ddot{\mathcal{S}}^+(\alpha^t,\beta)$ (see the place just before \cite[Lemma 3.2]{MawhinneyTodorov:DissertationesMath18}). Thus,
\[
f\cdot\ddot{s} = \ddot{\Theta}_{\infty,m+n}(E_{m+\ell}(\iota_{m,\ell}(f\otimes\Theta_{\ell,n}(g)))) \geq 0.
\]
Hence the former has been confirmed.

We then confirm the latter by using the former. Let $s \in \mathcal{S}^+(\alpha^t,\beta)$ be given. Then $s = q_\mathcal{S}(\ddot{s})$ with $\ddot{s}=\ddot{s}^* \in \ddot{\mathcal{S}}(\alpha^t,\beta)$ holds. By the construction of $\mathcal{S}^+(\alpha^t,\beta)$ (see the place just before \cite[Lemma 3.10]{MawhinneyTodorov:DissertationesMath18}) we observe that $\ddot{s} + \delta\ddot{e} \geq 0$ for all $\delta>0$. Since
\[
f\cdot\ddot{e} = \ddot{\Theta}_{\infty,n}(f) \leq \Vert f\Vert\,\ddot{\Theta}_{\infty,n}(1) = \Vert f\Vert\,\ddot{e}
\]
by Theorem \ref{T3.14}(2), we have
\[
0 \leq q_\mathcal{S}(f\cdot(\ddot{s}+r\ddot{e}))
\leq f\cdot s + \Vert f\Vert\,\delta e
\]
for all $\delta > 0$, where $q_\mathcal{S}(f\cdot\ddot{s}) = f\cdot q_\mathcal{S}(\ddot{s}) = f\cdot s$ by definition. This implies that $f\cdot s \geq 0$.
\end{proof}

The remaining question is what condition on the left multiplicative inductive sequence $A_n$ makes both $\ddot{\mathcal{S}}(\alpha^t,\beta)$ and $\mathcal{S}(\alpha^t,\beta)$ unital commutative algebras. We give such a condition below, though it is not satisfied in the case of $\mathrm{U}_q(n)$, our main example. Actually, $\Sigma(\alpha^t,\beta)$ is not commutative in the case. See Proposition \ref{P5.21}.

Let us introduce the following property $(c)$: For each pair $m,n \geq 0$ there are $\gamma_{m+n}^{(m,n)} \in \mathrm{Aut}(A_{m+n})$, $\gamma_m^{(m,n)} \in \mathrm{Aut}(A_m)$ and $\gamma_n^{(m,n)} \in \mathrm{Aut}(A_n)$ such that
\begin{itemize}\itemsep=0pt
\item[$(c1)$] $\gamma_{m+n}^{(m,n)}$, $\gamma_m^{(m,n)}$ and $\gamma_n^{(m,n)}$ trivially act on the centers,
\item[$(c2)$] $\gamma_{m+n}^{(m,n)}(\iota_{m,n}(a\otimes b)) = \iota_{n,m}\big(\gamma_n^{(m,n)}(b)\otimes\gamma_m^{(m,n)}(a)\big)$ holds for every pair $a \in A_m$ and $b \in A_n$,
\item[$(c3)$] $\gamma_{m+n}^{(m,n)}\circ\alpha_{m+n}^t = \alpha_{m+n}^t\circ\gamma_{m+n}^{(m,n)}$ for every $t \in \mathbb{R}$.
\end{itemize}
The motivation for introducing the property $(c)$ is to abstract the cases of infinite symmetric group $\mathfrak{S}_\infty$ as well as infinite-dimensional unitary group $\mathrm{U}(\infty)$ in the present framework. In~those classical cases, each $\gamma_{m+n}^{(m,n)}$ is given by an inner conjugacy and each $\alpha_{m+n}^t$ must be the trivial flow. The next proposition shows that this property is sufficient to make both $\ddot{\mathcal{S}}(\alpha^t,\beta)$ and~$\mathcal{S}(\alpha^t,\beta)$ unital commutative algebras.

\begin{Proposition}\label{P3.16}
If the left multiplicative structure $\iota_{m,n}$ has property $(c)$, then all $\Sigma(\alpha^t,\beta)$, $\ddot{\mathcal{S}}(\alpha^t,\beta)$ and $\mathcal{S}(\alpha^t,\beta)$ are unital commutative algebras with units $1_0$, $\ddot{e}$ and $e$, respectively. Moreover, the mappings $x \in \Sigma(\alpha^t,\beta) \mapsto x\cdot\ddot{e} \in \ddot{\mathcal{S}}(\alpha^t,\beta)$ and $x \in \Sigma(\alpha^t,\beta) \mapsto x\cdot e \in \mathcal{S}(\alpha^t,\beta)$ define surjective algebra homomorphisms, and furthermore, $\mathcal{S}(\alpha^t,\beta)$ is a normed algebra.
\end{Proposition}
\begin{proof}
Let $m,n\geq 0$ be fixed. For simplicity, we write $\gamma_{m+n} := \gamma_{m+n}^{(m,n)}$, $\gamma_{m} := \gamma_{m}^{(m,n)}$ and $\gamma_{n} := \gamma_{n}^{(m,n)}$.

Observe that for each $z \in \mathfrak{Z}_{m+n}$, $\gamma_{m+n}$ induces an $*$-automorphism of $zA_{m+n}$ by $(c1)$. Then, one can confirm, by checking the KMS condition thanks to $(c1)$ and $(c3)$, that $\tau_z\circ\gamma_{m+n} = \tau_z$ holds.

Let $f \in \mathcal{Z}(A_m)$ and $g \in \mathcal{Z}(A_n)$ be arbitrarily given. We have
{\samepage\begin{align*}
E_{m+n}(\iota_{n,m}(g\otimes f))
&=
E_{m+n}(\gamma_{m+n}(\iota_{m,n}(f\otimes g))) \quad \text{(by $(c1)$ and $(c2)$)} \\
&=
\sum_{z\in\mathfrak{Z}_{m+n}} \tau_z(z\,\gamma_{m+n}(\iota_{m,n}(f\otimes g)))z \\
&=
\sum_{z\in\mathfrak{Z}_{m+n}} \tau_z(\gamma_{m+n}(z\,\iota_{m,n}(f\otimes g)))z \quad \text{(by $(c1)$)} \\
&=
\sum_{z\in\mathfrak{Z}_{m+n}} \tau_z(z\,\iota_{m,n}(f\otimes g))z \quad \text{(by the above consideration)}\\
&=
E_{m+n}(\iota_{m,n}(f\otimes g)).
\end{align*}}
This immediately implies that $\Sigma(\alpha^t,\beta)$ is a commutative algebra.
It also follows that
\begin{align*}
g\cdot\ddot{\Theta}_{\infty,m}(f)
&=
\ddot{\Theta}_{\infty,m+n}(E_{m+n}(\iota_{n,m}(g\otimes f))) \\
&= \ddot{\Theta}_{\infty,m+n}(E_{m+n}(\iota_{m,n}(f\otimes g))) = f\cdot\ddot{\Theta}_{\infty,n}(g)
\end{align*}
holds.

We have
\[
\Theta_{m',m}(f)\cdot\ddot{\Theta}_{\infty,n}(g) = g\cdot\ddot{\Theta}_{\infty,m'}(\Theta_{m',m}(f)) = g\cdot\ddot{\Theta}_{\infty,m}(f) = f\cdot\ddot{\Theta}_{\infty,n}(g)
\]
as long as $m' > m$, and applying $q_\mathcal{S}$ to this identity we also obtain that
\[
\Theta_{m',m}(f)\cdot\Theta_{\infty,n}(g) = g\cdot\Theta_{\infty,m}(f) = f\cdot\Theta_{\infty,n}(g)
\]
as long as $m < m'$. In particular,
\[
(\ddot{\Theta}_{\infty,m}(f),\ddot{\Theta}_{\infty,n}(g)) \mapsto f\cdot\ddot{\Theta}_{\infty,n}(g)=g\cdot\ddot{\Theta}_{\infty,m}(f)
\]
defines a well-defined, commutative multiplication on $\ddot{\mathcal{S}}(\alpha^t,\beta)$.

The proof of Theorem \ref{T3.14}(2) shows that
\[
\Vert f\cdot\Theta_{\infty,n}(g)\Vert = \Vert \Theta_{m',m}(f)\cdot\Theta_{\infty,n}(g))\Vert \leq \Vert \Theta_{m',m}(f)\Vert\,\Vert\Theta_{\infty,n}(g)\Vert
\]
as long as $m' > m$. It follows, by taking the limit as $m'\to\infty$, that
\[
\Vert f\cdot\Theta_{\infty,n}(g)\Vert \leq \Vert\Theta_{\infty,m}(f)\Vert\,\Vert\Theta_{\infty,n}(g)\Vert
\]
by \cite[Proposition 4.16]{MawhinneyTodorov:DissertationesMath18}. Hence
\[
(\Theta_{\infty,m}(f),\Theta_{\infty,n}(g)) \mapsto f\cdot\Theta_{\infty,n}(g)=g\cdot\Theta_{\infty,m}(f)
\]
also defines a well-defined, commutative multiplication on $\mathcal{S}(\alpha^t,\beta)$.

That $\ddot{e}$ and $e$ become units of those algebra, respectively, are trivial. Also, the last assertion is clear now.
\end{proof}

\begin{Remark}\label{R3.17}
Propositions \ref{P3.15} and~\ref{P3.16} enable us to prove, by a standard method (see, e.g., \cite[Proposition 4.4 and Exercise 4.2]{BorodinOlshanski:Book}, that the so-called ring theorem holds for $\mathcal{S}(\alpha^t,\beta)$ if the left multiplicative structure $\iota_{m,n}$ has property $(c)$.
\end{Remark}

\subsection[Olshanski's representation ring of U(infty)]
{Olshanski's representation ring of $\boldsymbol{\mathrm{U}(\infty)}$}\label{S3.5}

Let $A_n = W^*(\mathrm{U}(n))$ be the group $W^*$-algebra of the unitary group $\mathrm{U}(n)$ of rank $n$, and consider the case when the flow $\alpha^t$ is the trivial one because no $q$-deformation is necessary and dealing with tracial states is suitable in this case (see Section 4). We will investigate the algebra $\Sigma := \Sigma(\alpha^t,\beta)$, the representation (operator) systems $\ddot{\mathcal{S}} = \ddot{\mathcal{S}}(\alpha^t,\beta)$ and $\mathcal{S}=\mathcal{S}(\alpha^t,\beta)$, all of which must be independent of the flow and the inverse temperature in the case. We also remark that every $E_n$ must be the center-valued trace in the present setup.

The consequence gives an operator algebraic point of view to Olshanski's rings $\mathscr{R}$ and \mbox{$\mathscr{R}/(\varphi-1)$} with the notation in \cite{Olshanski:AdvMath16}. Actually, we will show that $\Sigma$ and $\ddot{\mathcal{S}}$ are exactly the same as $\mathscr{R}$ and $\mathscr{R}/(\varphi-1)$, respectively. In what follows, we will freely use the notation in \cite[Sections 2 and 3.1--3.2]{Olshanski:AdvMath16}, but $n$ stands for the parameter of length of signature instead of $N$.

In the present case, we may and do identify $\mathfrak{Z}_n$ with the set $\mathbb{S}_n$ of all signatures of length~$n$. We will denote by $z_\lambda$ the corresponding minimal central projection for a $\lambda \in \mathbb{S}_n$. Then we have $\mathcal{Z}(A_n) \cong \ell^\infty(\mathbb{S}_n)$ by $z_\lambda \mapsto \delta_\lambda$, where $\delta_\lambda$ is the Dirac function taking $1$ at $\lambda$ and~$0$ elsewhere. By definition we also have $\mathscr{R}_n \cong \ell^\infty(\mathbb{S}_n)$, as normed spaces, by sending each $\psi = \sum_\lambda \psi(\lambda) \sigma_\lambda \in \mathscr{R}_n$ with $\psi(\lambda) \in \mathbb{C}$ to $(\psi(\lambda)/\dim(\lambda))_\lambda \in \ell^\infty(\mathbb{S}_n)$, where $\dim(\lambda)$ denotes the dimension of a~representation with the signature $\lambda$, so that $\dim(\lambda) := \dim(z_\lambda)$. Consequently, we have $\mathscr{R} = \sum_{n\geq 0} \mathscr{R}_n \cong \Sigma := \sum_{n\geq0} \mathcal{Z}(A_n)$ by $\sigma_\lambda \mapsto \hat{z}_\lambda := \dim(\lambda)^{-1}\cdot z_\lambda$ for each $\lambda \in \mathbb{S}_n$, where $\sum_{n\geq0}$ denotes the algebraic direct sum over $n \geq 0$.

We then investigate the multiplication on $\mathscr{R}$. For each $\lambda \in \mathbb{S}_{m+n}$ we have
\begin{equation}\label{Eq3.10}
s_\lambda(t_1,\dots,t_{m+n}) = \sum_{(\mu,\nu) \in \mathbb{S}_m\times\mathbb{S}_n}c(\lambda\mid\mu,\nu)\,s_\mu(t_1,\dots,t_m)\,s_\nu(t_{m+1},\dots,t_{m+n})
\end{equation}
with rational Schur functions $s_\lambda$. This indeed describes the expansion of the irreducible character of $\mathrm{U}(m+n)$ of label $\lambda$, whose restriction to the maximal torus is $s_\lambda$, restricted to $\mathrm{U}(m)\times\mathrm{U}(n)$, via the embedding
\begin{equation}\label{Eq3.11}
(g_1,g_2) \in \mathrm{U}(m)\times\mathrm{U}(n) \mapsto \begin{bmatrix} g_1 & \\ & g_2 \end{bmatrix} \in \mathrm{U}(m+n).
\end{equation}
With this quantity $c(\lambda\mid\mu,\nu)$, the multiplication on $\mathscr{R}$ is determined by
\begin{equation}\label{Eq3.12}
\sigma_\mu \sigma_\nu = \sum_{\lambda \in \mathbb{S}_{m+n}}c(\lambda\mid\mu,\nu)\,\sigma_\lambda.
\end{equation}
See \cite[equation (2.13)]{Olshanski:AdvMath16}.

It is time to specify the left multiplicative structure on the sequence $A_n = W^*(\mathrm{U}(n))$ that we are going to discuss. The embedding \eqref{Eq3.11} gives a left multiplicative structure $\iota_{m,n} \colon A_m\,\bar{\otimes}\,A_n \allowbreak\to A_{n+n}$. Note that the embedding is conjugate in $\mathrm{U}(m+n)$ to
\[
(g_1,g_2) \in \mathrm{U}(m)\times\mathrm{U}(n) \mapsto \begin{bmatrix} g_2 & \\ & g_1 \end{bmatrix} \in \mathrm{U}(m+n).
\]
This fact implies that the structure $\iota_{m,n}$ enjoys property $(c)$ since the $\alpha^t$ is assumed to be the trivial flow. Consequently, the resulting $\Sigma$, $\ddot{\mathcal{S}}$ and $\mathcal{S}$ are all unital commutative algebras by Proposition \ref{P3.16}.
We will look at the structure $\iota_{m,n}$ in some detail below.

For each $\lambda \in \mathbb{S}_n$, let $\pi_\lambda \colon \mathrm{U}(n) \curvearrowright \mathcal{H}_\lambda$ be the corresponding irreducible representation. Then $W^*(\mathrm{U}(n))$ is explicitly obtained as the $W^*$-algebra generated by $\bigoplus_{\lambda\in\mathbb{S}_n} \pi_\lambda(g)$, $g \in \mathrm{U}(n)$, on the Hilbert space $\bigoplus_{\lambda\in\mathbb{S}_n}\mathcal{H}_\lambda$. Thus, we may identify $\mathcal{H}_{z_\lambda} = \mathcal{H}_\lambda$, and $x \in W^*(\mathrm{U}(n)) \mapsto z_\lambda x \in z_\lambda W^*(\mathrm{U}(n)) = B(\mathcal{H}_\lambda)$ corresponds to the irreducible representation of label $\lambda$. It is natural to define $\iota_{m,n}$ by \eqref{Eq3.10} in such a way that
\[
z_\lambda \iota_{m,n}(x\otimes y)\ = \bigoplus_{(\mu,\nu)\in\mathbb{S}_m\times\mathbb{S}_n} (z_\mu x\otimes z_\nu y)^{\oplus c(\lambda\mid\mu,\nu)} \quad \text{(up to unitary equivalence)}.
\]
It follows that
\[
E_{m+n}(\iota_{m,n}(z_\mu \otimes z_\nu)) = \sum_{\lambda\in\mathbb{S}_{m+n}} c(\lambda\mid\mu,\nu)\frac{\dim(\mu)\cdot\dim(\nu)}{\dim(\lambda)}\,z_\lambda,
\]
implying that
\begin{equation}\label{Eq3.13}
(\hat{z}_\mu,\hat{z}_\nu) \mapsto E_{m+n}(\iota_{m,n}(\hat{z}_\mu \otimes\hat{z}_\nu)) = \sum_{\lambda\in\mathbb{S}_{m+n}} c(\lambda\mid\mu,\nu)\,\hat{z}_\lambda.
\end{equation}
Comparing \eqref{Eq3.12} with \eqref{Eq3.13} we observe that the vector space isomorphism $\mathscr{R} \cong \Sigma$ transplants the multiplication on $\mathscr{R}$ into that on $\Sigma$ determined by
\[
(x,y) \in \mathcal{Z}(A_m)\times\mathcal{Z}(A_n) \mapsto E_{m+n}(\iota_{m,n}(x\otimes y)) \in \mathcal{Z}(A_{m+n}),
\]
which is identical to the multiplication on $\Sigma$ introduced in Theorem \ref{T3.14}(1). Thus we have $\mathscr{R} \cong \Sigma$ as algebras.

Since $\varphi_k = \sigma_{(k)}$ by \cite[equation (2.9)]{Olshanski:AdvMath16} and $\dim((k))\equiv1$, we have $\varphi := \sum_{k\in\mathbb{Z}}\varphi_k \in \mathscr{R}_1 \subset \mathscr{R}$ and the isomorphism $\mathscr{R} \cong \Sigma$ sends $\varphi$ to $1_1$ and the principal ideal $(\varphi-1)$ generated by $\varphi-1$ to $\Sigma\cdot(1_1-1_0)$. Namely, the diagram
\[
\xymatrix{
\mathscr{R} \ar@{->>}[d] \ar@{->}[r]^{\cong} & \Sigma \ar@{->>}[d] \\
\mathscr{R}/(\varphi-1) \ar@{->>}[r]_{\cong} \ar@{->}[dd] & \ddot{\mathcal{S}} \ar@{->>}[d]^{q_\mathcal{S}} \ar@{}[ul]|{\circlearrowright} \\
& \mathcal{S} \ar@{->}[d]^{\Gamma\qquad\quad} \\
C(\widehat{\mathrm{U}(\infty)}) & C\big(\mathrm{ex}\big(S^\mathrm{ln}(\mathcal{S})\big)\big)
}
\]
commutes. Here $\Gamma \colon \mathcal{S} \to C\big(\mathrm{ex}\big(S^\mathrm{ln}(\mathcal{S})\big)\big)$ is defined by $\Gamma(s)(\omega) := \omega(s)$ for any $s \in \mathcal{S}$ and $\omega \in \mathrm{ex}\big(S^\mathrm{ln}(\mathcal{S})\big)$. (This map is multiplicative by Remark \ref{R3.17}.) This shows that the representation system $\ddot{\mathcal{S}}(\alpha^t,\beta)$ is a correct counterpart of representation ring. In this respect, we see that $\ddot{\mathcal{S}} \cong \mathcal{S}$, or other words, $q_\mathcal{S}$ is injective, if we take \cite[Proposition~3.9(iii)]{Olshanski:AdvMath16} into accounts with the help of $\widehat{\mathrm{U}(\infty)} \cong \mathrm{ex}\big(S^{\mathrm{ln}}(\mathcal{S})\big)$ thanks to Theorem~\ref{T3.11}. However, the mechanism behind this phenomenon is not so clear to us at the moment.

\begin{Remark}[{the infinite symmetric group $\mathfrak{S}_\infty = \varinjlim\mathfrak{S}_n$}]
The framework we have developed so far can also be used for the group $C^*$-algebra $C^*(\mathfrak{S}_\infty)$, which is the special case when $A_n = C^*(\mathfrak{S}_n) = \mathbb{C}\mathfrak{S}_n$, finite-dimensional $*$-algebras. Consider the tracial states on $C^*(\mathcal{S}_\infty)$ so that $\alpha^t$ must be trivial and all the $E_n$ become center-valued traces. By Remark \ref{R3.2}, $\ddot{\mathcal{S}} \cong K_0(C^*(\mathfrak{S}_\infty))\otimes_\mathbb{Z}\mathbb{C}$, and hence the study of $\ddot{\mathcal{S}}$ (or $\mathcal{S}$) is exactly the same as that of the dimension group of $C^*(\mathfrak{S}_\infty)$ due to Vershik--Kerov.
\end{Remark}

Consequently, the present framework generalizes both Vershik--Kerov's use of dimension groups as well as Olshanski's representation ring of $\mathrm{U}(\infty)$ in a unified way. Moreover, we will explain, at the end of this paper, how the framework is applied to the quantum group case.


\section{Quantum group setup} 

In this section, we apply the theory we have developed so far to inductive sequences of compact quantum groups. In fact, the notion of $(\alpha^t,\beta)$-spherical representations originally comes from a~consideration of quantum groups. A short comment on this was already given at \cite[Remark~9.4]{Ueda:Preprint20}.

\subsection{Compact quantum groups} \label{S5.1}

We first consider single compact quantum groups to fix the notation and to collect necessary facts on them.
To the best of our knowledge, no analytical attempts have been made to spherical unitary representations (rather than spherical functions) for (Gelfand) pairs $G < G\times G$ (via the diagonal embedding) even in the compact quantum group case.
The discussions on compact quantum groups below entirely follow Neshveyev--Tuset's book \cite[Chapters 1 and 2]{NeshveyevTuset:Book13} (a standard reference on the operator algebra side), but we sometimes use slightly different symbols, such as ``$\mathrm{id}$'' (instead of ``$\iota$'' there) denoting the identity map, from their book in order to keep the notation so far.

Let $G$ be a compact quantum group, which means a Hopf $*$-algebra $(\mathbb{C}[G],\Delta,S,\varepsilon)$ that satisfies the assumption of \cite[Theorem 1.6.7]{NeshveyevTuset:Book13}. Then, we have the ``dual Hopf $*$-algebra'' $\big(\mathcal{U}(G),\hat{\Delta},\hat{S},\hat{\varepsilon}\big)$, whose precise definition is given in the final paragraph of \cite[Section~1.6]{NeshveyevTuset:Book13}. Especially, one has to notice that the comultiplication $\hat{\Delta}$ is a $*$-homomorphism into $\mathcal{U}(G\times G)$ rather than $\mathcal{U}(G)\otimes\mathcal{U}(G)$, where $\mathcal{U}(G\times G)$ is the algebraic dual of the algebraic tensor product $\mathbb{C}[G]\otimes\mathbb{C}[G]$. The general principle of Hopf algebra duality suggests that $\mathcal{U}(G)$ plays a r\^{o}le of group algebra of $G$, and the ``group $W^*$-algebra $W^*(G)$'' should be defined inside $\mathcal{U}(G)$. A naive intuition is that $\hat{\Delta}(\mathcal{U}(G)) \subset \mathcal{U}(G\times G)$ is regarded as a quantum group analogue of $G < G\times G$. However, this pair of $*$-algebras is too big to discuss its unitary representation theory, and hence we develop the theory in terms of $W^*(G)$. We will clarify what $W^*(G)$ is by supplying a few facts, because \cite{NeshveyevTuset:Book13} does not touch it explicitly. The explanation below explicitly or implicitly uses Yamagami's idea \cite{Yamagami:CMP95}.

Any unitary representation $U \in B(\mathcal{H}_U)\otimes \mathbb{C}[G]$ on a finite-dimensional Hilbert space $\mathcal{H}_U$ gives a $*$-representation $\pi_U \colon \mathcal{U}(G) \curvearrowright \mathcal{H}_U$ by $\pi_U(x) := (\mathrm{id}\otimes x)(U)$ for any $x \in \mathcal{U}(G)$. Choose and fix representatives $U_z$, $z \in \mathfrak{Z}$, of the equivalence classes of irreducible unitary representations of $G$ throughout. The $U_z$, $z \in \mathfrak{Z}$, enables us to construct a bijective $*$-homomorphism
\begin{equation}\label{Eq5.1}
\Phi_{U_\bullet} \colon \quad \mathcal{U}(G) \stackrel{\cong}{\longrightarrow} \prod_{z\in\mathfrak{Z}} B(\mathcal{H}_{U_z}) \qquad \text{by}\quad x \mapsto (\pi_{U_z}(x))_{z\in\mathfrak{Z}}.
\end{equation}
One of the $U_z$ must be the trivial representation, and we denote by $\mathbbm{1}$ the corresponding $z$.

Here is the definition of group $W^*$-algebra $W^*(G)$ (and its natural $\sigma$-weakly dense $*$-subal\-gebra~$\mathcal{F}(G)$). Our spherical unitary representation theory will be constructed based on $W^*(G)$ rather than $\mathbb{C}[G]$.

\begin{Definition}[\cite{Yamagami:CMP95}]
The \emph{group $W^*$-algebra} $W^*(G)$ is defined to be all the $x \in \mathcal{U}(G)$ such that $\sup_{z\in\mathfrak{Z}} \Vert\pi_{U_z}(x)\Vert < +\infty$. Let $\mathcal{F}(G)$ be the $*$-subalgebra consisting of all $x \in \mathcal{U}(G)$ such that $\pi_{U_z}(x) = 0$ for all but finitely many $z$. The \emph{group $C^*$-algebra} $C^*(G)$ is defined to be the norm-closure of $\mathcal{F}(G)$.
\end{Definition}

Both $\mathcal{F}(G) \subset W^*(G)$ do not depend on the choice of representatives $U_z$, though the above definition itself does. Moreover, one easily sees that $x \in W^*(G)$ if and only if there exists a~$C > 0$ so that $\Vert \pi_U(x)\Vert \leq C$ for any finite-dimensional unitary representation $U$. Observe that
\begin{equation*}
W^*(G) \cong \bigoplus_{z\in\mathfrak{Z}} B(\mathcal{H}_{U_z})\qquad
\text{in the\ $W^*$-algebra\ or\ $\ell^\infty$-sense}
\end{equation*}
by the bijective $*$-homomorphism \eqref{Eq5.1}, and hence $W^*(G)$ is an atomic finite $W^*$-algebra. Similarly, $\mathcal{F}(G)$ is isomorphic to the algebraic direct sum $\sum_{z \in \mathfrak{Z}} B(\mathcal{H}_{U_z})$ by \eqref{Eq5.1} and becomes a $\sigma$-weakly dense (non-unital) $*$-subalgebra of $W^*(G)$. In what follows, we denote by~$z$ the central projection supporting the direct summand corresponding to each $z\in\mathfrak{Z}$ so that $z\,\mathcal{U}(G)=zW^*(G) =z\mathcal{F}(G) \cong B(\mathcal{H}_z)$ by \eqref{Eq5.1}.

Every $*$-representation $\pi_{U}\otimes\pi_{V} \colon \mathcal{U}(G)\otimes\mathcal{U}(G) \curvearrowright \mathcal{H}_{U}\otimes\mathcal{H}_{V}$ naturally extends to $\mathcal{U}(G\times G)$ by $x \mapsto (\mathrm{id}\otimes x)(U_{13}V_{24})$ with the leg notation ({\it n.b.}, $U_{13}V_{24}$ is the ``outer tensor product'' of $U$ and $V$). As before, we have a bijective $*$-homomorphism
\begin{gather*}
\Phi_{U_\bullet}^{(2)} \colon \ \ \mathcal{U}(G\times G) \stackrel{\cong}{\longrightarrow} \!\!\! \prod_{(z_1,z_2)\in\mathfrak{Z}\times\mathfrak{Z}}\!\!\!\! B\big(\mathcal{H}_{U_{z_1}}\otimes \mathcal{H}_{U_{z_2}}\big) \qquad \text{by}\quad x \mapsto \big((\pi_{U_{z_1}}\otimes\pi_{U_{z_2}})(x)\big)_{(z_1,z_2)\in\mathfrak{Z}\times\mathfrak{Z}}.
\end{gather*}
The tensor product $W^*$-algebra $W^*(G)\,\bar{\otimes}\,W^*(G)$ is $*$-isomorphic to
\begin{equation}\label{Eq5.5}
\bigoplus_{(z_1,z_2)\in\mathfrak{Z}\times\mathfrak{Z}} B\big(\mathcal{H}_{U_{z_1}}\otimes\mathcal{H}_{U_{z_2}}\big) \qquad
\text{in the\ $W^*$-algebra\ or\ $\ell^\infty$-sense},
\end{equation}
and identified with all the $x \in \mathcal{U}(G\times G)$ with $\sup_{(z_1,z_2)\in\mathfrak{Z}\times\mathfrak{Z}}\Vert (\pi_{U_{z_1}}\otimes\pi_{U_{z_2}}(x)\Vert < +\infty$. We denote by $W^*(G\times G)$ this $W^*$-subalgebra of $\mathcal{U}(G\times G)$.

Using
\begin{equation}\label{Eq5.6}
(\pi_{U}\otimes\pi_{V})\circ\hat{\Delta}(x)\,T = T\,\pi_{W}(x), \qquad x \in \mathcal{U}(G)\quad\text{and}\quad T \in \mathrm{Mor}(W, U\times V)
\end{equation}
(see the final paragraph of \cite[Section 1.6]{NeshveyevTuset:Book13}) one sees that
\begin{equation}\label{Eq5.7}
(\pi_{U}\otimes\pi_{V})\circ\hat{\Delta}(x)
=
\sum_{(T,W)} T\,\pi_W(x)\,T^*\ \bigg({\cong} \bigoplus_W \pi_W(x)\bigg), \qquad x \in \mathcal{U}(G),
\end{equation}
where the $T \in \mathrm{Mor}(W,U\times V)$ are selected by the (unique) irreducible decomposition $U\times V \cong \bigoplus_W W$. (See \cite{Yamagami:CMP95} too in this respect.) Hence, the restriction of $\hat{\Delta}$ to $W^*(G)$ defines a normal comultiplication $\hat{\Delta}\colon W^*(G) \to W^*(G\times G)$. We also remark
that
\begin{equation*}
\pi_z(\hat{\varepsilon}(x)) =
\begin{cases}
x(1) = \pi_z(x), & z=\mathbbm{1}, \\
0, &\text{otherwise}.
\end{cases}
\end{equation*}
It immediately follows
that the restriction of $\hat{\varepsilon}$ to $W^*(G)$ gives a normal $*$-character.

There is a special positive element $\rho \in \mathcal{U}(G)$ so that
\begin{equation}\label{Eq5.9}
\hat{\Delta}(\rho) = \rho\otimes\rho, \qquad \hat{S}(\rho) = \rho^{-1}
\end{equation}
and $\hat{R}(x) := \rho^{-1/2}\hat{S}(x)\rho^{1/2}$ for all $x \in \mathcal{U}(G)$ defines the unitary antipode on $\mathcal{U}(G)$. See \cite[Section 1.7]{NeshveyevTuset:Book13} for these facts and the interpretation of functional calculus $\rho^{\zeta} = \exp(\zeta \log\rho)$, \mbox{$\zeta \in \mathbb{C}$}. We have an action $\vartheta^{\zeta} := \mathrm{Ad}\rho^{{\rm i}\zeta} \colon \mathbb{C} \curvearrowright \mathcal{U}(G)$. Since $\rho^{{\rm i}t}$ falls in $W^*(G)$, it induces a flow $\vartheta^t \colon \mathbb{R} \curvearrowright W^*(G)$, which is clearly continuous in the $u$-topology (see, e.g., \cite[Lemma 7.1]{Ueda:Preprint20} for the topology) and fixes elements of the center $\mathcal{Z}(W^*(G))$. Since the second formula in \eqref{Eq5.9} leads~to
\begin{equation}\label{Eq5.10}
\hat{R}(\rho) = \rho^{-1/2}\hat{S}(\rho)\rho^{1/2} = \rho^{-1},
\end{equation}
we have
\begin{equation}\label{Eq5.11}
\hat{R}\circ\vartheta^t = \vartheta^t\circ\hat{R}
\end{equation}
for all $t \in \mathbb{R}$. It follows that
\begin{equation}\label{Eq5.12}
\hat{S}(x) = \vartheta^{-{\rm i}/2}\circ\hat{R}(x) = \hat{R}\circ\vartheta^{-{\rm i}/2}(x), \qquad x \in \mathcal{U}(G).
\end{equation}

Using \cite[Example 2.2.22]{NeshveyevTuset:Book13} we see that
$\hat{R}$ defines an involutive $*$-anti-automorphism on $W^*(G)$ and that $\hat{S}(\mathcal{F}(G)) = \mathcal{F}(G)$ (but $\hat{S}(W^*(G)) \subseteq W^*(G)$ does not hold in general). Moreover, both~$\hat{S}$ and $\hat{R}$ send each direct summand $z\,\mathcal{U}(G) = zW^*(G) \cong B(\mathcal{H}_{U_z})$ to the one $\bar{z}\,\mathcal{U}(G)=\bar{z}W^*(G)\cong B(\mathcal{H}_{U_{\bar{z}}})$ by \eqref{Eq5.1}, where $\bar{z} \in \mathfrak{Z}$ is a unique element such that $U_{\bar{z}}$ is (unitarily) equivalent to the conjugate representation $\bar{U}_z$ of $U_z$ (see \cite[Definition~1.4.5]{NeshveyevTuset:Book13}). In this way, we have an involutive bijection $z \mapsto \bar{z}$ on $\mathfrak{Z}$ so that
\begin{equation}\label{Eq5.14}
\hat{S}(\bar{z}) = \hat{R}(\bar{z}) = z 
\end{equation}
holds for any $z \in \mathfrak{Z}$ as a minimal central projection of $W^*(G)$.

Consequently, we have obtained a sextuplet $\big(W^*(G),\hat{\Delta},\hat{R},\rho^{{\rm i}t},\vartheta^t,\hat{\varepsilon}\big)$ that enjoys all the pro\-per\-ties of \cite[Definition 1.2]{Yamagami:CMP95} with encoding $M := W^*(G)$, $\varDelta := \hat{\Delta}$, $\tau := \hat{R}$, $u_t := \rho^{{\rm i}t/2}$, $\theta_t := \vartheta^{t/2}$, $\varepsilon := \hat{\varepsilon}$ and $\mathscr{M} := \mathcal{F}(G)$ in the notation there. As remarked there, this is a specialized (and even strengthened) version of Masuda--Nakagami's formalism \cite{MasudaNakagami:PRIMS94} so that there are no (essential) differences between Sato's understanding of compact quantum groups in \cite{Sato:JFA19,Sato:preprint19,Sato:ETDS2x} and ours. Thus, his works are available here and so is the present work in his framework.

Here is a technical lemma.

\begin{Lemma}\label{L5.2} Let $\Pi \colon W^*(G)\otimes W^*(G) \curvearrowright \mathcal{H}_\Pi$ be a bi-normal $*$-representation. Then there are a unique bi-normal $*$-representation $\overline{\Pi} \colon W^*(G)\otimes_{\max}W^*(G) \curvearrowright \mathcal{H}_\Pi$ and a unique normal $*$-representations $\underline{\Pi} \colon W^*(G\times G) \curvearrowright \mathcal{H}_\Pi$ such that
the diagram
\[
\xymatrix{
W^*(G)\otimes_{\max}W^*(G) \ar[rd]^{\overline{\Pi}} & \\
W^*(G)\otimes W^*(G) \ar@{^{(}-_>}[u] \ar@{^{(}-_>}[d] \ar[r]^{\quad \Pi} & B(\mathcal{H}_\Pi) \\
W^*(G\times G) \ar[ru]_{\underline{\Pi}} &
}
\]
commutes, where the vertical arrows mean the natural embeddings thanks to $W^*(G\times G) = W^*(G)\,\bar{\otimes}\,W^*(G)$ for the lower one.
\end{Lemma}
\begin{proof} The existence of desired $\overline{\Pi}$ is trivial by the construction of maximal $C^*$-tensor products.

Observe that the family $z_1\otimes z_2$ with $z_1,z_2 \in \mathfrak{Z}$ is a complete family of minimal projections of $\mathcal{Z}(W^*(G\times G))$ thanks to \eqref{Eq5.5}. Since $zW^*(G) \cong B(\mathcal{H}_z)$ with $\dim(\mathcal{H}_z) < +\infty$ for every $z \in \mathfrak{Z}$, we have
\[
W^*(G\times G) \twoheadrightarrow (z_1\otimes z_2)W^*(G\times G)
= (z_1\otimes z_2)(W^*(G)\otimes W^*(G))
\]
for every $z_1, z_2 \in \mathfrak{Z}$. Hence we have a normal $*$-representation ${}_{z_1}\Pi_{z_2} \colon W^*(G\times G) \curvearrowright {}_{z_1}(\mathcal{H}_\Pi)_{z_2} := \Pi(z_1\otimes z_2)\mathcal{H}_\Pi$ defined by ${}_{z_1}\Pi_{z_2}(x) := \Pi((z_1\otimes z_2)x)$ for all $x \in W^*(G\times G)$.

Since $\Pi$ is bi-normal, $\mathcal{H}_\Pi = \bigoplus_{(z_1,z_2)\in\mathfrak{Z}\times\mathfrak{Z}} {}_{z_1}(\mathcal{H}_\Pi)_{z_2}$. Consequently, we obtain a normal $*$-representation
\[
\big(\underline{\Pi} \colon W^*(G\times G) \curvearrowright \mathcal{H}_\Pi\big) := \bigoplus_{(z_1,z_2)\in\mathfrak{Z}\times\mathfrak{Z}} \big({}_{z_1}\Pi_{z_2} \colon W^*(G\times G) \curvearrowright {}_{z_1}(\mathcal{H}_\Pi)_{z_2}\big).
\]
By construction, $\underline{\Pi}$ clearly agrees with $\Pi$ on $W^*(G)\otimes W^*(G)$ thanks to the bi-normality of $\Pi$. Hence we are done.
\end{proof}

In what follows, we call a bi-normal $*$-representation $\Pi \colon W^*(G)\otimes W^*(G) \curvearrowright \mathcal{H}_\Pi$ \emph{a unitary representation of $G\times G$}. Then we have a unique bi-normal $*$-representations $\overline{\Pi} \colon W^*(G)\otimes_{\max}W^*(G) \curvearrowright \mathcal{H}_\Pi$ and a unique normal $*$-representations $\underline{\Pi} \colon W^*(G\times G) \curvearrowright \mathcal{H}_\Pi$ as in the above lemma.

\begin{Definition}\label{D5.3} A unitary representation $\Pi$ of $G\times G$ on a Hilbert space $\mathcal{H}_\Pi$ with a unit vector $\xi \in \mathcal{H}_\Pi$ is called a \emph{spherical unitary representation} of $G < G\times G$, if $\xi$ is cyclic for $\underline{\Pi}(W^*(G\times G))$ and $\underline{\Pi}\big(\hat{\Delta}(x)\big)\xi=\hat{\varepsilon}(x)\xi$ holds for every $x \in W^*(G)$. We call such a vector a \emph{spherical vector} in~$\mathcal{H}_\Pi$ of $G < G\times G$ through $\Pi$.
\end{Definition}

The next proposition, which is motivated by \cite[Proposition 3.4]{KolbStokman:SelMath09}, shows that the notion of spherical unitary representations of $G < G\times G$ is exactly equivalent to that of $(\vartheta^t,-1)$-spherical representations of $W^*(G)$.

\begin{Proposition}\label{P5.4} For a unitary representation $\Pi$ of $G\times G$ on $\mathcal{H}_\Pi$ and a vector $\xi \in \mathcal{H}_\Pi$, the following are equivalent:
\begin{itemize}\itemsep=0pt
\item[$(i)$] $\underline{\Pi}\big(\hat{\Delta}(x)\big)\xi = \hat{\varepsilon}(x)\xi$ for all $x \in W^*(G)$.
\item[$(ii)$] $\Pi\big(\hat{S}(x)\otimes1\big)\xi=\Pi(1\otimes x)\xi$ for all $x \in \mathcal{F}(G)$.
\item[$(iii)$] $\xi$ is a $(\vartheta^t,-1)$-spherical vector in the bi-normal $*$-representation $\overline{\Pi}^{\mathrm{rop}}$ of $W^*(G)\otimes_\mathrm{max} W^*(G)^\mathrm{op}$ on $\mathcal{H}_\Pi$ defined to be
\[
\begin{matrix}
W^*(G)\otimes_\mathrm{max} W^*(G)^\mathrm{op}
& \overset{\mathrm{id}\,\otimes\,\hat{R}}{\to} &
W^*(G)\otimes_\mathrm{max} W^*(G)
&\overset{\overline{\Pi}}{\to} &
B(\mathcal{H}_\Pi), \\
x\otimes y^\mathrm{op}
& \mapsto &
x\otimes\hat{R}(y)
& \mapsto &
\overline{\Pi}\big(x\otimes\hat{R}(y)\big).
\end{matrix}
\]
See {\rm \cite[Definition~5.1, Remark and Definition~5.2]{Ueda:Preprint20}} for the notion of $(\vartheta^t,-1)$-spherical vectors.
\end{itemize}
Moreover, under the above equivalent conditions, $\Pi(z_1\otimes z_2)\xi$ must be zero if $z_1 \neq \bar{z}_2$.
\end{Proposition}
\begin{proof}
 $(i)\Rightarrow (ii)$: Let $x \in \mathcal{F}(G)$ be arbitrarily given. Then we have
\begin{align*}
\big(m\circ\big(\hat{S}\otimes\mathrm{id}\big)\otimes\mathrm{id}\big)((\hat{\Delta}\otimes\mathrm{id})(\hat{\Delta}(x))))(a\otimes b)
&=
\big(\big(\hat{\Delta}\otimes\mathrm{id}\big)\big(\hat{\Delta}(x)\big)\big)((S\otimes\mathrm{id})(\Delta(a))\otimes b) \\
&=
\hat{\Delta}(x)(m((S\otimes\mathrm{id})(\Delta(a)))\otimes b) \\
&=
x(m((S\otimes\mathrm{id})(\Delta(a)))b) \\
&=
\varepsilon(a)x(b) = (1\otimes x)(a\otimes b)
\end{align*}
for all $a,b \in \mathbb{C}[G]$ with the multiplication map $m$. Thus,
we have
\begin{align*}
1\otimes x&=(m\otimes\mathrm{id})\big(\big(\hat{S}\otimes\mathrm{id}\otimes\mathrm{id}\big) \big(\big(\hat{\Delta}\otimes\mathrm{id}\big)\big(\hat{\Delta}(x)\big)\big)\big)
\\
&=(m\otimes\mathrm{id})\big(\big(\hat{S}\otimes\mathrm{id}\otimes\mathrm{id}\big) \big(\big(\mathrm{id}\otimes\hat{\Delta}\big)\big(\hat{\Delta}(x)\big)\big)\big).
\end{align*}
Observe that
\begin{align*}
&(z_1\otimes z_2)(m\otimes\mathrm{id})\big(\big(\hat{S}\otimes\mathrm{id}\otimes\mathrm{id}\big)\big(\big(\mathrm{id}\otimes\hat{\Delta}\big)\big(\hat{\Delta}(x)\big)\big)\big) \\
&\qquad=
(z_1\otimes z_2)(m\otimes\mathrm{id})\big(\big( \hat{S}\otimes\mathrm{id} \otimes\mathrm{id}\big)\big(\big(\mathrm{id}\otimes\hat{\Delta}\big) \big((\bar{z}_1\otimes1)\hat{\Delta}(x)\big)\big)\big)
\end{align*}
holds for any $z_1, z_2 \in \mathfrak{Z}$ thanks to \eqref{Eq5.14}. Since $(\bar{z}_1\otimes1)\hat{\Delta}(x) \in (\bar{z}_1\otimes1)W^*(G\times G) = (\bar{z}_1W^*(G))\otimes W^*(G) \cong B(\mathcal{H}_{\bar{z}_1})\otimes W^*(G)$ and $\dim(\bar{z}_1) < +\infty$, we can justify Sweedler's notation $(\bar{z}_1\otimes1)\hat{\Delta}(x) = x_{(1)}\otimes x_{(2)}$ in $\mathcal{F}(G)\otimes W^*(G)$. Then we have
\begin{align*}
\Pi(1\otimes x)\xi
&=
\sum_{(z_1,z_2) \in \mathfrak{Z}\times\mathfrak{Z}}
\Pi((z_1\otimes z_2)(1\otimes x))\xi \\
&=
\sum_{(z_1,z_2) \in \mathfrak{Z}\times\mathfrak{Z}}
\Pi\big((z_1\otimes z_2)(m\otimes\mathrm{id})\big(\big(\hat{S}\otimes\mathrm{id}\otimes\mathrm{id}\big)\big(\big(\mathrm{id}\otimes\hat{\Delta}\big) \big((\bar{z}_1\otimes1)\hat{\Delta}(x)\big)\big)\big)\big)\xi \\
&=
\sum_{(z_1,z_2) \in \mathfrak{Z}\times\mathfrak{Z}}\Pi\big((z_1\otimes z_2)\big(\hat{S}(x_{(1)})\otimes1\big)\hat{\Delta}(x_{(2)})\big)\xi \\
&=
\sum_{(z_1,z_2) \in \mathfrak{Z}\times\mathfrak{Z}}\Pi \big((z_1\otimes z_2)\big(\hat{S}(x_{(1)})\otimes1\big)\big)\underline{\Pi}\big(\hat{\Delta}(x_{(2)})\big) \xi \\
&=
\sum_{(z_1,z_2) \in \mathfrak{Z}\times\mathfrak{Z}}\Pi\big((z_1\otimes z_2)\big(\hat{S}(\hat{\varepsilon}(x_{(2)})x_{(1)})\otimes1\big)\big) \xi \quad \text{(by assumption)} \\
&=
\sum_{(z_1,z_2) \in \mathfrak{Z}\times\mathfrak{Z}}\Pi\big((z_1\otimes z_2)\big(\hat{S}\big((\mathrm{id}\otimes\hat{\varepsilon})\big((\bar{z}_1\otimes1)\hat{\Delta}(x)\big)\big)\otimes1\big)\big)\xi \\
&=
\sum_{(z_1,z_2) \in \mathfrak{Z}\times\mathfrak{Z}}\Pi\big((z_1\otimes z_2)\big(\hat{S}(x)\otimes1\big)\big)\xi \quad (\text{by Hopf $*$-algebra structure and \eqref{Eq5.12}}) \\
&=
\Pi\big(\hat{S}(x)\otimes1\big)\xi,
\end{align*}
where the final equality is justified thanks to $\hat{S}(x) \in \mathcal{F}(G) \subset W^*(G)$ due to the assumption that $x \in \mathcal{F}(G)$.

$(ii)\Rightarrow (i)$: Let $x \in W^*(G)$ be arbitrarily chosen. For each pair $(z_1,z_2) \in \mathfrak{Z}\times\mathfrak{Z}$, we have $(z_1\otimes z_2)\hat{\Delta}(x) = x_{(1)}\otimes x_{(2)}$ in $\mathcal{F}(G)\otimes\mathcal{F}(G)$ as above. Then we have
\begin{align*}
\Pi\big((z_1\otimes z_2)\hat{\Delta}(x)\big)\xi
&=
\Pi((z_1\otimes z_2)(x_{(1)}\otimes x_{(2)}))\xi \\
&=
\Pi((z_1\otimes z_2)(x_{(1)}\otimes1))\Pi(1\otimes x_{(2)})\xi \\
&=
\Pi((z_1\otimes z_2)(x_{(1)}\otimes1))\Pi\big(\hat{S}(x_{(2)})\otimes1\big)\xi \quad \text{(by assumption)}\\
&=
\Pi\big((z_1\otimes z_2)\big(\big(x_{(1)}\hat{S}(x_{(2)})\big)\otimes1\big)\big)\xi.
\end{align*}
Observe that
\[
x_{(1)}\hat{S}(x_{(2)}) = m\big(\big(\mathrm{id}\otimes\hat{S}\big)\big((z_1\otimes z_2)\hat{\Delta}(x)\big)\big)
= m\big((z_1\otimes\bar{z}_2)\big(\mathrm{id}\otimes\hat{S}\big)\big(\hat{\Delta}(x)\big)\big)
= \delta_{z_1,\bar{z}_2} \hat{\varepsilon}(x) \bar{z}_2
\]
by using \eqref{Eq5.14}. Hence we obtain that
\[
\Pi\big((z_1\otimes z_2)\hat{\Delta}(x)\big)\xi = \delta_{z_1,\bar{z}_2} \hat{\varepsilon}(x)\Pi(\bar{z}_2\otimes z_2)\xi.
\]
Substituting the unit $1$ for $x$ in the above identity implies $\Pi(z_1\otimes z_2)\xi = 0$ if $z_1 \neq \bar{z}_2$. Thus, we obtain that $\xi = \sum_{z \in \mathfrak{Z}} \Pi(\bar{z}\otimes z)\xi$, and consequently,
\[
\underline{\Pi}\big(\hat{\Delta}(x)\big)\xi
=
\sum_{(z_1,z_2) \in \mathfrak{Z}\times\mathfrak{Z}} \Pi\big((z_1\otimes z_2)\hat{\Delta}(x)\big)\xi
=
\hat{\varepsilon}(x) \sum_{z \in \mathfrak{Z}} \Pi(\bar{z}\otimes z)\xi
=
\hat{\varepsilon}(x)\xi.
\]

$(ii)\Leftrightarrow (iii)$: This was essentially explained in \cite[Remark 9.4]{Ueda:Preprint20}. In fact, one has
\begin{gather*}
\overline{\Pi}^{\mathrm{rop}}\big(x\otimes1^\mathrm{op}\big)\xi = \Pi(x\otimes1)\xi, \\
\overline{\Pi}^{\mathrm{rop}}\big(1\otimes\big(\vartheta^{{\rm i}/2}(x)\big)^\mathrm{op}\big)\xi = \Pi\big(1\otimes\hat{R}\big(\vartheta^{{\rm i}/2}(x)\big)\big)\xi = \Pi\big(1\otimes\hat{S}^{-1}(x)\big)\xi,
\end{gather*}
and hence
\[
\Pi\big(\hat{S}(x)\otimes1\big)\xi = \Pi(1\otimes x)\xi
\Longleftrightarrow
\overline{\Pi}^{\mathrm{rop}}\big(x\otimes1^\mathrm{op}\big)\xi = \overline{\Pi}^{\mathrm{rop}}\big(1\otimes\big(\vartheta^{{\rm i}/2}(x)\big)^\mathrm{op}\big)\xi
\]
for all $x \in \mathcal{F}(G)$. Since $\mathcal{F}(G)$ is a $\sigma$-weakly dense $*$-subalgebra of $W^*(G)$, a standard appro\-ximation technique (see the proof of \cite[Proposition~4.3]{Ueda:Preprint20}) enables us to confirm the desired equivalence based on the above two identities.

The last assertion was proved in the above proof of $(ii)\Rightarrow (i)$. Hence we are done.
\end{proof}

We give a complete classification of irreducible spherical unitary representations of $G < G\times G$. The consequence is no surprise and the essentially same as the classical case.

\begin{Theorem}
Any irreducible spherical unitary representation of $G < G\times G$ is of the form $\pi_{\bar{U}}\otimes\pi_U \colon W^*(G\times G) \curvearrowright \mathcal{H}_{\bar{U}}\otimes\mathcal{H}_U$ for some irreducible unitary representation $U$ of $G$, and any spherical vector in $\pi_{\bar{U}}\otimes\pi_U \colon W^*(G\times G) \curvearrowright \mathcal{H}_{\bar{U}}\otimes\mathcal{H}_U$ must be a scalar multiple of $R_U(1) \in \mathcal{H}_{\bar{U}}\otimes\mathcal{H}_U$. Here, $(R_U, \bar{R}_U)$ is a unique solution of the conjugate equations for $U$ and $\bar{U}$.
\end{Theorem}

\begin{proof} By \eqref{Eq5.5}, any irreducible normal representation of $W^*(G\times G)$ is unitarily equivalent to $\pi_{U_{z_1}}\otimes\pi_{U_{z_2}} \colon W^*(G\times G) \curvearrowright \mathcal{H}_{U_{z_1}}\otimes\mathcal{H}_{U_{z_2}}$. By the final assertion of Proposition \ref{P5.4} $z_1$ must be $\bar{z}_2$, and hence $\pi_{U_{\bar{z}_2}}\otimes\pi_{U_{z_2}}$. Consequently, any irreducible spherical unitary representation of $G<G\times G$ must be of the form $\pi_{\bar{U}}\otimes\pi_U \colon W^*(G\times G) \curvearrowright \mathcal{H}_{\bar{U}}\otimes\mathcal{H}_U$. Using \cite[Theorem 6.1]{Ueda:Preprint20} through Proposition 4.3(iii) we see that any spherical vector $\xi$ is unique up to scalar multiple. Hence, thanks to Proposition \ref{P5.4}$(ii)$ and \cite[Example 2.2.23]{NeshveyevTuset:Book13} with the notation $T^\vee$ there, it suffices to confirm that $(T^\vee \otimes1)R_U = (1\otimes T)R_U$ holds for any $T \in B(\mathcal{H}_U)$. However, this identity is known to be the characterization of map $T \mapsto T^\vee$; see the proof of \cite[Proposition~2.2.10]{NeshveyevTuset:Book13}. Hence we are done.
\end{proof}

\begin{Remarks}\label{R5.6} \quad
\begin{enumerate}\itemsep=0pt
\item[$1.$] The above theorem says that any irreducible spherical unitary representation of $G<G\times G$ can be captured by means of $\mathcal{U}(G\times G)$ and its ``spherical state'' must be of the form
\begin{align*}
x \in \mathcal{U}(G\times G) \mapsto
((\pi_{\bar{U}}\otimes\pi_U)(x)R_U(1)\,|\,R_U(1))_{\mathcal{H}_{\bar{U}}\otimes\mathcal{H}_U} = R_U^* (\pi_{\bar{U}}\otimes\pi_U)(x) R_U \in \mathbb{C}
\end{align*}
up to scalar multiple. Moreover,
\[
R_U^* (\pi_{\bar{U}}\otimes\pi_U)(x\otimes y) R_U
=
\begin{cases}
R_U^* \big(\pi_{\bar{U}}\big(x\hat{S}(y)\big)\otimes1\big)R_U = \mathrm{Tr}\big(\pi_{\bar{U}}\big(\rho x\hat{S}(y)\big)\big), \\
R_U^*\big(1\otimes\pi_U\big(y\hat{S}^{-1}(x)\big)\big)R_U = \mathrm{Tr}\big(\pi_U\big(\rho^{-1}y\hat{S}^{-1}(x)\big)\big)
\end{cases}
\]
holds true for any $x,y \in \mathcal{F}(G)$; see \cite[Example~2.2.3]{NeshveyevTuset:Book13}. This should be regarded as the origin of quantized characters in the sense of Sato.

\item[$2.$] We have an anti-automorphism $\hat{\Theta}$ on $\mathcal{U}(G\times G)$ defined by
\[
\hat{\Theta}(x)(a\otimes b) := x(S(b)\otimes S(a)), \qquad x \in \mathcal{U}(G\times G),\quad a,b \in \mathbb{C}[G].
\]
Then
\begin{gather*}
\hat{\Theta}\big(\hat{\Delta}(x)\big)(a\otimes b) = x(S(b)S(a)) = x(S(ab)) = \hat{S}(x)(ab) = \hat{\Delta}\big(\hat{S}(x)\big)(a\otimes b), \\
x \in \mathcal{U}(G),\qquad a,b \in \mathbb{C}[G],
\end{gather*}
and hence $\hat{\Theta}\big(\hat{\Delta}(\mathcal{U}(G)\big) = \hat{\Delta}(\mathcal{U}(G))$ holds. This suggests that $\hat{\Delta}(\mathcal{U}(G)) \subset \mathcal{U}(G\times G)$ is a~realization of ``quantum Gelfand pair'' associated with $G<G\times G$. However, $\hat{\Theta}$ is not well defined on $W^*(G\times G)$.

\item[$3.$] Item (1) above suggests us to ``extend'' the notion of spherical unitary representations to obtain the whole information of conjugate solutions. For the purpose, we remark that general quantum groups are not co-commutative so that $\hat{\Delta}^\mathrm{cop} := \sigma\circ\hat{\Delta}$ with the flip map $\sigma$ gives another possible analog of the diagonal embedding $g \mapsto (g,g)$ of ordinary groups. Here is a~trivial translation of Proposition \ref{P5.4}: In the same setup there, the following are equivalent:
\begin{itemize}\itemsep=0pt
\item[$(i)$] $\overline{\Pi}\big(\hat{\Delta}^{\mathrm{cop}}(x)\big)\xi = \hat{\varepsilon}(x)\xi$ for all $x \in W^*(G)$.
\item[$(ii)$] $\Pi\big(1\otimes\hat{S}(x)\big)\xi=\Pi(x\otimes1)\xi$ for all $x \in \mathcal{F}(G)$.
\item[$(iii)$] $\xi$ is a $(\vartheta^t,+1)$-spherical vector in the bi-normal $*$-representation $\overline{\Pi}^{\mathrm{rop}} \colon W^*(G)\otimes_\mathrm{max}W^*(G)^\mathrm{op} \curvearrowright \mathcal{H}_\Pi$.
\end{itemize}\pagebreak
Then any irreducible spherical unitary representation of $G<G\times G$ with replacing $\hat{\Delta}$ with $\hat{\Delta}^{\mathrm{cop}}$ and any spherical vector in it must be of the form $\pi_U\otimes\pi_{\bar{U}} \colon W^*(G\times G) \curvearrowright \mathcal{H}_U\otimes\mathcal{H}_{\bar{U}}$ and a scalar multiple of $\bar{R}_U(1) \in \mathcal{H}_U\otimes\mathcal{H}_{\bar{U}}$, respectively. (See \cite[Theorem~2.2.21]{NeshveyevTuset:Book13} and its preliminary discussion.) Consequently, the $\bar{R}$-part of solutions of conjugate equations can be obtained as $(\vartheta^t,+1)$-spherical representations. It may be possible to give further deformations of diagonal embedding $\hat{\Delta} \colon \mathcal{U}(G) \to \mathcal{U}(G\times G)$.

\item[$4.$] There is no special reason, in Sato's approach \cite{Sato:JFA19, Sato:preprint19} to asymptotic representation theory for quantum groups, for which choice $\beta=\pm1$ of inverse temperature of KMS states is appropriate. (Note that only sign, or the time direction, is important, because the general~$\beta$ case can be transformed into one of the cases of $\beta=\pm1$ by positive scaling.) Hence Sato's choice of~$\beta$ was just on an arbitrary basis and not essential. However, the spherical representation theory shows that the choice must be $\beta=-1$ and taking time reversal corresponds to the co-opposite transition (or the interchange of~$R_U$ and~$\bar{R}_U$ in conjugate solutions).
\end{enumerate}
\end{Remarks}

\subsection{Inductive limits of compact quantum groups} \label{S5.2}

Let $G_n$, $n \geq 0$, be an inductive sequence of compact quantum groups with $G_0$ the trivial one, that is, for each $n$ there is a surjective Hopf $*$-algebra morphism $\mathbb{C}[G_{n+1}] \twoheadrightarrow \mathbb{C}[G_n]$. We have to formulate their inductive limit ``$G=\varinjlim G_n$'' in the category of ``quantum groups'' to discuss spherical unitary representations of $G < G\times G$.

The most widely accepted approach to ``quantum groups'' is based upon the notion of quantized universal enveloping algebras. The approach looks purely algebraic, but in principle, it involves differential operators, so that it requires a ``differential structure'' on \emph{infinite-dimensional} (virtual) space $G=\varinjlim G_n$. Such a~structure has not completely been established so far even in the classical setting such as $\mathrm{U}(\infty) = \varinjlim \mathrm{U}(n)$, to the best of our knowledge. Thus, it is difficult to use quantized universal enveloping algebras directly for our purpose. See Section~\ref{S5.4.6} for a~related discussion.
Moreover, the most plausible approach to unitary representations is the use of group $C^*$-algebras. However, the lack of ``Haar measures'' even in the classical setting due to the infinite-dimensionality becomes a serious issue to use the standard construction of group $C^*$-algebras.

With these reasons, Sato \cite{Sato:preprint19} used the framework of $W^*$-bialgebras to understand $G$ and developed its asymptotic unitary representation theory of Vershik--Kerov's type based on his earlier work \cite{Sato:JFA19}. The approach to formulate $G$ below is slightly different from his (and hopefully more natural than his), though they are equivalent at least for the questions he considered.

\subsubsection[Inductive Baire group C\textasciicircum{}*-algebra]
{Inductive Baire group $\boldsymbol{C^*}$-algebra} \label{S5.2.1}

Taking the dual of $\mathbb{C}[G_{n+1}] \twoheadrightarrow \mathbb{C}[G_n]$ for each $n$ we have an embedding $\mathcal{U}(G_n) \hookrightarrow \mathcal{U}(G_{n+1})$ that respects the dual Hopf $*$-algebra structure. We can take the algebraic inductive limit $\mathcal{U}(G) := \varinjlim \mathcal{U}(G_n)$ (see, e.g., \cite[Chapter 2]{Effros:CBMSBook} for the notion of inductive (or direct) limits), which is too big to discuss unitary representations. Here, we observe that the embedding sends $W^*(G_n)$ into $W^*(G_{n+1})$. Thus we have an inductive sequence $W^*(G_n)$ with the following additional structures: comultiplications $\hat{\Delta}_n \colon W^*(G_n) \to W^*(G_n\times G_n)$, counits $\hat{\varepsilon}_n \colon W^*(G_n)\allowbreak \to\mathbb{C}$, unitary antipodes $\hat{R}_n \colon W^*(G_n) \to W^*(G_n)$ and flows $\vartheta_n^t \colon \mathbb{R} \curvearrowright W^*(G_n)$. Those structures are well behaved under the embeddings $W^*(G_n) \hookrightarrow W^*(G_{n+1})$ that come from the natural ones $\mathcal{U}(G_n) \hookrightarrow \mathcal{U}(G_{n+1})$; namely, the restrictions of $\hat{\Delta}_{n+1}$, $\hat{\varepsilon}_{n+1}$, $\hat{R}_{n+1}$ and $\vartheta_{n+1}^t$ to $W^*(G_n)$ are exactly $\hat{\Delta}_n$, $\hat{\varepsilon}_n$, $\hat{R}_n$ and $\vartheta_n^t$ , respectively.

Here is a remark. Since the special positive elements $\rho_n$ (see \eqref{Eq5.7}) do not form an inductive sequence in any sense, it is a bit non-trivial that the compatibility of flows $\vartheta_n^t$ as well as the unitary antipodes $\hat{R}_n$ with the embeddings $\mathcal{U}(G_n) \hookrightarrow \mathcal{U}(G_{n+1})$. This fact can be confirmed by using the compatibility of antipodes $\hat{S}_n$ as follows. For any $x \in \mathcal{U}(G_n)$, we have
\[
\rho_n x \rho_n^{-1} = \hat{S}_n^2(x) = \hat{S}_{n+1}^2(x) = \rho_{n+1}x\rho_{n+1}^{-1}
\]
by the definition of unitary antipodes $\hat{R}_n$ with the help of \eqref{Eq5.10}. It follows that $[\rho_n,\rho_{n+1}] = 0 = \big[\rho_{n}^{-1}\rho_{n+1},x\big]$ for all $x \in \mathcal{U}(G_n)$. Hence we have
\[
\vartheta_{n+1}^t(x) = \rho_{n+1}^{{\rm i}t}\big(\rho_n^{-1}\rho_{n+1}\big)^{-{\rm i}t} x \big(\rho_n^{-1}\rho_{n+1}\big)^{{\rm i}t}\rho_{n+1}^{-{\rm i}t} = \rho_n^{{\rm i}t} x \rho_n^{-{\rm i}t} = \vartheta_n^t(x)
\]
for all $x \in \mathcal{U}(G_n)$. This and \eqref{Eq5.11}, \eqref{Eq5.12} imply the compatibility of unitary antipodes~$\hat{R}_n$. We~remark that this argument clearly works for any pair of compact quantum group and its subgroup.

So far and in what follows, we do not use any symbols to specify the embedding of $W^*(G_n)$ into $W^*(G_{n+1})$ unlike \cite{Sato:JFA19} for simplicity, though it encodes an important representation theoretic information; in fact, the embedding comes from the restriction procedure for irreducible unitary representations of $G_{n+1}$ to $G_n$. See Lemma \ref{L5.14}.

In the above setup, we consider the $C^*$-algebraic inductive limit $\mathcal{B}(G) := \varinjlim W^*(G_n)$ and denote by $\mathcal{B}_\infty(G)$ its canonical local $W^*$-subalgebra (that is the algebraic inductive limit of $W^*(G_n)$'s). We think of this $C^*$-algebra as a group algebra corresponding to a kind of ``inductive Baire structure'' associated with $G=\varinjlim G_n$.

We also consider the $C^*$-algebraic inductive limit $\mathcal{B}(G\times G) := \varinjlim W^*(G_n\times G_n)$, and denote by $\mathcal{B}_\infty(G\times G)$ its local $W^*$-algebra (the algebraic inductive limit of $W^*(G_n\times G_n)$). Then we have
\[
\hat{\Delta} := \varinjlim \hat{\Delta}_n \colon \ \mathcal{B}(G) \to \mathcal{B}(G\times G), \qquad \hat{R} := \varinjlim \hat{R}_n, \qquad \vartheta^t := \varinjlim \vartheta_n^t, \qquad \hat{\varepsilon} := \varinjlim\hat{\varepsilon}_n
\]
according to $\mathcal{B}(G) = \varinjlim W^*(G_n)$. Clearly, all the maps above are locally normal. The pentad $\big(\mathcal{B}(G), \hat{\Delta} \colon \mathcal{B}(G) \to \mathcal{B}(G\times G), \hat{R}, \vartheta^t, \hat{\varepsilon}\big)$ enjoys almost all the requirements to be ``Hopf $*$-algebra'' like $(\mathcal{U}(G_n),\hat{\Delta}_n)$.

\begin{Definition}
We call $\mathcal{B}(G)$ \emph{the inductive Baire group $C^*$-algebra of $G=\varinjlim G_n$} and understand the pentad $\big(\mathcal{B}(G), \hat{\Delta} \colon \mathcal{B}(G) \to \mathcal{B}(G\times G), \hat{R}, \vartheta^t, \hat{\varepsilon}\big)$ as \emph{the quantum group structure} of $G=\varinjlim G_n$.
\end{Definition}

Let $W^*(G)$ be the locally normal enveloping $W^*$-algebra of $\mathcal{B}(G)$, i.e., a unique $W^*$-algebra whose predual is given by all the locally normal linear functionals on $\mathcal{B}(G)$ so that $\mathcal{B}(G)$ sits in $W^*(G)$ as a $\sigma$-weakly dense subalgebra. Similarly, let $W^*(G\times G)$ be the locally normal enveloping $W^*$-algebra of $\mathcal{B}(G\times G)$. These $W^*(G)$ and $W^*(G\times G)$ are nothing but the $W^*$-inductive limits of $W^*(G_n)$ and $W^*(G_n\times G_n)$ in the sense of Takeda. (See \cite[Section 5.2]{Ueda:Preprint20} for the notion of locally normal enveloping $W^*$-algebras and $W^*$-inductive limits.) We have normal extensions of $\hat{\varepsilon}$, $\hat{R}$, $\vartheta^t$ to $W^*(G)$ and also a normal $*$-homomorphism $\hat{\Delta} \colon W^*(G) \to W^*(G\times G)$ with keeping the same symbols. With a canonical surjection $W^*(G\times G) \twoheadrightarrow W^*(G)\,\bar{\otimes}\,W^*(G)$ (whose existence is an easy exercise), we have obtained a normal $*$-homomorphism
\[
W^*(G) \overset{\hat{\Delta}}{\to} W^*(G\times G) \twoheadrightarrow W^*(G)\,\bar{\otimes}\,W^*(G),
\]
which is nothing but the comultiplication in \cite{Sato:JFA19}. In this way, Sato formulated $G=\varinjlim G_n$ within the framework of $W^*$-bialgebras. Although his approach has a merit as an attempt to enlarge the operator algebraic framework of quantum groups beyond the locally compact setting, we prefer to regard $\hat{\Delta} \colon \mathcal{B}(G) \to \mathcal{B}(G\times G)$ as a ``true comultiplication'' instead.

\subsubsection{Spherical unitary representations}
It is easy to see that the algebraic tensor product $\mathcal{B}_\infty(G)\otimes\mathcal{B}_\infty(G)$ is naturally identified with the algebraic inductive limit of the sequence $W^*(G_n)\otimes W^*(G_n)$, which sits in $\mathcal{B}_\infty(G\times G)$. Also, it is trivial that $\mathcal{B}_\infty(G)\otimes\mathcal{B}_\infty(G)$ is naturally embedded into $\mathcal{B}(G)\otimes_{\max}\mathcal{B}(G)$. We will translate Lemma \ref{L5.2}, Definition \ref{D5.3} and Proposition \ref{P5.4} above into Lemma \ref{L5.8}, Definition \ref{D5.9} and Proposition \ref{P5.10}, respectively, for the inductive limit $G = \varinjlim G_n$.

\begin{Lemma} \label{L5.8} Let $\Pi \colon \mathcal{B}_\infty(G)\otimes\mathcal{B}_\infty(G) \curvearrowright \mathcal{H}_\Pi$ be a locally bi-normal $*$-representation. Then, there are a locally bi-normal $*$-representation $\overline{\Pi} \colon \mathcal{B}(G)\otimes_{\max}\mathcal{B}(G) \curvearrowright \mathcal{H}_\Pi$ and a locally normal $*$-representation $\underline{\Pi} \colon \mathcal{B}(G\times G) \curvearrowright \mathcal{H}_\Pi$ such that the diagram
\[
\xymatrix{
\mathcal{B}(G)\otimes_{\max}\mathcal{B}(G) \ar[rd]^{\overline{\Pi}} & \\
\mathcal{B}_\infty(G)\otimes\mathcal{B}_\infty(G) \ar@{^{(}-_>}[u] \ar@{^{(}-_>}[d] \ar[r]^{\quad \Pi} & B(\mathcal{H}_\Pi) \\
\mathcal{B}(G\times G) \ar[ru]_{\underline{\Pi}} &
}
\]
commutes, where the vertical arrows mean the embeddings explained above.
\end{Lemma}

\begin{proof}
It is obvious that $\Pi$ extends to $\mathcal{B}(G)\otimes\mathcal{B}(G)$, and hence does to $\mathcal{B}(G)\otimes_{\max}\mathcal{B}(G)$ by means of maximal $C^*$-algebraic tensor products. This extension is nothing but the desired~$\overline{\Pi}$.\looseness=1

For each $n$, we consider the restriction $\Pi_n$ of $\Pi$ to $W^*(G_n)\otimes W^*(G_n)$ and apply Lem\-ma~\ref{L5.2} to this restriction. Then we have a unique normal $*$-representation $\underline{\Pi}_n \colon W^*(G_n\times G_n) \curvearrowright \mathcal{H}_\Pi$ that agrees with $\Pi_n$. Thanks to the normality, we see that $\underline{\Pi}_{n+1}$ agrees with $\underline{\Pi}_n$ over the whole $W^*(G_n\times G_n)$ rather than $W^*(G_n)\otimes W^*(G_n)$. Thus the inductive structure of $\mathcal{B}(G\times G)$ enables us to construct a locally normal $*$-representation $\underline{\Pi} \colon \mathcal{B}(G\times G) \curvearrowright \mathcal{H}_\Pi$. By construction, this $*$-representation agrees with $\Pi$ on every $W^*(G_n)\otimes W^*(G_n)$ so that the commutative diagram holds. Hence we have obtained the desired locally normal $*$-represen\-tation~$\underline{\Pi}$.\looseness=1
\end{proof}

We call a locally bi-normal $*$-representation $\Pi \colon \mathcal{B}_\infty\otimes\mathcal{B}_\infty(G) \curvearrowright \mathcal{H}_\Pi$ \emph{a unitary representation of $G\times G$} with $G = \varinjlim G_n$. The above lemma shows that there are two natural extensions $\overline{\Pi} \colon \mathcal{B}(G)\otimes_{\max}\mathcal{B}(G) \curvearrowright \mathcal{H}_\Pi$ and $\underline{\Pi} \colon \mathcal{B}(G\times G) \curvearrowright \mathcal{H}_\Pi$ with appropriate continuity. Here is a~defi\-nition.

\begin{Definition}\label{D5.9} A unitary representation $\Pi$ of $G\times G$ on $\mathcal{H}_\Pi$ together with a unit vector $\xi \in \mathcal{H}_\Pi$ is called a \emph{spherical unitary representation} of $G < G\times G$ (with $G = \varinjlim G_n$), if $\xi$ is cyclic for $\underline{\Pi}(\mathcal{B}(G\times G))$ and $\underline{\Pi}\big(\hat{\Delta}(x)\big)\xi=\hat{\varepsilon}(x)\xi$ holds for every $x\in \mathcal{B}(G)$. We call such a vector \emph{a spherical vector} in $\mathcal{H}_\Pi$ of $G < G\times G$ through $\Pi$.
\end{Definition}

It is now easy to confirm that this definition completely agrees with the definition of $(\vartheta^t,-1)$-spherical representation. Namely, all the results in our previous paper can successfully be applied to spherical representations of $G < G\times G$ with $G = \varinjlim G_n$. In other words, our studies certainly justify Sato's asymptotic representation theory for $G = \varinjlim G_n$ from the viewpoint of spherical representations, and especially, explains why his notion of quantized characters playing a central r\^{o}le in \cite{Sato:JFA19,Sato:preprint19,Sato:ETDS2x} is appropriate.

{\samepage\begin{Proposition} \label{P5.10} Let $\Pi$ be a unitary representation of $G\times G$ on a Hilbert space $\mathcal{H}_\Pi$ and $\xi \in \mathcal{H}_\Pi$ be a non-zero vector. The following are equivalent:
\begin{itemize}\itemsep=0pt
\item[$(i)$] $\underline{\Pi}\big(\hat{\Delta}(x)\big)\xi=\hat{\varepsilon}(x)\xi$ holds for every $x\in\mathcal{B}(G)$.
\item[$(ii)$] $\Pi\big(\hat{S}_n(x)\otimes1\big)\xi=\Pi(1\otimes x)\xi$ holds for every $x \in \mathcal{F}(G_n)$ and $n \geq 0$.
\item[$(iii)$] $\xi$ is $(\vartheta^t,-1)$-spherical vector in the $*$-representation $\overline{\Pi}^{\mathrm{rop}} \colon \mathcal{B}(G)\otimes_{\max}\mathcal{B}(G)^{\mathrm{op}} \curvearrowright \mathcal{H}_\Pi$ defined to be
\[
\begin{matrix}
\mathcal{B}(G)\otimes_\mathrm{max}\mathcal{B}(G)^\mathrm{op}
& \overset{\mathrm{id}\,\otimes\,\hat{R}}{\to} &
\mathcal{B}(G)\otimes_\mathrm{max}\mathcal{B}(G)
&\overset{\overline{\Pi}}{\to} &
B(\mathcal{H}_\Pi), \\
x\otimes y^\mathrm{op}
& \mapsto &
x\otimes\hat{R}(y)
& \mapsto &
\overline{\Pi}\big(x\otimes\hat{R}(y)\big).
\end{matrix}
\]
See {\rm \cite[Definition~5.1, Remark and Definition~5.2]{Ueda:Preprint20}} for the notion of $(\vartheta^t,-1)$-spherical vectors.
\end{itemize}
\end{Proposition}}

\begin{proof}
$(i)\Leftrightarrow (ii)$: Since $\hat{\Delta} = \varinjlim\hat{\Delta}_n$ and $\hat{\varepsilon} = \varinjlim\hat{\varepsilon}_n$ together with density, item $(i)$ is equivalent to that $\underline{\Pi}\big(\hat{\Delta}_n(x)\big)\xi = \hat{\varepsilon}_n(x)\xi$ holds for all $x \in W^*(G_n\times G_n)$ and $n \geq 0$. Applying Proposition~\ref{P5.4} to each level $n$, we see that this is equivalent to item $(ii)$.

$(ii)\Leftrightarrow (iii)$: As in the proof of Proposition~\ref{P5.4}$(ii)\Leftrightarrow (iii)$ we see that item $(ii)$ is equiva\-lent~to
\[
\overline{\Pi}^{\mathrm{rop}}(x\otimes1)\xi = \overline{\Pi}^{\mathrm{rop}}\big(1\otimes\big(\vartheta_n^{{\rm i}/2}(x)\big)^{\mathrm{op}}\big)\xi = \overline{\Pi}^{\mathrm{rop}}\big(1\otimes\big(\vartheta^{{\rm i}/2}(x)\big)^{\mathrm{op}}\big)\xi, \qquad
x \in \bigcup_n \mathcal{F}(G_n).
\]
Using the $\sigma$-weak density of $\mathcal{F}(G_n)$ in $W^*(G_n)$ and the norm density of $\mathcal{B}_\infty(G)$ in $\mathcal{B}(G)$ together with the Phragmen--Lindel\"{o}f method, we can prove that $\xi$ is a $(\vartheta^t,-1
)$-spherical vector in $\overline{\Pi}^{\mathrm{rop}} \colon \mathcal{B}(G)\otimes_{\max}\mathcal{B}(G)^{\mathrm{op}} \curvearrowright \mathcal{H}_\Pi$.
\end{proof}

Let $\Pi$ be a unitary representation of $G\times G$ with $G = \varinjlim G_n$ on a Hilbert space $\mathcal{H}_\Pi$. By~construction we observe that
\begin{align*}
\Pi(\mathcal{B}_\infty(G)\otimes\mathcal{B}_\infty(G))'' &= \underline{\Pi}(\mathcal{B}(G\times G))'' = \overline{\Pi}(\mathcal{B}(G)\otimes_\mathrm{max}\mathcal{B}(G))''
\\
&= \overline{\Pi}^\mathrm{rop}\big(\mathcal{B}(G)\otimes_\mathrm{max}\mathcal{B}(G)^\mathrm{op}\big)''
\end{align*}
on $\mathcal{H}_\Pi$, and hence \cite[Theorem~6.1]{Ueda:Preprint20} is available in this setup thanks to the above proposition. Therefore, we conclude that $G < G\times G$ is a Gelfand pair in the sense of Olshanski; see, e.g., \cite[Section 6]{Ueda:Preprint20}.

The above proposition also gives the next desired assertion.

\begin{Corollary}\label{C5.11} Let $\Pi$ be a unitary representation of $G\times G$ with $G = \varinjlim G_n$ on a Hilbert space~$\mathcal{H}_\Pi$ and $\xi \in \mathcal{H}_\Pi$ be a unit vector. Then the following are equivalent:
\begin{itemize}\itemsep=0pt
\item[$(i)$] $(\Pi,\mathcal{H}_\Pi,\xi)$ is a spherical unitary representation of $G<G\times G$ with $G = \varinjlim G_n$.
\item[$(ii)$] $\big(\overline{\Pi}^{\mathrm{rop}},\mathcal{H}_\Pi,\xi\big)$ is a locally normal $(\vartheta^t,-1)$-spherical representation of $\mathcal{B}(G)$.
\end{itemize}
\end{Corollary}

Remark that if the $\hat{\Delta}_n$ are replaced with $\hat{\Delta}_n^{\mathrm{cop}}$ as in Remarks \ref{R5.6}(3), then one obtains locally normal $(\vartheta^t,+1)$-spherical representations.

The above corollary guarantees that the general theory for $(\alpha^t,\beta)$-spherical representations developed in the previous and this papers can be applied to spherical unitary representations of $G<G\times G$ with $G = \varinjlim G_n$ as follows.

To understand $G\times G$ with $G = \varinjlim G_n$ rigorously, we have introduced four algebras $\mathcal{B}_\infty(G)\otimes\mathcal{B}_\infty(G)$, $\mathcal{B}(G\times G)$, $\mathcal{B}(G)\otimes_{\max}\mathcal{B}(G)$ and $\mathcal{B}(G)\otimes_{\max}\mathcal{B}(G)^{\mathrm{op}}$. Hence we first have to show that no differences occur among those choices of algebras to discuss unitary representations of $G\times G$. The next lemma guarantees that one can use each of those algebras to examine the unitary equivalence for unitary representations of $G\times G$.

\begin{Lemma}\label{L5.12} Let $\Pi_i$ be a unitary representation of $G\times G$ on a Hilbert space $\mathcal{H}_{\Pi_i}$ for each $i=1,2$. For any $u \in B(\mathcal{H}_{\Pi_1},\mathcal{H}_{\Pi_2})$ the following are equivalent:
\begin{itemize}\itemsep=0pt
\item[$(i)$] $u\Pi_1(x) = \Pi_2(x)u$ for all $x \in \mathcal{B}_\infty(G)\otimes\mathcal{B}_\infty(G)$.
\item[$(ii)$] $u\underline{\Pi}_1(x) = \underline{\Pi}_2(x)u$ for all $x \in \mathcal{B}(G\times G)$.
\item[$(iii)$] $u\overline{\Pi}_1(x) = \overline{\Pi}_2(x)u$ for all $x \in \mathcal{B}(G)\otimes_{\max}\mathcal{B}(G)$.
\item[$(iv)$] $u\overline{\Pi}^{\mathrm{rop}}_1(x) = \overline{\Pi}^{\mathrm{rop}}_2(x)u$ for all $x \in \mathcal{B}(G)\otimes_{\max}\mathcal{B}(G)^{\mathrm{op}}$.
\end{itemize}
\end{Lemma}

\begin{proof}
$(i)\Rightarrow (ii)$: Notice that
\[
\mathcal{B}_\infty(G)\otimes\mathcal{B}_\infty(G) \supset W^*(G_n)\otimes W^*(G_n) \subset W^*(G_n\times G_n) \subset \mathcal{B}(G\times G)
\]
and $\underline{\Pi}_i(x) = \Pi_i(x)$ holds for every $x \in W^*(G_n)\otimes W^*(G_n)$. Assume that item $(i)$ holds. By~the local normality we observe that $u\underline{\Pi}_1(x) = \underline{\Pi}_2(x)u$ holds for every $x \in W^*(G_n\times G_n)$, $n\geq0$. Then, by the norm-continuity we obtain item $(ii)$.

$(ii)\Rightarrow (iii)$: Assume that item $(ii)$ holds. Since
\[
\mathcal{B}_\infty(G)\otimes\mathcal{B}_\infty(G) = \varinjlim W^*(G_n)\otimes W^*(G_n) \hookrightarrow \mathcal{B}(G\times G)
\]
naturally, we see that $\overline{\Pi}_i(x) = \underline{\Pi}_i(x)$ holds for every $x \in \mathcal{B}_\infty(G)\otimes\mathcal{B}_\infty(G)$ and hence $u\overline{\Pi}_1(x) = \overline{\Pi}_2(x)u$ holds for every $x \in \mathcal{B}_\infty(G)\otimes\mathcal{B}_\infty(G)$. By the norm-continuity we obtain item $(iii)$.

$(iii)\Rightarrow(i)$: This is trivial, because $\mathcal{B}_\infty(G)\otimes\mathcal{B}_\infty(G) \subset \mathcal{B}(G)\otimes_{\max}\mathcal{B}(G)$ and $\overline{\Pi}_i$ and $\Pi_i$ agree on $\mathcal{B}_\infty(G)\otimes\mathcal{B}_\infty(G)$.

$(iii)\Leftrightarrow (iv)$: This follows from the definition of $\overline{\Pi}^\mathrm{rop}_i$.
\end{proof}

Let $\chi \in K_{-1}^\mathrm{ln}(\vartheta^t)$ be given, and $\pi_\chi \colon \mathcal{B}(G) \curvearrowright \mathcal{H}_\chi$ with a distinguished unit vector $\xi_\chi \in \mathcal{H}_\chi$ be the GNS triple associated with $\chi$. Let $J_\chi$ be the modular conjugation operator for $(\pi_\chi(\mathcal{B}(G))'',\xi_\chi)$. Then a locally bi-normal $*$-representation $\pi_\chi^{(2)} \colon \mathcal{B}(G)\otimes_{\max}\mathcal{B}(G)^\mathrm{op} \curvearrowright \mathcal{H}_\chi$ is defined by
\begin{equation*}
\pi_\chi^{(2)}\big(x\otimes y^\mathrm{op}\big) := \pi_\chi(x)J_\chi\pi_\chi(y)^* J_\chi, \qquad x,y \in \mathcal{B}(G).
\end{equation*}
See \cite[Theorem~5.7(1)]{Ueda:Preprint20}. We then define the other locally bi-normal $*$-representation $\Pi_\chi \colon \mathcal{B}_\infty(G)\allowbreak\otimes\mathcal{B}_\infty(G) \curvearrowright \mathcal{H}_\chi$ by
\begin{equation*}
\Pi_\chi(x\otimes y) = \pi_\chi^{(2)}\big(x\otimes\big(\hat{R}(y)\big)^\mathrm{op}\big), \qquad x,y \in \mathcal{B}(G).
\end{equation*}
Thanks to Lemma \ref{L5.8} we obtain a unique locally bi-normal $*$-representation $\overline{\Pi}_\chi \colon \mathcal{B}(G)\otimes_{\max}\mathcal{B}(G) \curvearrowright \mathcal{H}_\chi$ and a unique locally normal $*$-representation $\underline{\Pi}_\chi \colon \mathcal{B}(G\times G) \curvearrowright \mathcal{H}_\chi$ in such a way that
\begin{equation*}
\overline{\Pi}_\chi(x) = \underline{\Pi}_\chi(x) = \Pi_\chi(x)
\end{equation*}
holds for every $x \in \mathcal{B}_\infty(G)\otimes\mathcal{B}_\infty(G)$. Remark that the identity $\overline{\Pi}_\chi^\mathrm{rop} = \pi_\chi^{(2)}$ trivially holds. Hence, \cite[Theorem 5.7(1)]{Ueda:Preprint20} together with Corollary \ref{C5.11} and Lemma \ref{L5.12} immediately implies the following:

\begin{Theorem}\label{T5.13} For each $\chi \in K_{-1}^\mathrm{ln}(\vartheta^t)$, the associated triple $(\Pi_\chi, \mathcal{H}_\chi,\xi_\chi)$ is a spherical unitary representation of $G<G\times G$, and moreover, the correspondence $\chi \leftrightarrow [(\Pi_\chi, \mathcal{H}_\chi,\xi_\chi)]$ gives a~bijection between $K_{-1}^\mathrm{ln}(\vartheta^t)$ and the unitary equivalence classes of spherical unitary representation of $G<G\times G$.
\end{Theorem}

In closing of this section, we discuss the spectral decomposition for spherical unitary representations of $G < G\times G$ in terms of $\mathcal{B}(G\times G)$. Let $(\Pi,\mathcal{H}_\Pi,\xi)$ be a spherical unitary representation of $G<G\times G$. By \cite[Theorem 8.4(3)]{Ueda:Preprint20} there is a unique Borel probability measure $\mu$ on the extreme points $\mathrm{ex}\big(K_{-1}^\mathrm{ln}(\vartheta^t)\big)$ so that
\[
\big(\overline{\Pi}^\mathrm{rop},\mathcal{H}_\Pi,\xi\big) = \int^\oplus_{\mathrm{ex}(K_{-1}^\mathrm{ln}(\vartheta^t))} \big(\pi_\chi^{(2)}, \mathcal{H}_\chi,\xi_\chi\big)\,\mu({\rm d}\chi).
\]
It follows that
\[
\big(\overline{\Pi},\mathcal{H}_\Pi,\xi\big) = \int^\oplus_{\mathrm{ex}(K_{-1}^\mathrm{ln}(\vartheta^t))} \big(\overline{\Pi}_\chi, \mathcal{H}_\chi,\xi_\chi\big)\,\mu({\rm d}\chi),
\]
and hence
\[
\underline{\Pi}(x\otimes y) = \overline{\Pi}(x\otimes y) = \int^\oplus_{\mathrm{ex}(K_{-1}^\mathrm{ln}(\vartheta^t))} \overline{\Pi}_\chi(x\otimes y)\,\mu({\rm d}\chi) = \int^\oplus_{\mathrm{ex}(K_{-1}^\mathrm{ln}(\vartheta^t))} \underline{\Pi}_\chi(x\otimes y)\,\mu({\rm d}\chi)
\]
holds for all $x,y \in \mathcal{B}_\infty(G)$.
Since the $\sigma$-weak density of each $W^*(G_n)\otimes W^*(G_n)$ in $W^*(G_n\times G_n)$ together with the local normality of $\underline{\Pi}_\chi$ plus the Kaplansky density theorem and then since the norm-density of $\bigcup_{n\geq0} W^*(G_n\times G_n)$ in $\mathcal{B}(G\times G)$, we can easily confirm that $\chi \mapsto \underline{\Pi}_\chi(x)$ is Borel for every $x \in \mathcal{B}(G\times G)$, and obtain that
\[
 \underline{\Pi}(x) = \int^\oplus_{\mathrm{ex}(K_{-1}^\mathrm{ln}(\vartheta^t))} \underline{\Pi}_\chi(x)\,\mu({\rm d}\chi), \qquad x \in \mathcal{B}(G\times G).
\]
In this way, \emph{all the formulations $\Pi$, $\underline{\Pi}$, $\overline{\Pi}$ and $\overline{\Pi}^\mathrm{rop}$ of a spherical unitary representation of $G < G\times G$ admit a simultaneous spectral decomposition with a unique probability measure over the extreme points $\mathrm{ex}\big(K_{-1}^\mathrm{ln}(\vartheta^t)\big)$}.

\subsection{Quantized character functions} 

So far, we have provided all the necessary ingredients to discuss the (spherical) unitary representation theory for $G = \varinjlim G_n$. As a consequence, we saw that the study of (spherical) unitary representations comes down to that of locally normal $(\vartheta^t,-1)$-KMS states. Thus, one may regard such a state as a quantum analog of character, and this idea was indeed employed by Sato in \cite{Sato:JFA19} (and his later works) to develop the asymptotic representation theory for quantum groups (without appealing to spherical representations). In our opinion, it is more appropriate to formulate quantized characters as ``functions over $G$''. This aspect was already considered by Sato \cite[Section 3.3]{Sato:JFA19} based on Arano's suggestion to him. However, Sato's discussion there is not entirely satisfactory (at least to us), because his discussion depends on Gorin's work \cite{Gorin:AdvMath12} and is performed only in a special case. Thus, we give here a new formulation of quantized characters of $G$, which enables us to give a (hopefully nice) interpretation to Gorin's work in terms of quantum groups in the next section.

Let $C(G_n)$ be the $C^*$-enveloping algebra of $\mathbb{C}[G_n]$. See \cite[Section 1.7]{NeshveyevTuset:Book13}, where the symbol $C_u(G_n)$ is used instead. Then the surjective $*$-homomorphism $\mathbb{C}[G_{n+1}] \twoheadrightarrow \mathbb{C}[G_n]$ extends to the $C^*$-level, i.e., $r_n^{n+1} \colon C(G_{n+1}) \twoheadrightarrow C(G_n)$.
Consider the $\sigma$-$C^*$-algebra
\[
C(G) := \varprojlim C(G_n) = \bigg\{ (x_n) \in \prod_{n\geq0} C(G_n);\,  r_n^{n+1}(x_{n+1}) = x_n  \ \text{for all} \ n\geq0 \bigg\},
\]
in which we can construct a ``function'' associated with any $\chi \in K_{-1}^\mathrm{ln}(\vartheta^t)$. This type of algebras was considered in \cite{MahantaMathai:LMP11} to understand the quantized universal gauge group ${\rm SU}_q(\infty)$ in Woro\-no\-wicz's formalism. We remark that if all the $G_n$ are ordinary groups, then $C(G)$ is exactly identified with the algebra of continuous functions over the topological group $G = \varinjlim G_n$.

We will use the bijective $*$-isomorphism $\Phi_{U_\bullet}$; see \eqref{Eq5.1}. We define a unitary element $u_z \in zW^*(G_n)\otimes\mathbb{C}[G_n] \subset W^*(G_n)\otimes\mathbb{C}[G_n]$ by $\big(\Phi_{U_\bullet}^{-1}\otimes\mathrm{id}\big)(U_z)$, where we regard $U_z \in B(\mathcal{H}_{U_z})\otimes\mathbb{C}[G_n] \subset (\prod_{z\in\mathfrak{Z}_n}B(\mathcal{H}_{U_z}))\otimes\mathbb{C}[G_n]$ naturally. It is easy (see the proof of the lemma below) to confirm that this unitary element $u_z \in zW^*(G_n)\otimes\mathbb{C}[G_n]$ depends only on the unitary equivalence class of $U_z$. Here is a technical lemma.

\begin{Lemma}\label{L5.14} For every $z \in \mathfrak{Z}_{n+1}$, we have
\[
\big(\mathrm{id}\otimes r_n^{n+1}\big)(u_z) = \sum_{z' \in \mathfrak{Z}_n} (z\otimes1)u_{z'},
\]
a finite sum in $zW^*(G_n)\otimes\mathbb{C}[G_n] \subset zW^*(G_{n+1})\otimes\mathbb{C}[G_n]$.
\end{Lemma}

The above identity is essentially the irreducible decomposition of the restriction of $U_z$ to~$G_n$. The reason why the information of multiplicities does not appear in the above identity is that it is built in the embedding $W^*(G_n) \hookrightarrow W^*(G_{n+1})$. Actually, the multiplicity of $U_{z'}$ in $U_z$ is exactly $\mathrm{Tr}(zz')$ with a (unique) non-normalized trace $\mathrm{Tr}$ on $zW^*(G_{n+1}) \cong B(\mathcal{H}_{U_z})$; this is consistent with Bratteli diagram description of the inductive sequence $W^*(G_n)$, see Section \ref{S2} or \cite[Section 9]{Ueda:Preprint20} in detail. We remark that the essentially same observation has been given in \cite[Lemma 2.10]{Tomatsu:CMP07} in a slightly different language.

\begin{proof} It is convenient here to write the embedding map of $\mathcal{U}(G_n) \hookrightarrow \mathcal{U}(G_{n+1})$ explicitly. The embedding map is the transpose ${}^t r_n^{n+1} \colon \mathcal{U}(G_n) \to \mathcal{U}(G_{n+1})$ of $r_n^{n+1} \colon \mathbb{C}[G_{n+1}] \twoheadrightarrow \mathbb{C}[G_n]$, and it also induces the embedding $W^*(G_n) \hookrightarrow W^*(G_{n+1})$. With the map ${}^t r_n^{n+1}$ we will show that
\[
\big(\mathrm{id}\otimes r_n^{n+1}\big)(u_z) = \sum_{z' \in \mathfrak{Z}_n} (z\otimes1)\big({}^t r_n^{n+1}\otimes\mathrm{id}\big)(u_{z'}).
\]

Note that $z\mathcal{U}(G_{n+1}) = z W^*(G_{n+1})$ by definition. We will compute the both sides of the above desired identity in the following framework:
\[
\xymatrix{
zW^*(G_{n+1})\otimes\mathbb{C}[G_{n+1}]
\ar@{->}[d]_{\mathrm{id}\otimes r_n^{n+1}}
\ar@{^{(}-_{>}}[r]
& L(\mathcal{U}(G_{n+1}), z\,\mathcal{U}(G_{n+1}))
\ar@{->}[d]^{T \mapsto T\circ {}^t r_n^{n+1}} \\
zW^*(G_{n+1})\otimes\mathbb{C}[G_n]
\ar@{^{(}-_{>}}[r] \ar@{}[ur]|{\circlearrowright}
& L(\mathcal{U}(G_n), z\,\mathcal{U}(G_{n+1})).
}
\]
Here, $L(V,W)$ stands for all the linear maps from a vector space $V$ to another $W$. The horizontal embeddings are given by $(a\otimes b)x = x(b)\,a$ for any element $x$ of $\mathcal{U}(G_{n+1})$ and $\mathcal{U}(G_n)$, respectively.

We first remark that $u_z \in zW^*(G_{n+1})\otimes\mathbb{C}[G_{n+1}]$ becomes the mapping $x \in \mathcal{U}(G_{n+1}) \mapsto zx \in z\,\mathcal{U}(G_{n+1})$ as an element of $L(\mathcal{U}(G_{n+1}), z\,\mathcal{U}(G_{n+1}))$. In fact, for every $x \in \mathcal{U}(G_{n+1})$, we have
\begin{equation}\label{Eq5.15}
u_z x = \big(\Phi_{U_\bullet}^{-1}\otimes\mathrm{id}\big)(U_z)x
= x((U_z)_{(2)}) \Phi_{U_\bullet}^{-1}((U_z)_{(1)})
= \Phi_{U_\bullet}^{-1}(\pi_{U_z}(x)) = zx
\end{equation}
with Sweedler's notation $U_z = (U_z)_{(1)}\otimes(U_z)_{(2)}$, where we regard
\[
\pi_{U_z}(x) \in B(\mathcal{H}_{U_z}) \subset \prod_{z\in\mathfrak{Z}_{n+1}} B(\mathcal{H}_{U_z})
\]
naturally as before.

Since $zW^*(G_{n+1})$ is finite dimensional, we see that $z\,{}^t r_n^{n+1}(z') \neq 0$ for only finitely many $z' \in \mathfrak{Z}_n$. Then, for every $y \in \mathcal{U}(G_n)$ we have\pagebreak
\begin{align*}
\big(\mathrm{id}\otimes r_n^{n+1}\big)(u_z)y
&= u_z {}^t r_n^{n+1}(y) = z\, {}^t r_n^{n+1}(y) \\
&= \sum_{z' \in \mathfrak{Z}_n} z\, {}^t r_n^{n+1}(z' y) \\
&= \sum_{z' \in \mathfrak{Z}_n} z\, {}^t r_n^{n+1}(u_{z'}y) \quad \text{(as in \eqref{Eq5.15})}\\
&= \sum_{z' \in \mathfrak{Z}_n} y((u_{z'})_{(2)}) z\, {}^t r_n^{n+1}((u_{z'})_{(1)})\\
&= \sum_{z' \in \mathfrak{Z}_n} \big((z\otimes 1)\big({}^t r_z^{n+1}\otimes\mathrm{id}\big)(u_{z'})\big)y
\end{align*}
with Sweedler's notation $u_{z'} = (u_{z'})_{(1)}\otimes (u_{z'})_{(2)}$. This shows the desired identity.
\end{proof}

Let $\omega \in S^\mathrm{ln}(\mathcal{B}(G))$, the convex set of all locally normal states on $\mathcal{B}(G)$, be arbitrarily given. For each $z \in \mathfrak{Z}_n$ we define $\omega(u_z) \in C(G_n)$ by $\omega(u_z) := (\omega\otimes\mathrm{id})(u_z) \in \mathbb{C}[G_n] \subset C(G_n)$. Since
\[
\Vert\omega(u_z)\Vert = \omega(z)\big\Vert\big(\big(\omega(z)^{-1}\omega\big)\otimes\mathrm{id}\big)(u_z)\big\Vert \leq \omega(z)
\]
and since $\sum_{z\in\mathfrak{Z}_n}\omega(z) = 1$ thanks to the local normality of $\omega$, we see that
\[
\omega(u_n) := \sum_{z\in\mathfrak{Z}_n} \omega(u_z)
\]
converges in $C(G_n)$ with respect to the norm topology.

\begin{Proposition}\label{P5.15} For every $n\geq0$, $r_n^{n+1}(\omega(u_{n+1})) = \omega(u_n)$ holds. Hence the sequence $(\omega(u_n))$ defines an element $\omega(u) \in C(G)$.
\end{Proposition}

This should be understood as a generalization of \cite[Proposition 4.6]{Gorin:AdvMath12}. Namely, the discussion here enables us to give a representation-theoretic interpretation to Gorin's $q$-Schur generating functions; in fact, such a function is nothing less than the restriction of a quantized character function to the ``diagonal subgroup'' of $\mathrm{U}_q(\infty)$; see the next section.

\begin{proof}
We have
\[
r_n^{n+1}(\omega(u_{n+1})) = \sum_{z\in\mathfrak{Z}_{n+1}} r_n^{n+1}(\omega(u_z)) = \sum_{z\in\mathfrak{Z}_{n+1}}\sum_{z'\in\mathfrak{Z}_n} (\omega\otimes\mathrm{id})((z\otimes1)u_{z'})
\]
by Lemma \ref{L5.14}.
Since
\[
\Vert (\omega\otimes\mathrm{id}((z\otimes1)u_{z'})\Vert \leq \omega(zz')
\]
as before, we have
\[
r_n^{n+1}(\omega(u_{n+1})) = \sum_{z'\in\mathfrak{Z}_n}\sum_{z\in\mathfrak{Z}_{n+1}} (\omega\otimes\mathrm{id})((z\otimes1)u_{z'}).
\]
Since $u_{z'} \in z'W^*(G_n)\otimes\mathbb{C}[G_n]$ and $z'W^*(G_n)$ is finite dimensional, we have
\[
\sum_{z\in\mathfrak{Z}_{n+1}} (\omega\otimes\mathrm{id})((z\otimes1)u_{z'})
=
(\omega\otimes\mathrm{id})(u_{z'}) = \omega(u_{z'})
\]
thanks to the local normality of $\omega$. Hence we are done.
\end{proof}

We call this $\omega(u) \in C(G)$ the \emph{quantized function} on $G$ associated with $\omega$. Namely, each $\omega \in S^\mathrm{ln}(\mathcal{B}(G))$ defines a quantized function $\omega(u) \in C(G)$, and its projection $\omega(u_n) \in C(G_n)$ should be understood as the restriction of $\omega(u)$ to $G_n$.

\begin{Definition}
The quantized function $\chi(u) \in C(G)$ associated with any $\chi \in K_{-1}^\mathrm{ln}(\vartheta^t)$ is called a \emph{quantized character function} of $G$.
\end{Definition}

This formulation of quantized characters is different from \cite{Sato:JFA19} and \cite{Sato:preprint19}. The next proposition enables one to compute any quantized character functions explicitly by utilizing the represen\-ta\-tion theory of $G_n$.

\begin{Proposition}\label{P5.17} Every quantized character function $\chi(u) = (\chi(u_n)) \in C(G)$ with $\chi \in K_{-1}^\mathrm{ln}(\vartheta^t)$ admits the following expression:
\[
\chi(u_n) = \sum_{z\in\mathfrak{Z}_n} \chi(z)
\frac{1}{\mathrm{Tr}(\pi_{U_z}(\rho_n))}
(\mathrm{Tr}\otimes\mathrm{id})\big((\pi_{U_z}(\rho_n)\otimes1)U_z\big),
\]
where $\rho_n$ denotes the special positive element of $\mathcal{U}(G_n)$ $($see \eqref{Eq5.7}$)$ and $\mathrm{Tr}$ stands for the non-normalized trace on $B(\mathcal{H}_{U_z})$.
\end{Proposition}
\begin{proof}
We have $\chi = \chi(z)\tau_z$ on $zW^*(G_n)$ with a unique normal $(\vartheta_z^t,-1)$-KMS state $\tau_z$ on $zW^*(G_n) \cong B(\mathcal{H}_{U_z})$. Since $\vartheta_z^t = \mathrm{Ad}(z\rho_n^{{\rm i}t})$ (see Section \ref{S5.1}), we see that
\[
\pi_{U_z}(\vartheta_z^t(x)) = \pi_{U_z}(\rho_n)^{{\rm i}t}\pi_{U_z}(x)\pi_{U_z}(\rho_n)^{-{\rm i}t}
\]
holds for every $x \in zW^*(G_n)$. Then, $\Phi_{U_\bullet}\otimes\mathrm{id}$ sends $u_z \in zW^*(G_n)\otimes\mathbb{C}[G_n]$ to $U_z \in B(\mathcal{H}_{U_z})\otimes\mathbb{C}[G_n] \subset (\bigoplus_{z\in\mathfrak{Z}_n} B(\mathcal{H}_{U_z}))\otimes\mathbb{C}[G_n]$. These three facts immediately imply the desired assertion. Actually, we have
\begin{equation}\label{Eq5.16}
\tau_z(u_z) := (\tau_z\otimes\mathrm{id})(u_z) =
\frac{1}{\mathrm{Tr}(\pi_{U_z}(\rho_n))}
(\mathrm{Tr}\otimes\mathrm{id})\big((\pi_{U_z}(\rho_n)\otimes1)U_z\big).
\end{equation}
Hence we are done.
\end{proof}

The above proposition says that the right-hand side of \eqref{Eq5.16} depends only on the unitary equivalence class of unitary representation $U_z$, and moreover, that any quantized character function $\chi(u)$ can explicitly be computed in principle by the data $\chi(z)$, $z \in \mathfrak{Z} := \bigsqcup_{n\geq0}\mathfrak{Z}_n$, and the representation theory of each $G_n$ that is enough to compute \eqref{Eq5.16} explicitly.

\begin{Proposition}
The mapping $\chi \in K_{-1}^\mathrm{ln}(\vartheta^t) \mapsto \chi(u) \in C(G)$ is injective, continuous and open, where $K_{-1}^\mathrm{ln}(\vartheta^t)$ is equipped with the weak$^*$ topology.
\end{Proposition}
\begin{proof}
We first claim that $h((\tau_{z_1}\otimes\mathrm{id})(u_{z_1})^* (\tau_{z_2}\otimes\mathrm{id})(u_{z_2})) = 0$
if $z_1\not=z_2$; otherwise non-zero, where~$h$ is a unique Haar state thanks to \cite[Theorem~1.4.3]{NeshveyevTuset:Book13} together with~\eqref{Eq5.16}.

Assume that $\chi_1(u) = \chi_2(u)$ in $C(G)$ with $\chi_1,\chi_2 \in K_{-1}^\mathrm{lin}(\vartheta^t)$. By definition $\chi_1(u_n) = \chi_2(u_n)$ in $C(G_n)$ for every $n \geq 0$. It follows, from the above claim, that $\chi_1(z) = \chi_2(z)$ holds for every $z \in \mathfrak{Z} = \bigsqcup_{n\geq0}\mathfrak{Z}_n$. Thanks to \cite[Proposition 7.10]{Ueda:Preprint20} (or Theorem \ref{T3.11} of this paper) we conclude that $\chi_1 = \chi_2$.

{\samepage Assume that $\chi_i \to \chi$ in $K_{-1}^\mathrm{ln}(\vartheta^t)$ with respect to the weak$^*$-topology ({\it n.b.}, the weak$^*$-topology is metrizable; see \cite[Corollary 8.3]{Ueda:Preprint20}). Let $n\geq0$ be arbitrarily fixed and choose an increasing sequence $\mathfrak{F}_k$ of finite subsets of $\mathfrak{Z}_n$. We have
\begin{align*}
\Vert \chi_i(u_n) - \chi(u_n)\Vert_{C(G_n)}
&\leqq
\sum_{z \in \mathfrak{Z}_n} |\chi_i(z)-\chi(z)| \\
&\leqq 2\sum_{z\in\mathfrak{F}_k} |\chi_i(z)-\chi(z)| + 2\bigg(1 - \sum_{z\in\mathfrak{F}_k} \chi(z)\bigg),
\end{align*}}
and hence
\[
\varlimsup_{i\to\infty} \Vert \chi_i(u_n) - \chi(u_n)\Vert_{C(G_n)} \leq 2\bigg(1 - \sum_{z\in\mathfrak{F}_k} \chi(z)\bigg) \to 0 \quad \text{as $k \to \infty$}.
\]
Therefore, we conclude that $\Vert \chi_i(u_n) - \chi(u_n)\Vert_{C(G_n)} \to 0$ as $i\to\infty$ for every $n\geq0$. This means that $\chi_i(u) \to \chi(u)$ in $C(G)$ as $i\to\infty$. Hence $\chi \mapsto \chi(u)$ is continuous.

Assume that $\chi_i(u) \to \chi(u)$ as $i\to\infty$, that is, $\Vert\chi_i(u_n)-\chi(u_n)\Vert \to 0$ as $i\to\infty$ for all $n \geq 0$. Set $v_z := (\tau_z\otimes\mathrm{id})(u_z) \in \mathbb{C}[G_n]$ and $h(v_z^* v_z) > 0$ as before. Then, we have
\[
|\chi_i(z)-\chi(z)|\,h(v_z^* v_z) =
\big|(\chi_i(u_n)-\chi(u_n)\,|\,v_z)_h\big| \leq \Vert \chi_i(u_n)-\chi(u_n)\Vert\,h(v_z^* v_z)^{1/2},
\]
where $(a\,|\,b)_h := h(b^* a)$ for $a,b \in \mathbb{C}[G_n]$. Hence $\chi_i(z) \to \chi(z)$ as $i\to\infty$ for all $z \in \mathfrak{Z} = \bigcup_{n\geq0}\mathfrak{Z}_n$. Consequently, $\chi_i \to \chi$ in $K_{-1}^\mathrm{ln}(\vartheta^t)$ thanks to \cite[Proposition 7.10]{Ueda:Preprint20}. Thus, we have shown that $\chi \mapsto \chi(u)$ is open.
\end{proof}

We can discuss the multiplication as well as the adjoint operations for quantized character functions in $C(G)$, which will be translated in terms of $K_{-1}^\mathrm{ln}(\vartheta^t)$ in the next proposition that contains \cite[Theorem 4.1]{Sato:preprint19} essentially.

\begin{Proposition}
\quad
\begin{enumerate}\itemsep=0pt
\item[$1.$] For any pair $\omega_1,\omega_2 \in S^\mathrm{ln}(\mathcal{B}(G))$, there is a unique locally normal state $\omega_1\times\omega_2 \colon \mathcal{B}(G\times G) \allowbreak\to \mathbb{C}$ that agrees with $\omega_1\otimes\omega_2$ on $\mathcal{B}_\infty(G)\otimes\mathcal{B}_\infty(G)$.

\item[$2.$] For any pair $\chi_1,\chi_2 \in K_{-1}^\mathrm{ln}(\vartheta^t)$, a linear functional $\chi_1\cdot\chi_2 := (\chi_1\times\chi_2)\circ\hat{\Delta} \colon \mathcal{B}(G) \to \mathbb{C}$ falls into $K_{-1}^\mathrm{ln}(\vartheta^t)$, and moreover, $(\chi_1\cdot\chi_2)(u) = \chi_1(u)\,\chi_2(u)$ in $C(G)$ holds. In parti\-cu\-lar, $K_{-1}^\mathrm{ln}(\vartheta^t)$ becomes a semigroup with multiplication $(\chi_1,\chi_2) \mapsto \chi_1\cdot\chi_2$ and neutral element~$\hat{\varepsilon}$, and the mapping $\chi \mapsto \chi(u)$ is an injective, semigroup homomorphism from $K_{-1}^\mathrm{ln}(\vartheta^t)$ into~$C(G)$.

\item[$3.$] For any $\chi \in K_{-1}^\mathrm{ln}(\vartheta^t)$, the linear functional $\chi^* := \chi\circ\hat{R}$ falls into $K_{+1}^\mathrm{ln}(\vartheta^t)$ and $\chi^*(u) = \chi(u)^*$ holds. In particular, the family of quantized character functions is closed under the adjoint operation if and only if $\vartheta^t$ is a trivial flow.
\end{enumerate}
 \end{Proposition}

\begin{proof}
(1) Since $\omega_1, \omega_2$ are locally normal, we can define a normal state $[\omega_1\times\omega_2]_n$ on $W^*(G_n\times G_n)\allowbreak =W^*(G_n)\,\bar{\otimes}\,W^*(G_n)$ to be $(\omega_1\!\upharpoonright_{W^*(G_n)})\,\bar{\otimes}\,(\omega_2\!\upharpoonright_{W^*(G_n)})$. By definition, it is clear that $[\omega_1\times\omega_2]_{n+1}$ agrees with $[\omega_1\times\omega_2]_n$ on $W^*(G_n\times G_n)$. Thus, the inductive limit $\omega_1\times\omega_2 := \varinjlim [\omega_1\times\omega_2]_n$ defines the desired locally normal state on $\mathcal{B}(G\times G)$, which clearly agrees with $\omega_1\otimes\omega_2$ on $\mathcal{B}_\infty(G)\otimes\mathcal{B}_\infty(G)$. The uniqueness of $\omega_1\times\omega_2$ follows from its local normality.

(2) Since $\hat{\Delta} = \varinjlim \hat{\Delta}_n$ is clearly locally normal and \eqref{Eq5.11}, one can easily confirm that $\chi_1\cdot\chi_2$ falls into $K_{-1}^\mathrm{ln}(\vartheta^t)$.

Let $n \geq 0$ and $z \in \mathfrak{Z}_n$ be arbitrarily chosen and fixed. By \eqref{Eq5.6} we have
\[
\big(\hat{\Delta}_n\otimes\mathrm{id}\big)(u_z)
=
\big(\Phi_{U_\bullet}^{(2)}{}^{-1}\otimes\mathrm{id}\big)
\bigg(\bigoplus_{(z_1,z_2)}\big((U_{z_1}\times U_{z_2})\big(P^{(z_1,z_2)}_z \otimes 1\big)\big)\bigg),
\]
where $(z_1,z_2)$ runs as long as $U_z$ appears in $U_{z_1}\times U_{z_2}$ as dirrect summands and $P_z^{(z_1,z_2)}$ is the orthogonal projection onto the whole invariant subspace corresponding to $U_z$. By the normality of $[\chi_1\times\chi_2]_n$ we observe that
\begin{align*}
(\chi_1\cdot\chi_2)(u_z)
&=
\big([\chi_1\times\chi_2]_n\circ\hat{\Delta}_n\otimes\mathrm{id}\big)(u_z) \\
&=
\sum_{(z_1,z_2)}
([\chi_1\times\chi_2]_n\otimes\mathrm{id})
\big((\Phi_{U_\bullet}^{(2)}{}^{-1}\otimes\mathrm{id})\big((U_{z_1}\times U_{z_2}) \big(P^{(z_1,z_2)}_z \otimes 1\big)\big)\big) \\
&=
\sum_{(z_1,z_2)} \chi_1(z_1)\,\chi_2(z_2)
\frac{
(\mathrm{Tr}\otimes\mathrm{id})((\pi_{U_{z_1}\times U_{z_2}}(\rho_n)\otimes1)\big((U_{z_1}\times U_{z_2})\big(P^{(z_1,z_2)}_z \otimes 1\big)\big)
}{\mathrm{Tr}(\pi_{U_{z_1}\times U_{z_2}}(\rho_n))}
\end{align*}
and the (operator) norm of each term of the sum in the last expression is not greater than
\[
\chi_1(z_1)\,\chi_2(z_1)\,\frac{\mathrm{Tr}\big(\pi_{U_{z_1}\times U_{z_2}}(\rho_n)P_z^{(z_1,z_2)}\big)}{\mathrm{Tr}(\pi_{U_{z_1}\times U_{z_2}}(\rho_n))}.
\]
Consequently, we obtain that
\begin{align*}
(\chi_1\cdot\chi_2)(u_n)\!
&=\!
\sum_{z\in\mathfrak{Z}_n} (\chi_1\cdot\chi_2)(u_z) \\
&=\!
\sum_{(z_1,z_2)} \!\sum_z \chi_1(z_1)\,\chi_2(z_2)
\frac{
(\mathrm{Tr}\otimes\mathrm{id})((\pi_{U_{z_1}\times U_{z_2}}(\rho_n)\!\otimes\!1)(U_{z_1}\!\times \! U_{z_2})(P^{(z_1,z_2)}_z\! \otimes\! 1))
}{\mathrm{Tr}(\pi_{U_{z_1}\times U_{z_2}}(\rho_n))}
\\
&=\!
\sum_{(z_1,z_2)} \chi_1(z_1)\,\chi_2(z_2)
\frac{
(\mathrm{Tr}\otimes\mathrm{id})((\pi_{U_{z_1}\times U_{z_2}}(\rho_n)\otimes1)(U_{z_1}\times U_{z_2}))
}{\mathrm{Tr}(\pi_{U_{z_1}\times U_{z_2}}(\rho_n))}
\end{align*}
in norm, and the last expression is easily shown to be $\chi_1(u_n)\,\chi_2(u_n)$ since
\[
U_{z_1}\times U_{z_2} = (U_{z_1})_{13}(U_{z_2})_{23}, \qquad
\pi_{U_{z_1}\times U_{z_2}}(\rho_n) = \pi_{U_{z_1}}(\rho_n)\otimes\pi_{U_{z_2}}(\rho_n)
\]
(see \cite[Definitions 1.3.2 and 1.7.1, Theorem 1.4.9]{NeshveyevTuset:Book13}).

It is easy to see that $\hat{\varepsilon}$ falls into $K_{-1}^\mathrm{ln}(\vartheta^t)$ and $\hat{\varepsilon}\cdot\chi = \chi = \chi\cdot\hat{\varepsilon}$ holds for every $\chi \in K_{-1}^\mathrm{ln}(\vartheta^t)$.

(3) By $\hat{R} = \varinjlim \hat{R}_n$ and \eqref{Eq5.9}, one easily sees that $\chi^*$ falls into $K_{+1}^\mathrm{ln}(\vartheta^t)$.

Thanks to \cite[Definitions 1.3.8 and 1.4.5]{NeshveyevTuset:Book13} we observe that
\begin{align*}
U_z^* = (j\otimes\mathrm{id})(U_z^c)
= \pi_{U_z}(\rho_n)^{-1/2}(j\otimes\mathrm{id})(\bar{U}_z)\pi_{U_z}(\rho_n)^{1/2},
\end{align*}
and hence
\begin{align*}
\chi(u_z)^*
&=
\chi(z)\,\tau_z(u_z) \\
&=
\chi(z)\,\frac{1}{\mathrm{Tr}(\pi_{U_z}(\rho_n))}
(\mathrm{Tr}\otimes\mathrm{id})\big((\pi_{U_z}(\rho_n)\otimes1)(j\otimes\mathrm{id})\big(\bar{U}_z\big)\big)
\\
&=
\chi(z)\,\frac{1}{\mathrm{Tr}\big(\pi_{\bar{U}_z}\big(\hat{R}(\rho_n)\big)\big)} (\mathrm{Tr}\otimes\mathrm{id})\big(\big(\pi_{\bar{U}_z}\big(\hat{R}(\rho_n)\big)\otimes1\big)\big(\bar{U}_z\big)\big)
\\
&=
\chi^*(\bar{z})\,\frac{1}{\mathrm{Tr}\big(\pi_{\bar{U}_z}\big(\rho_n^{-1}\big)\big)}
(\mathrm{Tr}\otimes\mathrm{id})\big(\big(\pi_{\bar{U}_z}\big(\rho_n^{-1}\big)\otimes1\big)\big(\bar{U}_z\big)\big)
\\
&=
\chi^*(\bar{z})\,\frac{1}{\mathrm{Tr}\big(\pi_{U_{\bar{z}}}\big(\rho_n^{-1}\big)\big)}
(\mathrm{Tr}\otimes\mathrm{id})\big(\big(\pi_{U_{\bar{z}}}\big(\rho_n^{-1}\big)\otimes1\big)(U_{\bar{z}})\big)
\end{align*}
by \cite[Example 2.2.23]{NeshveyevTuset:Book13} and by \eqref{Eq5.10} and \eqref{Eq5.14}. The uniqueness of normal $(\vartheta_{\bar{z}}^t,+1)$-KMS state on $\bar{z}W^*(G_n) \cong B(\mathcal{H}_{\bar{z}})$ implies that $\chi(u_z)^* = \chi^*(u_{\bar{z}})$. Thus we conclude that
\[
\chi(u_n)^* = \sum_{z\in\mathfrak{Z}_n} \chi(u_z)^* = \sum_{z\in\mathfrak{Z}_n} \chi^*(u_{\bar{z}}) = \chi^*(u_n).
\]
The other assertion is trivial now.
\end{proof}

The discussion in this section clearly works for single compact quantum groups. It is also (probably) possible to give a similar formulation of spherical functions of $G < G\times G$. This will be discussed elsewhere in a wider setup.

\subsection[A concrete example: U\_q(infty)]{A concrete example: $\boldsymbol{\mathrm{U}_q(\infty)}$} \label{S5.4}

We will examine a concrete example, which was already investigated by Sato \cite{Sato:JFA19,Sato:preprint19}, from our viewpoint. In what follows, we choose and fix $0 < q < 1$ throughout.

\subsubsection[Quantum unitary group U\_q(n)]
{Quantum unitary group $\boldsymbol{\mathrm{U}_q(n)}$} 

We begin by making it clear what the quantum unitary group $\mathrm{U}_q(n)$ we employ is.

Let $U_q\mathfrak{gl}(n)$ be the Drinfeld--Jimbo quantized universal enveloping algebra associated with the general linear Lie algebra $\mathfrak{gl}(n)$. Namely, it is the unital algebra generated by the Cartan-type elements $K_i^{\pm1}$, $1 \leq i \leq n$ and the elements $x_j$, $y_j$, $1 \leq j \leq n-1$, corresponding to the standard positive and negative simple root vectors of $\mathfrak{gl}(n)$ with the relations given in \cite[Section~1.2]{OblomkovStokman:ComposMath05}.
The algebra $U_q\mathfrak{gl}(n)$ becomes a Hopf $*$-algebra
with the following structure:
\begin{itemize}\itemsep=0pt
\item ($*$-structure) $\big(K_i^{\pm1}\big)^* = K_i^{\pm1}$, $x_j^* = q^{-1} y_j K_jK_{j+1}^{-1}$, $y_j^* = qK_j^{-1}K_{j+1} x_j$,
\item (comultiplication) $\hat{\Delta}_{n,q}\big(K_i^{\pm1}\big) = K_i^{\pm1}\otimes K_i^{\pm1}$, $\hat{\Delta}_{n,q}(x_j) = x_j\otimes1+K_j K_{j+1}^{-1}\otimes x_j$, $\hat{\Delta}_{n,q}(y_j) = y_j\otimes K_j^{-1} K_{j+1} + 1\otimes y_j$,
\item (antipode) $\hat{S}_{n,q}\big(K_i^{\pm1}\big) = K_i^{\mp1}$, $\hat{S}_{n,q}(x_j) = -K_j^{-1}K_{j+1} x_j$, $\hat{S}_{n,q}(y_j) = -y_j K_j K_{j+1}^{-1}$,
\item (counit) $\hat{\varepsilon}_{n,q}\big(K_i^{\pm1}\big) = 1$, $\hat{\varepsilon}_{n,q}(x_j) = \hat{\varepsilon}_{n,q}(y_j) = 0$.
\end{itemize}
See, e.g., \cite[Section 1.2]{OblomkovStokman:ComposMath05}. Applying \cite[Theorem 2.3.13]{NeshveyevTuset:Book13} to this Hopf $*$-algebra with its finite-dimensional type $1$ $*$-representations, we obtain the compact quantum group $\mathrm{U}_q(n)$ with Hopf $*$-algebra $(\mathbb{C}[\mathrm{U}_q(n)],\Delta_n,S_n,\varepsilon_n)$, which is exactly the coordinate ring $A(\mathrm{U}_q(n))$ of the so-called quantum unitary group of rank $n$ discussed, e.g., in~\cite[Section 1]{Noumi:AdvMath96}, \cite{NoumiYamadaMimachi:JpnJMath93} (see \cite[Section 2]{Dijkhuizen:ActaApplMath96} for its concise review) in the context of $q$-analysis. (Note that the choice of generators of $U_q\mathfrak{gl}(n)$ here follows \cite{Noumi:AdvMath96} (rather than \cite{NoumiYamadaMimachi:JpnJMath93}) with $K_i^{\pm1} := q^{\pm\epsilon_i}$, $x_j := e_j$, $y_j := f_j$.)

As in Section \ref{S5.1} we then obtain the Hopf $*$-algebra $\big(\mathcal{U}(\mathrm{U}_q(n)),\hat{\Delta}_n,\hat{S}_n,\hat{\varepsilon}_n\big)$ and the group $W^*$-algebra $W^*(\mathrm{U}_q(n)) \subset \mathcal{U}(\mathrm{U}_q(n))$. Note that there is a canonical embedding $U_q\mathfrak{gl}(n) \hookrightarrow (U_q\mathfrak{gl}(n))^{**} \hookrightarrow \mathbb{C}[\mathrm{U}_q(n)]^* = \mathcal{U}(\mathrm{U}_q(n))$ as Hopf $*$-algebras. Thus, we may and do regard $U_q\mathfrak{gl}(n)$ as a Hopf $*$-subalgebra of $\mathcal{U}(\mathrm{U}_q(n))$, and hence obtain
\begin{equation}\label{Eq5.17}
W^*(\mathrm{U}_q(n)) \subset \mathcal{U}(\mathrm{U}_q(n)) \supset U_q\mathfrak{gl}(n)
\end{equation}
as Hopf $*$-algebras, that is, $\hat{\Delta}_{n,q}$, $\hat{S}_{n,q}$, $\hat{\varepsilon}_{n,q}$ are just the restrictions of $\hat{\Delta}_n$, $\hat{S}_n$, $\hat{\varepsilon}_n$, respectively. Hence we will use only the symbols $\hat{\Delta}_n$, $\hat{S}_n$, $\hat{\varepsilon}_n$ via the above embedding.

An important thing is that the representation tensor categories (with conjugates) of $U_q\mathfrak{gl}(n)$ and $\mathrm{U}_q(n)$ are naturally identified with each other; see \cite[Definition 2.4.5]{NeshveyevTuset:Book13}. Hence the algebras $W^*(\mathrm{U}_q(n)) \subset \mathcal{U}(\mathrm{U}_q(n))$ can directly be constructed from $U_q\mathfrak{gl}(n)$ (or more precisely, its representation category).
Another important fact is that the special positive element $\rho_n \in \mathcal{U}(\mathrm{U}_q(n))$ (see \eqref{Eq5.7}) is given by
\begin{equation}\label{Eq5.18}
\rho_n = K_1^{-n+1} K_2^{-n+3} \cdots K_n^{n-1} \in U_q\mathfrak{gl}(n).
\end{equation}
See \cite[Remark 1.7.7 and Proposition 2.4.10]{NeshveyevTuset:Book13} and the proof of \cite[Lemma 5.1]{OnnStokman:IMRP06} (also see \cite[formula~(1.15)]{Noumi:AdvMath96}).
These facts show that all essential ingredients used in the present theory can directly be calculated in terms of $U_q\mathfrak{gl}(n)$, though our theory is developed based on group $W^*$-algebras $W^*(\mathrm{U}_q(n))$.

The minimal projections $\mathfrak{Z}_n$ of $W^*(\mathrm{U}_q(n))$ are known to be labeled by the signatures $\mathbb{S}_n$ of length $n$, and thus we will use the signatures $\mathbb{S}_n$ as an index set instead of $\mathfrak{Z}_n$ itself. Hence the natural mapping $\lambda \in \mathbb{S}_n \mapsto z_\lambda \in \mathfrak{Z}_n$ will be used in what follows.

\subsubsection[Infinite-dimensional quantum unitary group U\_q(infty)]
{Infinite-dimensional quantum unitary group $\boldsymbol{\mathrm{U}_q(\infty)}$} 

Let us consider the standard embedding $U_q\mathfrak{gl}(n) \hookrightarrow U_q\mathfrak{gl}(n+1)$ by sending the first $n$ Cartan-type elements $K_i^{\pm1}$'s and the first $n-1$ positive and negative simple root vectors $x_j$, $y_j$'s to the corresponding ones (with the same indices).
By the construction of $\mathbb{C}[\mathrm{U}_q(n)]$, the restriction map from $(U_q\mathfrak{gl}(n+1))^*$ to $(U_q\mathfrak{gl}(n))^*$ via $U_q\mathfrak{gl}(n) \hookrightarrow U_q\mathfrak{gl}(n+1)$ induces a surjective Hopf $*$-algebra morphism $\mathbb{C}[\mathrm{U}_q(n+1)] \twoheadrightarrow \mathbb{C}[\mathrm{U}_q(n)]$.
Hence we can take the inductive limit $\mathrm{U}_q(\infty) = \varinjlim \mathrm{U}_q(n)$ in the sense of Section \ref{S5.2}. Namely, we have
\[
 \mathcal{B}(\mathrm{U}_q(\infty)) \hookleftarrow \mathcal{B}_\infty(\mathrm{U}_q(\infty)) \subset \mathcal{U}(\mathrm{U}_q(\infty)) \supset U_q\mathfrak{gl}(\infty) := \varinjlim U_q\mathfrak{gl}(n),
\]
where the first algebra is a $C^*$-algebraic inductive limit, while the last three algebras are algebraic inductive limits. We also have $\mathcal{B}(\mathrm{U}_q(\infty)\times\mathrm{U}_q(\infty)) := \varinjlim W^*(\mathrm{U}_q(n)\times\mathrm{U}_q(n))$ and
\begin{gather*}
\hat{\Delta} = \varinjlim \hat{\Delta}_n \colon\ \mathcal{B}(\mathrm{U}_q(\infty)) \to \mathcal{B}(\mathrm{U}_q(\infty)\times\mathrm{U}_q(\infty)), \\
 \hat{R} = \varinjlim \hat{R}_n, \qquad \vartheta^t = \varinjlim \vartheta_n^t, \qquad \hat{\varepsilon} = \varinjlim \hat{\varepsilon}_n
\end{gather*}
as in Section \ref{S5.2}.

\subsubsection[Diagonal subgroup of U\_q(infty)]
{Diagonal subgroup of $\boldsymbol{\mathrm{U}_q(\infty)}$} 

Let $U_q\mathfrak{h}(n)$ be the unital $*$-subalgebra generated by the $K_i^{\pm1}$, $1 \leq i \leq n$. Clearly, it is a~com\-mutative, co-commutative Hopf $*$-subalgebra of $U_q\mathfrak{gl}(n)$. Let $\mathbb{C}[T_n]$ be the unital $*$-algebra consisting of all the restrictions of elements of $\mathbb{C}[\mathrm{U}_q(n)]$ ($\subset (U_q\mathfrak{gl}(n))^*$) to $U_q\mathfrak{h}(n)$ so that we have a surjective Hopf $*$-algebra homomorphism $|_{T_n} \colon \mathbb{C}[\mathrm{U}_q(n)]\twoheadrightarrow\mathbb{C}[T_n]$ that extends to the $C^*$-level $|_{T_n} \colon C(\mathrm{U}_q(n)) \twoheadrightarrow C(T_n) = C(\mathbb{T}^n)$ sending the canonical generators $t_{ij}$ and $\det_q^{-1}$ to $\delta_{i,j}t_i$ and $(t_1\cdots t_n)^{-1}$, respectively, \big(or determined by the paring $\big\langle K_1^{m_1}\cdots K_n^{m_n}, t_1^{m'_1}\cdots t_n^{m'_n}\big\rangle = q^{\sum_i m_i m'_i}$\big), where $C(\mathbb{T}^n)$ denotes the algebra of continuous functions on the $n$-torus $\mathbb{T}^n$ and $t_i \in C(\mathbb{T}^n)$ does the $i$th coordinate function on $\mathbb{T}^n$.

It is well known that the following diagram commutes:
\[
\xymatrix{
C(\mathrm{U}_q(n+1)) \ar@{->}[d]_{r_n^{n+1}} \ar@{->}[r]^{\quad |_{T_{n+1}}} & C\big(\mathbb{T}^{n+1}\big) \ar@{->}[d]^{f(t_1,\dots,t_n,t_{n+1}) \mapsto f(t_1,\dots,t_n,1)} \\
C(\mathrm{U}_q(n)) \ar@{->}[r]_{\quad |_{T_n}} \ar@{}[ur]|{\circlearrowright} & C(\mathbb{T}^n).
}
\]
Thus there is a unital $*$-homomorphism $|_{T_{\infty}}$ from the $\sigma$-$C^*$-algebra $C(\mathrm{U}_q(\infty)) = \varprojlim C(\mathrm{U}_q(n))$ to the $\sigma$-$C^*$-algebra $C(\mathbb{T}^\infty) = \varprojlim C(\mathbb{T}^n)$ by $f = (f_n) \mapsto f|_{T_\infty} := (f_n|_{T_n})$. We call this unital $*$-homomorphism $|_{T_\infty} \colon C(\mathrm{U}_q(\infty)) \to C(\mathbb{T}^\infty)$ the \emph{diagonal subgroup} of $\mathrm{U}_q(\infty)$.

\subsubsection[Gorin's q-Schur generating functions and quantized character functions]
{Gorin's $\boldsymbol q$-Schur generating functions and quantized character functions}

Theorem \ref{T5.13} shows that all the unitary equivalence classes of spherical unitary representations of $\mathrm{U}_q(\infty)$ and all the locally normal $(\vartheta^t,-1)$-KMS states $K_{-1}^\mathrm{ln}(\vartheta^t)$ are in one-to-one correspondence. Thus, we can label such a unitary equivalence class with a unique element of $K_{-1}^\mathrm{ln}(\vartheta^t)$. This certainly justifies Sato's formulation of asymptotic representation theory of $\mathrm{U}_q(\infty)$ as mentioned before.

Consider any quantized character function $\chi(u) \in C(\mathrm{U}_q(\infty))$ with $\chi \in K_{-1}^\mathrm{ln}(\vartheta^t)$. Then $\chi(u)|_{T_\infty} = (\chi(u_n)|_{T_n}) \in C(\mathbb{T}^\infty)$ becomes
\[
\chi(u_n)|_{T_n} = \sum_{\lambda \in \mathbb{S}_n} \chi(z_\lambda)\,\frac{1}{\mathrm{Tr}(\pi_{U_\lambda}(\rho_n))} (\mathrm{Tr}\otimes\mathrm{id})\big((\pi_{U_\lambda}(\rho_n)\otimes1)U_\lambda\big)\bigg|_{T_n}
\]
as in Proposition \ref{P5.17}, where $U_\lambda$ is an irreducible unitary representation associated with $\lambda \in \mathbb{S}_n$.

We may think that $U_\lambda$ is obtained from a (unique up to unitary equivalence) irreducible, type $1$ $*$-representation $\pi_\lambda \colon U_q\mathfrak{gl}(n) \curvearrowright \mathcal{H}_\lambda$ of label $\lambda$ via the tensor product duality $L(U_q\mathfrak{gl}(n),B(\mathcal{H}_\lambda)) \allowbreak\cong B(\mathcal{H}_\lambda)\otimes(U_q\mathfrak{gl}(n))^*$ (see the proof of Lemma \ref{L5.14} for the notation), that is, $(\mathrm{id}\otimes X)(U_\lambda)\allowbreak = \pi_\lambda(X)$ holds for any $X \in U_q\mathfrak{gl}(n)$, where $X$ is regarded as an element of the second dual of~$U_q\mathfrak{gl}(n)$ naturally. By \eqref{Eq5.18} we have
\begin{equation}\label{Eq5.19}
\pi_{U_\lambda}(\rho_n) = \pi_\lambda\big(K_1^{-n+1} K_2^{-n+3} \cdots K_n^{n-1}\big),
\end{equation}
and similarly to the calculations in \cite[Lemma 5.1, equation~(5.12)]{OnnStokman:IMRP06} or \cite[equations~(3.20)--(3.23)]{NoumiYamadaMimachi:JpnJMath93}, we see that
\begin{align*}
\frac{1}{\mathrm{Tr}(\pi_{U_\lambda}(\rho_n))} (\mathrm{Tr}\otimes\mathrm{id})\big((\pi_{U_\lambda}(\rho_n)\otimes1)U_\lambda\big)\bigg|_{T_n}
&=
\frac{s_\lambda\big(q^{-n+1}t_1,q^{-n+3}t_2,\dots,q^{n-1}t_n\big)}{s_\lambda\big(q^{-n+1},q^{-n+3},\dots,q^{n-1}\big)} \\
&=
\frac{s_\lambda\big(t_1,q^2 t_2,\dots,q^{2(n-1)}t_n\big)}{s_\lambda\big(1,q^2,\dots,q^{2(n-1)}\big)}
\end{align*}
with rational Schur functions $s_\lambda$.
Consequently, we have the following statement:

\begin{Proposition}
The restriction $\chi(u)|_{T_\infty} = (\chi(u_n)|_{T_n})$ of a quantized character function~$\chi(u)$ with $\chi \in K_{-1}^\mathrm{ln}(\vartheta^t)$ to the diagonal subgroup $T_\infty$ is given by Gorin's $q^{-2}$-Schur generating functions~$\mathcal{S}_{q^{-2}}$ {\rm(}see {\rm \cite[equation (12)]{Gorin:AdvMath12}}{\rm):}
\[
\chi(u_n)|_{T_n} = \mathcal{S}_{q^{-2}}\big(t_1,q^2 t_2,\dots,q^{2(n-1)}t_n;\nu[\chi]_n\big),
\]
where $\nu[\chi]_n(\lambda) := \chi(z_\lambda)$, $\lambda \in \mathbb{S}_n$ $($see {\rm \cite[Proposition~7.10]{Ueda:Preprint20}} for this notation$)$. On the other hand, Gorin's $q^{2}$-Schur generating functions~$\mathcal{S}_{q^{2}}$ also appear as
\[
\chi^*(u_n)|_{T_n} \big(= \chi(u_n)^*|_{T_n}\big) = \mathcal{S}_{q^{2}}\big(t_1,q^{-2} t_2,\dots,q^{-2(n-1)}t_n;\nu[\chi^*]_n\big).
\]
\end{Proposition}

We emphasize that some of the results in \cite[Section 4]{Gorin:AdvMath12} based on symmetric functions immediately follow from the general theory we have developed so far. The result shows that Gorin's setup fits the co-opposite version of spherical unitary representations of our $\mathrm{U}_q(\infty)$ or equivalently the locally normal $(\vartheta^t,+1)$-KMS states $K_{+1}^\mathrm{ln}(\vartheta^t)$; see Remarks \ref{R5.6}(3) and the short sentense just after Corollary \ref{C5.11}. However, this difference only causes the rather minor transition from $q$ to $q^{-1}$.

We have
\begin{align*}
\omega \in S^\mathrm{ln}(\mathcal{S}(\vartheta^t,-1))
&\longleftrightarrow
\chi_\omega := E_\bullet^*(\omega) \in K_{-1}^\mathrm{ln}(\vartheta^t) \\
&\longleftrightarrow \chi_\omega(u) \in C(\mathrm{U}_q(\infty)) \\
&\longleftrightarrow \chi_\omega(u)|_{T_\infty} \in C(\mathbb{T}^\infty)
\end{align*}
(see Theorem \ref{T3.11} for the first line). The first two lines are trivially in one-to-one correspondence, and that so is the last line too needs the orthogonality of matrix elements of $U_\lambda$'s with respect to the Haar state of each $\mathrm{U}_q(n)$. We are interested in finding any representation-theoretic interpretation of $S^\mathrm{ln}(\mathcal{S}(\vartheta^t,-1))$. A very first step to this question will be given below.

\subsubsection[The representation operator system S(vartheta\textasciicircum{}t,-1) as left-module over the algebra Sigma(vartheta\textasciicircum{}t,-1)]
{The representation operator system $\boldsymbol{\mathcal{S}(\vartheta^t,-1)}$ as left-module \\over the algebra $\boldsymbol{\Sigma(\vartheta^t,-1)}$} 

We first introduce a natural multiplicative structure on the inductive sequence $W^*(\mathrm{U}_q(n))$.

Let $m, n$ be arbitrarily natural numbers, and the embedding $U_q\mathfrak{gl}(m)\otimes U_q\mathfrak{gl}(n) \hookrightarrow U_q\mathfrak{gl}(m+n)$ is given in the standard way, e.g., in~\cite[Section 5.2]{OnnStokman:IMRP06}. Namely, the embedding is given by
\begin{itemize}\itemsep=1pt
\item $K_{i}\otimes 1 \mapsto K_{i}$ ($1 \leq i \leq m$) and $1\otimes K_{i} \mapsto K_{m+i}$ ($1 \leq i \leq n$),
\item $x_j \otimes 1 \mapsto x_j$ ($1 \leq j \leq m-1$) and $1\otimes x_j \mapsto x_{j+m}$ ($1\leq j \leq n-1$),
\item $y_j \otimes 1 \mapsto y_j$ ($1 \leq j \leq m-1$) and $1\otimes y_j \mapsto y_{j+m}$ ($1\leq j \leq n-1$).
\end{itemize}
This defines a surjective Hopf $*$-algebra homomorphism $\mathbb{C}[\mathrm{U}_q(m+n)] \twoheadrightarrow \mathbb{C}[\mathrm{U}_q(m)]\otimes\mathbb{C}[\mathrm{U}_q(n)] =: \mathbb{C}[\mathrm{U}_q(m)\times\mathrm{U}_q(n)]$. (See, e.g., \cite[Section 5]{DijkhuizenStokman:PRIMS99}, \cite[Section 5.2]{OnnStokman:IMRP06}.) Then we obtain $\mathcal{U}(\mathrm{U}_q(m)\times\mathrm{U}_q(n)) \hookrightarrow \mathcal{U}(\mathrm{U}_q(m+n))$ as before. Taking the bounded part of these $*$-algebras we obtain the desired left multiplicative structure $\iota_{m,n} \colon W^*(\mathrm{U}_q(m))\,\bar{\otimes}\,W^*(\mathrm{U}_q(n)) = W^*(\mathrm{U}_q(m)\times\mathrm{U}_q(n)) \to W^*(\mathrm{U}_q(m+n))$. Note that requirement $(m3)$ of left multiplicative structure is checked in the same way as in the construction of $\mathcal{B}(G)$ with $G = \varinjlim G_n$; see Section \ref{S5.2.1}. Also requirement~$(L)$ is the consequence of embedding $U_q\mathfrak{gl}(n) \hookrightarrow U_q\mathfrak{gl}(n+1)$ taking the ``upper diagonal corner''. (Remark that item $(R)$ would be used instead for the ``lower diagonal corner'' embedding.)

As before (see the discussion around \eqref{Eq5.19}), we have
\begin{align}
\dim_q(\lambda)& := \mathrm{Tr}(\pi_{U_\lambda}(\rho_n)) = \mathrm{Tr}\big(\pi_\lambda\big(K_1^{-n+1} K_2^{-n+3} \cdots K_n^{n-1}\big)\big) \nonumber
\\
&= s_\lambda\big(q^{-n+1},q^{-n+3},\dots,q^{n-1}\big),
\label{Eq5.20}
\end{align}
which is called the quantum dimension of $\lambda$. We set $\hat{z}_\lambda := \dim_q(\lambda)^{-1}z_\lambda$ as in Section \ref{S3.5}, and will compute $E_{m+n}(\iota_{m,n}(\hat{z}_\mu\otimes\hat{z}_\nu))$ with $(\mu,\nu) \in \mathbb{S}_m\times\mathbb{S}_n$. Since the branching rule for this quantum case is known to be the same as for the classical case (see, e.g., \cite[Proposition~5.4]{DijkhuizenStokman:PRIMS99} with the help of ``differentiation''), we have
\begin{gather*}
(\pi_\lambda \colon U_q\mathfrak{gl}(m)\otimes U_q\mathfrak{gl}(n) \curvearrowright \mathcal{H}_\lambda)
\\ \qquad{}
\cong \bigoplus_{(\mu,\nu) \in \mathbb{S}_m\times\mathbb{S}_n} (\pi_\mu\otimes\pi_\nu \colon U_q\mathfrak{gl}(m)\otimes U_q\mathfrak{gl}(n) \curvearrowright \mathcal{H}_\mu\otimes\mathcal{H}_\nu)^{\oplus c(\lambda \mid \mu,\nu)}
\end{gather*}
(up to unitary equivalence), where $c(\lambda\mid\mu,\nu)$ is the same as in Section \ref{S3.5}. By \eqref{Eq5.19}, \eqref{Eq5.20} and the embedding of the $K_i$'s from $U_q\mathfrak{gl}(m)\otimes U_q\mathfrak{gl}(n)$ into $U_q\mathfrak{gl}(m+n)$, we have
\begin{align*}
&z_\lambda E_{m+n}(\iota_{m,n}(z_\mu\otimes z_\nu)) \\
&\ \quad=\!
\frac{1}{\dim_q(\lambda)}\,c(\lambda\mid\mu,\nu)\,
\mathrm{Tr}\big(\pi_\mu\big(K_1^{-(m+n)+1}\cdots K_m^{m-n-1}\big)\otimes\pi_\nu\big(K_{m+1}^{m-n+1}\cdots K_{m+n}^{m+n-1}\big)\big)\,z_\lambda \\
&\ \quad=\!
c(\lambda\! \mid \!\mu,\nu)\,\frac{s_\mu(q^{-(m+n)+1},q^{-(m+n)+3},\dots,q^{m-n-1})\,
s_\nu(q^{m-n+1},q^{m-n+3},\dots,q^{m+n-1})}{\dim_q(\lambda)}\,z_\lambda \\
&\ \quad=\!
c(\lambda \mid \mu,\nu)\,q^{[\mu,\nu]}\frac{\dim_q(\mu)\,\dim_q(\nu)}{\dim_q(\lambda)}\,z_\lambda,
\end{align*}
where we define
\begin{equation*}
[\mu,\nu] := \det\begin{bmatrix} \ell(\mu) & \ell(\nu) \\ |\mu| & |\nu| \end{bmatrix}
= \ell(\mu)|\lambda| - \ell(\lambda)|\mu|
\end{equation*}
with $\ell(\mu) := m$, $\ell(\nu) := n$ and $|\lambda| := \sum_{i=1}^n\lambda_i$, $|\mu| := \sum_{i=1}^n \mu_i$ as usual.
Thus, we obtain
\begin{equation*}
E_{m+n}(\iota_{m,n}(\hat{z}_\mu\otimes\hat{z}_\nu)) = \sum_{\lambda \in \mathbb{S}_{m+n}} c_q(\lambda \mid \mu,\nu)\,\hat{z}_\lambda \qquad \text{with} \quad c_q(\lambda\mid\mu,\nu) := c(\lambda \mid \mu,\nu)\,q^{[\mu,\nu]}.
\end{equation*}
We conclude the following:

\begin{Proposition}\label{P5.21} The multiplication on $\Sigma(\vartheta^t,-1)$ is determined by
\[
\hat{z}_\mu\cdot\hat{z}_\nu = \sum_{\lambda\in\mathbb{S}_{m+n}} c_q(\lambda\mid\mu,\nu)\,\hat{z}_\lambda = q^{2[\mu,\nu]}\,\hat{z}_\nu\cdot\hat{z}_\mu, \qquad \mu \in \mathbb{S}_m,\quad \nu \in \mathbb{S}_n,
\]
and accordingly, the left $\Sigma(\vartheta^t,-1)$-module multiplications on $\ddot{\mathcal{S}}(\vartheta^t,-1)$ and $\mathcal{S}(\vartheta^t,-1)$ enjoy that
\begin{gather*}
\hat{z}_\mu\cdot\ddot{\Theta}_{\infty,n}(\hat{z}_\nu) = q^{2[\mu,\nu]}\,\hat{z}_\nu\cdot\ddot{\Theta}_{\infty,m}(\hat{z}_\mu), \\
\hat{z}_\mu\cdot\Theta_{\infty,n}(\hat{z}_\nu) = q^{2[\mu,\nu]}\,\hat{z}_\nu\cdot\Theta_{\infty,m}(\hat{z}_\mu)
\end{gather*}
for any pair $\mu \in \mathbb{S}_m$, $\nu \in \mathbb{S}_n$, respectively.
\end{Proposition}

This proposition explains that $\Sigma(\vartheta^t,-1)$ never becomes commutative. This is a completely different phenomenon from $\mathrm{U}(\infty)$; compare this with Section~\ref{S3.5}.

Since $\dim_q((k)) = 1$ we have $1_1 = \sum_{k \in \mathbb{Z}} z_{(k)} = \sum_{k \in \mathbb{Z}} \hat{z}_{(k)} \in \mathcal{Z}(W^*(\mathrm{U}_q(1))) = \mathcal{Z}(W^*(\mathbb{T})) \subset \Sigma(\vartheta^t,-1)$. We observe that for each $\mu \in \mathbb{S}_n$,
\begin{align*}
E_{n+1}(z_\mu)
&=
z_\mu\cdot1_1 = \dim_q(\mu)\,\hat{z}_\mu\cdot1_1 \\
&=
\dim_q(\mu)\,\sum_{k\in\mathbb{Z}}\hat{z}_\mu\cdot\hat{z}_{(k)} \\
&=
\dim_q(\mu)\,\sum_{\lambda \in \mathbb{S}_{n+1}} \sum_{k \in \mathbb{Z}} c_q(\lambda\mid\mu,(k))\,\hat{z}_\lambda \\
&=
\sum_{\lambda \in \mathbb{S}_{n+1}; \lambda \succ \mu} q^{[\lambda,(|\lambda|-|\mu|)]}\frac{\dim_q(\mu)}{\dim_q(\lambda)}\,z_\lambda \\
&=
\sum_{\lambda \in \mathbb{S}_{n+1}; \lambda \succ \mu} q^{n|\lambda|-(n+1)|\mu|}\frac{\dim_q(\mu)}{\dim_q(\lambda)}\,z_\lambda,
\end{align*}
where $\lambda \succ \mu$ denotes the interlace relation, i.e., $\lambda_i \geq \mu_i \geq \lambda_{i+1}$, $i=1,2,\dots,n$.
Here, we have used Theorem \ref{T3.14} and computed Littlewood--Richardson coefficients $c(\lambda\mid\mu,(k))$. The above computation is performed in $\Sigma(\vartheta^t,-1)$, and by \eqref{Eq3.6} we arrive at the following explicit description of the link $\kappa_{(\vartheta^t,-1)}$ on the Gelfand--Tsetlin graph:
\begin{equation}\label{Eq5.23}
\kappa_{(\vartheta^t,-1)}(\lambda,\mu) =
\begin{cases}
q^{n|\lambda|-(n+1)|\mu|}\dfrac{\dim_q(\mu)}{\dim_q(\lambda)}, & \lambda \succ \mu, \\
0, & \text{otherwise}.
\end{cases}
\end{equation}

The above computation is purely quantum group representation-theoretic, and its keys are only \eqref{Eq5.18} and the branching rule for $U_q\mathfrak{gl}(m)\otimes U_q\mathfrak{gl}(n) \hookrightarrow U_q\mathfrak{gl}(m+n)$.
The actual computation for the classification of irreducible spherical unitary representations of $\mathrm{U}_q(\infty) < \mathrm{U}_q(\infty)\times\mathrm{U}_q(\infty)$ comes down to the analysis on the link \eqref{Eq5.23} due to Gorin \cite{Gorin:AdvMath12}.

We have the natural map $\Gamma \colon \mathcal{S}(\vartheta^t,-1) \to C\big(\mathrm{ex}\big(S^\mathrm{ln}(\mathcal{S}(\vartheta^t,-1))\big)\big) \cong C\big(\mathrm{ex}\big(K_{-1}^\mathrm{ln}(\vartheta^t)\big)\big)$ defined by $\Gamma(s)(\chi) := \chi(f)$ for any $s = \Theta_{\infty,n}(f) \in \mathcal{S}(\vartheta^t,-1)$ with $f \in \mathcal{Z}(W^*(\mathrm{U}_q(n)))$ and $\chi \in \mathrm{ex}\big(K_{-1}^\mathrm{ln}(\vartheta^t)\big)$, where $C\big(\mathrm{ex}\big(S^\mathrm{ln}(\mathcal{S}(\vartheta^t,-1))\big)\big) \cong C\big(\mathrm{ex}\big(K_{-1}^\mathrm{ln}(\vartheta^t)\big)\big)$ is induced by Theorem \ref{T3.11}. There are natural questions on this map along the line of~\cite{Olshanski:AdvMath16}, though the map is not multiplicative in the present setup unlike the case of~$\mathrm{U}(\infty)$. We thank one of the referees for asking this kind of insightful question.

\subsubsection{Smooth quantized characters} \label{S5.4.6}

The short discussion below is just a remark to give a basis for future research.

Let $\langle\,\cdot\,,\,\cdot\,\rangle_n \colon \mathbb{C}[\mathrm{U}_q(n)] \times \mathcal{U}(\mathrm{U}_q(n)) \to \mathbb{C}$ be the dual pairing. Observe (see \eqref{Eq5.16}) that $\chi(u_\lambda)$ falls in $\mathbb{C}[\mathrm{U}_q(n)]$ and
\[
\langle \chi(u_\lambda),X\rangle_n
=
\chi(z_\lambda)\frac{\mathrm{Tr}(\pi_{U_\lambda}(\rho_n)\pi_{U_\lambda}(X))}{\mathrm{Tr}(\pi_{U_\lambda}(\rho_n))}
=
\chi(z_\lambda X)
\]
for all $\chi \in K_{-1}^\mathrm{ln}(\vartheta^t)$ and $X \in \mathcal{U}(\mathrm{U}_q(n))$, where $z_\lambda X$ sits in $z_\lambda\,\mathcal{U}(\mathrm{U}_q(n)) = z_\lambda W^*(\mathrm{U}_q(n)) \subset W^*(\mathrm{U}_q(\infty))$. Thus, we are interested in determining the intermediate space $\mathfrak{S}[\mathrm{U}_q(n)]$ between $\mathbb{C}[\mathrm{U}_q(n)] \subset C(\mathrm{U}_q(n))$, to which all the linear functionals $a \in \mathbb{C}[\mathrm{U}_q(n)] \mapsto \langle a,X\rangle_n \in \mathbb{C}$ with arbitrary $X \in U_q\mathfrak{gl}(n)$ ($\subset \mathcal{U}(\mathrm{U}_q(n))$ via \eqref{Eq5.17}) can ``extend''.

We should call $\chi \in K_{-1}^\mathrm{ln}(\vartheta^t)$ a \emph{smooth quantized character}, if
$\chi(u_n) \in \mathfrak{S}[\mathrm{U}_q(n)]$ holds for every $n \geq 1$. If this is the case, then the $\chi$ naturally ``extends'' to $U_q\mathfrak{gl}(\infty)$ as $X \mapsto \lim_{n\to\infty}\langle \chi(u_n),X\rangle_n$, since
\[
\langle \chi(u_{n+1}), X\rangle_{n+1} = \langle \chi(u_n), X\rangle_n
\]
holds for every $X \in U_q\mathfrak{gl}(n) \subset U_q\mathfrak{gl}(\infty) \subset \mathcal{U}(\mathrm{U}_q(\infty))$ by Proposition \ref{P5.15}. See \cite[Section 3.1]{BorodinBufetov:DukeMathJ14} for a classical counterpart. A condition guaranteeing that a given $\chi \in K_{-1}^\mathrm{ln}(\vartheta^t)$ is smooth is an interesting question. Remark that
\[
\frac{\mathrm{Tr}(\pi_{U_\lambda}(\rho_n)\pi_{U_\lambda}(X))}{\mathrm{Tr}(\pi_{U_\lambda}(\rho_n))} = \frac{\mathrm{Tr}\big(\pi_\lambda\big(K_1^{-n+1} K_2^{-n+3} \cdots K_n^{n-1}X\big)\big)}{\mathrm{Tr}\big(\pi_\lambda\big(K_1^{-n+1} K_2^{-n+3} \cdots K_n^{n-1}\big)\big)},
\]
and hence one can also investigate smooth quantized characters by means of quantized universal enveloping algebra $U_q\mathfrak{gl}(n)$.

\subsection*{Acknowledgements}
The author thanks Ryosuke Sato for many discussions on his works on the asymptotic representation theory for quantum groups. The author also thanks Reiji Tomatsu and Yuki Arano for explaining many things about quantum groups and tensor categories during the last decade. In~fact, this work benefited from what Reiji and Yuki explained to the author on various occasions. Finally, the author thanks Makoto Yamashita for presenting him with his Japanese book~\cite{Yamashita:Book} some years ago. The book was so helpful to the author in getting a perspective to quantum groups prior to this work. Finally, the author thanks the referees for their comments.

This work was supported by Grant-in-Aid for Scientific Research (B) JP18H01122.

\pdfbookmark[1]{References}{ref}
\LastPageEnding

\end{document}